\documentclass{article}
\usepackage{amsmath}
\usepackage{amssymb}
\usepackage{amsthm}
\usepackage{algorithm}
\usepackage{algorithmic}
\usepackage{array}
\usepackage{ascmac}
\usepackage{bm}
\usepackage{comment}
\usepackage{multirow}
\usepackage{graphicx}
\usepackage{color}
\usepackage{lscape}
\usepackage{hyperref}
\hypersetup{
    colorlinks=true, citecolor=black, linkcolor=black, anchorcolor=blue,urlcolor=blue
    }
\usepackage[top=20truemm,bottom=20truemm,left=16truemm,right=16truemm]{geometry}
\usepackage{subcaption}
\usepackage{here}
\usepackage{booktabs}
\def\@themcountersep{}
\newtheorem{theorem}{Theorem}[section]
\newtheorem{assumption}[theorem]{Assumption}
\newtheorem{corollary}[theorem]{Corollary}
\newtheorem{definition}[theorem]{Definition}
\newtheorem{example}[theorem]{Example}
\newtheorem{lemma}[theorem]{Lemma}
\newtheorem{proposition}[theorem]{Proposition}

\renewcommand{\algorithmicrequire}{\textbf{Input:}}
\renewcommand{\algorithmicensure}{\textbf{Output:}}
\newcommand{\argmax}{\mathop{\rm arg~max}\limits}
\newcommand{\argmin}{\mathop{\rm arg~min}\limits}

\usepackage{threeparttable}
\usepackage{url}

\begin{document}

\title{Post-Processing with Projection and Rescaling Algorithms for Semidefinite Programming}

\author{Shin-ichi Kanoh\thanks{
Graduate School of Systems and Information Engineering, University of Tsukuba, Tsukuba, Ibaraki 305-8573, and Japan Society for the Promotion of Science, 5-3-1 Kojimachi, Chiyoda-ku, Tokyo 102-0083, Japan. email: s2130104@s.tsukuba.ac.jp
}
and 
Akiko Yoshise\thanks{Corresponding author. Faculty of Engineering, Information and Systems, University of Tsukuba, Tsukuba, Ibaraki 305-8573, Japan. email: yoshise@sk.tsukuba.ac.jp
}     
}


\maketitle     
\begin{abstract}
We propose the algorithm that solves the symmetric cone programs (SCPs) by iteratively calling the projection and rescaling methods the algorithms for solving exceptional cases of SCP. 
Although our algorithm can solve SCPs by itself, we propose it intending to use it as a post-processing step for interior point methods since it can solve the problems more efficiently by using an approximate optimal (interior feasible) solution. 
We also conduct numerical experiments to see the numerical performance of the proposed algorithm when used as a post-processing step of the solvers implementing interior point methods, using several instances where the symmetric cone is given by a direct product of positive semidefinite cones. 
Numerical results show that our algorithm can obtain approximate optimal solutions more accurately than the solvers. 
When at least one of the primal and dual problems did not have an interior feasible solution, the performance of our algorithm was slightly reduced in terms of optimality. 
However, our algorithm stably returned more accurate solutions than the solvers when the primal and dual problems had interior feasible solutions. 
\end{abstract}


\section{Introduction}
\label{sec1}
Let $\mathbb{E}$ be a real-valued vector space with an inner product $\langle \cdot, \cdot \rangle$.
Consider the following symmetric cone programs (SCPs): 
\begin{equation}
\notag
\begin{array}{ccccll}
{\rm (P)} & \displaystyle \inf_x &\langle c,x  \rangle & {\rm s.t.} & \mathcal{A}x = b, &x \in \mathcal{K}, \\
{\rm (D)} & \displaystyle \sup_{(z,y)} &b^\top y & {\rm s.t.} & z = c - \mathcal{A}^*y, &(z,y) \in \mathcal{K}^* \times \mathbb{R}^m,
\end{array}
\end{equation}
where $\mathcal{K} \subseteq \mathbb{E}$ is a symmetric cone, $\mathcal{K}^*$ is a dual cone of $\mathcal{K}$, i.e., $\mathcal{K}^* = \{s \in \mathbb{E} : \langle s,x \rangle \geq 0, \forall x \in \mathcal{K}\}$, $\mathcal{A} : \mathbb{E} \rightarrow \mathbb{R}^m$ is a linear operator, $b \in \mathbb{R}^m$, $c \in \mathbb{E}$ and $\mathcal{A}^* : \mathbb{R}^m \rightarrow \mathbb{E}$ is the adjoint operator of $\mathcal{A}$, i.e., $\langle \mathcal{A}x , y \rangle = \left \langle x , \mathcal{A}^* y \right\rangle$ for all $x \in \mathbb{E}$ and $y \in \mathbb{R}^m$.
Note that $\mathcal{K}^*=\mathcal{K}$ holds since $\mathcal{K}$ is a symmetric cone.
If we choose a positive semidefinite cone as a conic constraint, SCP results in semidefinite programming (SDP).
In this study, we are mainly interested in the case where $\mathcal{K}$ is a Cartesian product of $p$ simple positive semidefinite cones $\mathbb{S}^{r_1}_+, \dots, \mathbb{S}^{r_p}_+$, i.e., $\mathcal{K} = \mathbb{S}^{r_1}_+ \times \dots \times \mathbb{S}^{r_p}_+$.

Interior point methods are among the most popular and practical algorithms to solve SCPs.
In addition to its practical performance, interior point methods have the theoretical strength of a polynomial-time algorithm and have been implemented in many solvers.
Recently, projection and rescaling methods were proposed as new polynomial-time algorithms for the special case of SCP by Louren\c{c}o et al. \cite{Lourenco2019}, and Kanoh and Yoshise \cite{Kanoh2023}.
Their studies were motivated by a work by Chubanov \cite{Chubanov2015}.
A similar type of algorithm, although not a polynomial-time algorithm, was proposed by Pe\~{n}a and Soheili \cite{Pena2017}.
The numerical result of \cite{Kanoh2023} showed that projection and rescaling algorithms solved ill-conditioned instances, i.e., instances with feasible solutions only near the boundaries of the cone, more stably than the commercial solver Mosek.
It is well known that in iterations of interior point methods, the closer the current solution moves to the optimal solution, the more difficult it becomes to compute the search direction accurately.
In other words, interior point methods cannot work stably near the boundaries of the cone. Therefore, the numerical results of \cite{Kanoh2023} motivated us to use projection and rescaling algorithms to solve general SCPs more accurately.

In this study, we propose the algorithms to find approximate optimal solutions to (P) and (D) using projection and rescaling algorithms as an inner procedure.
Although the proposed methods can find approximate optimal solutions to (P) and (D) by themselves, they can work more efficiently by using an approximate optimal (interior feasible) solution.
To take advantage of this feature, the proposed algorithm is designed to be used as a post-processing step for interior point methods.
We also conduct numerical experiments to compare the performance of our algorithm with the solvers. 

The main contribution of this study is to provide comprehensive numerical results showing whether the projection and rescaling algorithms can be used to obtain accurate optimal solutions, which will lead to the development of practical aspects of the projection and rescaling methods. 
We now mention two studies \cite{SDPAGMP, Henrion2016} focusing on accurately solving SDP. 
SDPA-GMP \cite{SDPAGMP} is a very accurate SDP solver that executes the primal-dual interior point method using multiple-precision arithmetic libraries.
Using multiple-precision arithmetic enhances the accuracy of the solution but, at the same time, increases the computational time.
In \cite{Henrion2016}, Henrion, Naldi and Din proposed an algorithm that solves Linear matrix inequality (LMI) problems in exact arithmetic, which means that their algorithm can be used to check the feasibility of dual SDPs.
Although their algorithm can accurately solve small LMI problems, it is unsuitable for moderate or large problems due to exact computations. 
This study differs from these previous studies in that it focuses on the projection and rescaling methods.
Furthermore, this study differs from these studies in that it proposes an algorithm that can be used in a post-processing step. 
Our algorithm is available on the website: \url{https://github.com/Shinichi-K-4649/Post-processing-using-projection-and-rescaling-methods.git}

We developed our algorithm for SDP.
The theoretical foundations of our algorithm discussed in Section 3 hold for SCP. 

\subsection{Motivation}
\label{sec1-1}
The primary motivation for focusing on solving SDP accurately is to address the practical barriers that prevent the implementation of the Facial Reduction Algorithm (FRA) \cite{Borwein1981-1, Borwein1981-2, Pataki2013, Waki2013}. 
The FRA is a regularization technique to address SDP that does not satisfy Slater's condition and was developed originally by Browein and Wolkowicz \cite{Borwein1981-1, Borwein1981-2}. 

We say that (P) satisfies Slater's condition or (P) is strongly feasible if there exists a vector $x \in {\rm int} \mathcal{K}$ satisfying $\mathcal{A}x = b$, where ${\rm int} \mathcal{K}$ is an interior of $\mathcal{K}$. 
Similarly, Slater's condition holds for (D) if there exists $z \in {\rm int} \mathcal{K}^*$ satisfying $c - \mathcal{A}^* y = z$. 
If (P) and/or (D) do not have the interior feasible solutions, optimal solutions might not exist, and a duality gap might exist \cite{Todd2001, Luo1997, Tunccel2012}, which prevents interior point methods from working stably. 

The FRA transforms SDP with no interior feasible solution into SDP satisfying Slater's condition or detects its infeasibility, without changing the optimal value by iteratively projecting the cone into a specific subspace. 
Let us briefly see how the FRA works for (P) here. 
If (P) has no interior feasible solution, then there exists a vector $f \in \mathbb{R}^m$ such that $b^\top f \geq 0$ and $-\mathcal{A}^* f \in \mathcal{K}^* \setminus \{0\}$, see Proposition \ref{pro: not sf iff} in Section \ref{sec2-1}.
If $b^\top f = 0$ holds, $f$ is called a reducing direction for (P).
Otherwise, $f$ is an infeasibility certificate for (P).
The FRA first finds such a vector.
If a reducing direction is obtained, then the FRA generates a new problem with a feasible region equivalent to the feasible region of (P) by replacing the conic constraint $\mathcal{K}$ with $\mathcal{K} \cap \{-\mathcal{A}^* f \}^\bot$, where $\{-\mathcal{A}^* f \}^\bot$ is the space of elements orthogonal to $-\mathcal{A}^* f$.
If a reducing direction exists for a new problem, the FRA tries to find it and generates a new reduced problem again.
The FRA repeats the above process until an infeasibility certificate for a reduced problem or a strongly feasible problem whose feasible region is equivalent to the feasible region of (P). 
Thus, one iteration of the FRA is equivalent to solving a certain SDP to find a reducing direction or an infeasibility certificate. 
For example, implementation at the initial iteration of the FRA for (P) corresponds to solving the following auxiliary system. 
\begin{equation}
\begin{array}{cccclll}
\mbox{(AUX-P)} & {\rm find} \ \ f \in \mathbb{R}^m &{\rm s.t.}  & b^\top f \geq 0, &-\mathcal{A}^*f \in \mathcal{K}^* \setminus \{ 0\}.
\end{array}
\notag
\end{equation}
Therefore, what is required to implement the FRA is to solve the auxiliary systems exactly.
However, this requirement is inaccessible for the current solvers because the auxiliary systems rarely satisfy Slater's condition.
Despite such a barrier to implementing the FRA, several studies have conducted the facial reduction scheme for some problems that arise in applications \cite{Zhao1998, Wolkowicz1999, Krislock2010, Waki2010}.  
These studies directly obtain reducing directions using the structure of the problem without solving the auxiliary systems numerically. 

While these approaches, recently, an implementation of the FRA for any SDP has been studied \cite{Permenter2018, Zhu2019, Permenter2017, Lourenco2021}.
Permenter and Parrilo \cite{Permenter2018}, and Zhu, Pataki and Tran-Dinh \cite{Zhu2019} proposed a practical facial reduction scheme. 
The common idea of these two studies is that instead of solving an auxiliary system, we solve another problem whose feasible region is contained in the feasible region of the original auxiliary system. 
Although their algorithms might not recover the interior feasible solutions or detect infeasibility, we can easily implement them. 
The algorithms of \cite{Permenter2018} and \cite{Zhu2019} can be implemented using only LP solvers and eigenvalue decomposition, respectively. 

Permenter, Friberg and Andersen \cite{Permenter2017}, and Louren\c{c}o, Muramatsu and Tsuchiya \cite{Lourenco2021} proposed an algorithm solving for arbitrary SDP based on a facial reduction scheme. 
Permenter, Friberg and Andersen \cite{Permenter2017} showed that reducing directions for (P) and (D) can be obtained from relative interior feasible solutions to a self-dual homogenous model of (P) and (D). 
In addition, based on this result, they proposed a theoretical algorithm for solving arbitrary SDP.
We remark that their algorithm finds reducing directions only when needed, unlike the conventional FRA. 
Even if (P) and/or (D) do not satisfy Slater's condition, their algorithm directly obtains the optimal solution as long as a complementary solution for (P) and (D) exists. 
However, their algorithm requires an oracle that returns relative interior feasible solutions to the self-dual homogenous model of (P) and (D). 
Although such solutions are theoretically obtained using an interior point method tracking a central path, the numerical experiments of \cite{Permenter2017} showed practical barriers.
On the other hand, the algorithm of \cite{Lourenco2021} requires only an oracle that returns the optimal solution to (P) and (D) satisfying Slater's condition, which is a milder requirement than \cite{Permenter2017}.
In \cite{Lourenco2021}, the authors showed that solving the auxiliary systems can be replaced by solving a certain SDP with primal and dual interior feasible points.
This study shows that the first step toward fully implementing the FRA is to solve SDP with primal and dual interior feasible points as accurately as possible. 

One might think this goal is already achievable with the current SDP solvers.
However, the accuracy of the solution that satisfies the termination tolerances specified by the solvers will not be sufficient for a stable execution of the FRA.
Furthermore, setting the value of the termination tolerances too small can lead to numerical instability of interior point methods.
Thus, it is worth studying how to solve SDP more accurately than ever before. 

To see practical barriers to implementing the FRA, let us consider a simple SDP:
\begin{example}
\begin{equation}
\begin{array}{llc}
{\rm (Ex1.1)} \ \ 
\displaystyle \inf_{x} \ \left \langle
\begin{pmatrix}
1 & 0 & 0 \\
0 & 1 & 0 \\
0 & 0 & 1 
\end{pmatrix}, x  \right \rangle
&{\rm s.t}
&
\left \langle \begin{pmatrix}
1 & 1 & 0 \\
1 & 0 & 0 \\
0 & 0 & 0 
\end{pmatrix}, x \right \rangle  = 1,	\\
&
&
\left \langle \begin{pmatrix}
0 & 0 & 0 \\
0 & 1 & 0 \\
0 & 0 & 0 
\end{pmatrix}, x \right \rangle  = 0,	\\
&
&
\left \langle \begin{pmatrix}
0 & 0 & 1 \\
0 & 0 & 0 \\
1 & 0 & 2 
\end{pmatrix}, x \right \rangle  = 0, \\
&
&
x \in \mathbb{S}^3_+.
\end{array}
\notag
\end{equation}
\end{example}
The optimal value of (Ex1.1) is $1$.
This problem is feasible but not strongly feasible.
Indeed, any feasible solution $x \in \mathbb{S}^3_+$ to (Ex1.1) can be represented as 
\begin{equation}
\notag
x = \begin{pmatrix}
1 & 0 & -t \\
0 & 0 & 0 \\
-t & 0 & t 
\end{pmatrix},
\end{equation}
where $t$ is a real value such that $0 \leq t \leq 1$. 
Thus, using a reducing direction, we can reformulate (Ex1.1) over a lower dimensional positive semidefinite cone.
Any reducing direction $f \in \mathbb{R}^3$ for (Ex1.1) can be represented as $(0, -k, 0)^\top$, where $k>0$.
Indeed, the following holds for such a vector $f$.
\begin{equation}
\notag
b^\top f = 1 \times 0 + 0 \times (-k) + 0 \times 0 = 0, \ {\rm and}
-\mathcal{A}^* f =
\begin{pmatrix}
0 & 0 & 0 \\
0 & k & 0 \\
0 & 0 & 0 
\end{pmatrix}
\in \mathbb{S}^3_+.
\end{equation}
Since $-\mathcal{A}^*f \in \mathbb{S}^3_+$, the region $\mathbb{S}^3_+ \cap \{-\mathcal{A}^* f \}^\bot$ can be represented as $\mathbb{S}^3_+ \cap \{-\mathcal{A}^* f \}^\bot = U \mathbb{S}^{3-r}_+ U^\top$, where $r$ is a rank of $-\mathcal{A}^*f$ and $U \in \mathbb{R}^{3 \times 3-r}$ is a matrix whose columns are the eigenvectors corresponding to the zero eigenvalues of $-\mathcal{A}^*f$.
Suppose that $U$ is given by
\begin{equation}
U = 
\begin{pmatrix}
1 &0 \\
0 &0 \\
0 &1
\end{pmatrix}
.
\notag
\end{equation}
Then, we can generate a reduced problem with interior feasible points as follows:
\begin{equation}
\begin{array}{llc}
\displaystyle \inf_{\bar{x}} \ \left \langle
\begin{pmatrix}
1 & 0 \\
0 & 1 
\end{pmatrix}, \bar{x}  \right \rangle
&{\rm s.t}
&
\left \langle
\begin{pmatrix}
1 & 0 \\
0 & 0 
\end{pmatrix}, \bar{x} \right \rangle  = 1,	\\
&
&
\left \langle \begin{pmatrix}
0 & 0 \\
0 & 0 
\end{pmatrix} , \bar{x} \right \rangle  = 0,	\\
&
&
\left \langle  \begin{pmatrix}
0 & 1 \\
1 & 2 
\end{pmatrix} , \bar{x} \right \rangle  = 0, \\
&
&
\bar{x} \in \mathbb{S}^2_+.
\end{array}
\notag
\end{equation}

Let us apply the FRA to this problem using the commercial solver Mosek.
First, we computed the reducing direction with Mosek by using the formulation proposed in \cite{Lourenco2021} (see Section \ref{sec: check status pre}). 
Then, the following vectors were obtained.  
\begin{equation}
\notag
f_{\rm Mosek} =
\begin{pmatrix}
\mbox{3.46e-09} \\
\mbox{-4.00e-00} \\
\mbox{-1.50e-08}
\end{pmatrix}
\ {\rm and} \
-\mathcal{A}^* f_{\rm Mosek} =
\begin{pmatrix}
\mbox{-3.46e-09} & \mbox{-3.46e-09} & \mbox{1.50e-08} \\
\mbox{-3.46e-09} & \mbox{4.00e-00} & 0 \\
\mbox{1.50e-08} & 0 & \mbox{3.00e-08}
\end{pmatrix}.
\end{equation}
The eigenvalue decomposition of $-\mathcal{A}^* f_{\rm Mosek}$ was computed as follows:
\begin{equation}
\notag
-\mathcal{A}^* f_{\rm Mosek} =
P_{\rm Mosek}
\begin{pmatrix}
\mbox{-9.20e-09}	&0	&0 \\
0	&\mbox{3.57e-08}	&0 \\
0	&0	&\mbox{4.00e-00}
\end{pmatrix}
P_{\rm Mosek}^\top,
\end{equation}
where $P_{\rm Mosek} \in \mathbb{R}^{3 \times 3}$ is a matrix whose columns are the eigenvectors corresponding to the eigenvalues of $-\mathcal{A}^*f_{\rm Mosek}$.
Next, we generated a reduced problem.
To identify $\mathbb{S}^3_+ \cap \{-\mathcal{A}^* f_{\rm Mosek} \}^\bot$, for each eigenvalue of $-\mathcal{A}^* f_{\rm Mosek}$, we must correctly determine whether it is zero or not. 
In this example, we estimated eigenvalues of $-\mathcal{A}^* f_{\rm Mosek}$ are zero if they are less than a threshold $T_\lambda$.
Table \ref{Table: Ex-Status} shows that the computed optimal value of the reduced problem is very sensitive to the value of $T_\lambda$.
If the threshold $T_\lambda$ is set to 1e-7, the rank of $-\mathcal{A}^* f_{\rm Mosek}$ is 1, yielding a reduced problem with an optimal value very close to 1. 
On the other hand, if the threshold $T_\lambda$ is set to 1e-8, the rank of $-\mathcal{A}^* f_{\rm Mosek}$ is 2. 
Then, solving this reduced problem with Mosek, we obtained the infeasibility certificate.

\begin{table}
\caption{Comparison of the reduced problems}
\label{Table: Ex-Status}
\begin{center}
\begin{tabular}{ccccc} \toprule
Thresholds	&$\mathbb{S}^3_+ \cap \{-\mathcal{A}^* f_{\rm Mosek} \}^\bot$	& Computed optimal value of the reduced problem	\\	\midrule
$T_\lambda =$ 1e-7      &$U \mathbb{S}^2_+ U^\top$   &1	 \\
$T_\lambda =$ 1e-8      &$U \mathbb{S}^1_+ U^\top$   &infeasible    \\
\bottomrule
\end{tabular}
\end{center}
\end{table}

Let us now compute the reducing direction using our algorithm proposed in Section \ref{sec: numerical results} and the approximate optimal solution returned by Mosek. 
The computed reducing direction was as follows. 
\begin{equation}
\notag
f_{\rm Pro} =
\begin{pmatrix}
\mbox{7.34e-14} \\
\mbox{-4.00e-00} \\
\mbox{-4.90e-14}
\end{pmatrix}
\ {\rm and} \ 
-\mathcal{A}^* f_{\rm Pro} =
\begin{pmatrix}
\mbox{-7.34e-14} & \mbox{-7.34e-14} & \mbox{4.90e-14} \\
\mbox{-7.34e-14} & \mbox{4.00e-00} & 0 \\
\mbox{4.90e-14} & 0 & \mbox{9.80e-14}
\end{pmatrix}.
\end{equation}
The eigenvalue decomposition of $-\mathcal{A}^* f_{\rm Pro}$ was computed as follows:
\begin{equation}
\notag
-\mathcal{A}^* f_{\rm Pro} =
P_{\rm Pro}
\begin{pmatrix}
\mbox{-8.67e-14}	&0	&0 \\
0	&\mbox{1.11e-13}	&0 \\
0	&0	&\mbox{4.00e-00}
\end{pmatrix}
P_{\rm Pro}^\top,
\end{equation}
where $P_{\rm Pro} \in \mathbb{R}^{3 \times 3}$ is a matrix whose columns are the eigenvectors corresponding to the eigenvalues of $-\mathcal{A}^*f_{\rm Pro}$.
The proposed method found a more accurate solution to the formulation proposed in \cite{Lourenco2021}, resulting in a higher accuracy of the computed reducing direction.
The behavior of the FRA is more stable if the reduced problem is generated using the vector $f_{\rm Pro}$ rather than the vector $f_{\rm Mosek}$.
From this example, we can see that a more accurate method of solving SDP is very effective in obtaining accurate reducing directions and that accurate reducing directions are essential for a stable execution of the FRA. 
\subsection{Structure of this paper}
\label{sec1-2}
The remainder of this paper is organized as follows. 
Section \ref{sec2} is devoted to some preliminaries and notations. 
We review the feasibility statuses of SCP and the projection and rescaling algorithm proposed by Kanoh and Yoshise \cite{Kanoh2023}.
Section \ref{sec3} presents the theoretical foundations of our algorithm and explains our algorithm. 
We also describe a practical version of our algorithm that employs some implementation strategies. 
Section \ref{sec: numerical results} presents numerical results showing that our algorithm stably obtains optimal solutions with higher accuracy than the solvers. 
The conclusions are summarized in Section \ref{sec: pp conclusion}.

\section{Preliminaries and Notations}
\label{sec2}
This section describes the preliminaries and notations used throughout this paper. 
Section \ref{sec2-1} provides an overview of SCP, focusing on the feasibility statuses. 
Section \ref{sec2-2} briefly introduces Euclidean Jordan algebras. 
Although our algorithm uploaded on the website is an algorithm for solving SDP, its theoretical foundations can be extended to SCP without problems.
The algorithm for SCP can be concisely described using Euclidean Jordan algebras.
Thus, we summarize the minimum knowledge of Euclidean Jordan algebras required to understand this study. 
Section \ref{sec2-3} briefly describes the projection and rescaling algorithm proposed by \cite{Kanoh2023}, which is used in our algorithm. 
Our algorithm employed the algorithm of \cite{Kanoh2023} because the numerical results in \cite{Kanoh2023} showed that their algorithm was superior to the other methods in terms of computation time.

\subsection{Feasibility classes of SCP}
\label{sec2-1}
Let $\theta_p \in \mathbb{R} \cup \{ \pm \infty \}$ and $\theta_d \in \mathbb{R} \cup \{ \pm \infty \}$ be the primal and dual optimal values, respectively. 
If (P) is infeasible, $\theta_p$ takes $+\infty$ and if (D) is infeasible, $\theta_d$ takes $-\infty$. 
By $C^\bot \subseteq \mathbb{E}$, we denote the space of elements orthogonal to a set $C \subseteq \mathbb{E}$. 
Here, we review four different mutually exclusive feasibility classes of SCP. 
These feasibility classes are defined in the field of conic linear programming.
Therefore, in this section, we purposely leave $\mathcal{K}^*$ as it is.
First, let us look at the definition of strong feasibility.
\begin{definition}[Strongly feasible]
We say that (P) is strongly feasible (or satisfies Slater's condition) if there exists $x \in \mbox{int } \mathcal{K}$ satisfying $\mathcal{A}x=b$. 
Similarly, we say that (D) is strongly feasible (or satisfies Slater's condition) if there exists $y \in \mathbb{R}^m$ satisfying $c - \mathcal{A}^*y \in \mbox{int } \mathcal{K}^*$. 
\end{definition}

If primal and dual problems are strongly feasible, $\theta_p=\theta_d$ and the existence of optimal solutions to both problems are guaranteed.
\begin{proposition}
\label{pro: strong duality}
\begin{enumerate}
\item If (P) is strongly feasible, then $\theta_p=\theta_d$. In addition, if (D) is feasible, then (D) has an optimal solution.
\item If (D) is strongly feasible, then $\theta_p=\theta_d$. In addition, if (P) is feasible, then (P) has an optimal solution.
\end{enumerate}
\end{proposition}
\begin{proof}
See Theorems 3.2.6 and 3.2.8 in~\cite{Renegar2001}.
\end{proof}

We next see the characterization of strong infeasibility.
\begin{definition}[Strongly infeasible]
We say that (P) is strongly infeasible if there exists $y \in \mathbb{R}^m$ such that $-\mathcal{A}^*y \in \mathcal{K}^*$ and $b^\top y > 0$. 
Similarly, we say that (D) is strongly infeasible if there exists $x \in \mathcal{K}$ such that $\mathcal{A}x=0$ and $\langle c,x \rangle <0$. 
\end{definition}

A vector $y \in \mathbb{R}^m$ satisfying $-\mathcal{A}^*y \in \mathcal{K}^*$ and $b^\top y > 0$ is called an improving ray of (D) because $y$ makes (D) unbounded, i.e., $\theta_d = +\infty$, if (D) is feasible.
We also call $x \in \mathcal{K}$ an improving ray of (P) if $\mathcal{A}x=0$ and $\langle c,x \rangle <0$. 
If there exists an improving ray of (P) (or (D)), then we can see that the hyperplane $\{ \mathcal{A}^*y \}^\bot$ (or $\{ x \}^\bot$) strictly separates the affine space of (P) (or (D)) and $\mathcal{K}$ (or $\mathcal{K}^*$), which implies the infeasibility of (P) (or (D)).

To define the remaining feasibility class, we introduce the following proposition.  
\begin{proposition}
\label{pro: not sf iff}
\begin{enumerate}
\item (P) is not strongly feasible if and only if there exists $y \in \mathbb{R}^m$ such that $-\mathcal{A}^* y \in \mathcal{K}^*$, $\mathcal{A}^* y  \neq 0$ and $b^\top y \geq 0$. 
\item (D) is not strongly feasible if and only if there exists a nonzero $x \in \mathcal{K}$ such that $\mathcal{A}x=0$ and $\langle c,x \rangle \leq 0$. 
\end{enumerate}
\end{proposition}
\begin{proof}
See Lemma 3.2 in~\cite{Waki2013} or Theorems 3.1.3 and 3.1.4 in~\cite{Drusvyatskiy2017}.
\end{proof}

Here, we define the following feasibility class.
\begin{definition}[Weak status]
We say that (P) is in weak status if there exists $y \in \mathbb{R}^m$ such that $-\mathcal{A}^*y \in \mathcal{K}^*$, $\mathcal{A}^* y  \neq 0$ and $b^\top y = 0$. 
Similarly, we say that (D) is in weak status if there exists a nonzero $x \in \mathcal{K}$ such that $\mathcal{A}x=0$ and $\langle c,x \rangle = 0$. 
\end{definition}
We sometimes divide this class into two classes and call them weakly feasible and weakly infeasible, respectively. 
It is understood that a problem is weakly feasible if it is feasible and does not have interior feasible solutions. 
Similarly, it is understood that a problem is weakly infeasible if it is infeasible and its dual does not have an improving ray. 
See~\cite{Luo1997} for more details.

A vector $y \in \mathbb{R}^m$ (or a vector $-\mathcal{A}^*y \in \mathbb{E}$) satisfying $-\mathcal{A}^*y \in \mathcal{K}^*$, $\mathcal{A}^* y  \neq 0$ and $b^\top y = 0$ is called a reducing direction for (P).
We also call a nonzero $x \in \mathcal{K}$ a reducing direction for (D) if $\mathcal{A}x=0$ and $\langle c,x \rangle = 0$. 
As discussed in Section \ref{sec1-1}, reducing directions are used to regularize problems in the FRA.

\subsection{Euclidean Jordan algebras and Symmetric cones}
\label{sec2-2}

Let $\mathbb{E}$ be a real-valued vector space equipped with an inner product $\langle \cdot, \cdot \rangle$ and a bilinear operation $\circ$ : $\mathbb{E} \times \mathbb{E} \rightarrow \mathbb{E}$,
and $e$ be the identity element, i.e.,$x \circ e = e \circ x = x$ holds for any $ x \in \mathbb{E}$.
$(\mathbb{E}, \circ)$ is called a Euclidean Jordan algebra if it satisfies
\begin{align*}
x \circ y = y  \circ x,  \ \
x \circ (x^2 \circ y) = x^2 \circ (x \circ y), \ \
\langle x \circ y , z \rangle = \langle y , x \circ z \rangle
\end{align*}
for all $x,y,z \in \mathbb{E}$ and $x^2 := x \circ x$.
We denote $y \in \mathbb{E}$ as $x^{-1}$ if  $y$ satisfies $x \circ y = e$.
$c \in \mathbb{E}$ is called an {\em idempotent} if it satisfies $c \circ c = c$, and an idempotent $c$ is called {\em primitive} if it can not be written as a sum of two or more nonzero idempotents. 
A set of primitive idempotents $c_1, c_2, \ldots c_k$ is called a {\em Jordan frame} if $c_1, \ldots c_k$ satisfy
\begin{equation}
\notag
c_i \circ c_j = 0 \ (i \neq j), \ \
c_i \circ c_i = c_i \ (i= 1 , \dots , k), \ \
\sum_{i=1}^k c_i = e.
\end{equation}
For $x \in  \mathbb{E}$, the {\em degree} of $x$ is the smallest integer $d$ such that the set $\{e,x,x^2,\ldots,x^d\}$ is linearly independent.
The {\em rank} of $\mathbb{E}$ is the maximum integer $r$ of the degree of $x$ over all $x \in \mathbb{E}$.
The following properties are known.

\begin{proposition}[Spectral theorem (cf. Theorem III.1.2 of \cite{Faraut1994})]
\label{prop:spectral}
Let $(\mathbb{E}, \circ)$ be a Euclidean Jordan algebra having rank $r$. 
For any $x \in \mathbb{E} $, there exist real numbers $\lambda_1 , \dots , \lambda_r$ and a Jordan frame $c_1 , \dots , c_r $ for which the following holds:
\begin{equation}
x = \sum_{i=1}^r \lambda_i c_i \notag.
\end{equation}
The numbers $\lambda_1 , \dots , \lambda_r$ are uniquely determined {\em eigenvalues} of $x$ (with their multiplicities).
Furthermore, ${\rm trace}(x) := \sum_{i=1}^r \lambda_i$, $\det(x) := \prod_{i=1}^r \lambda_i$.
\end{proposition}
For any $x,y \in \mathbb{E}$, we define the inner product $\langle \cdot,\cdot \rangle$ and the norm $\| \cdot \|$ as $\langle x , y \rangle := {\rm trace}(x \circ y)$ and $\|x\| := \sqrt{\langle x,x \rangle}$, respectively.
For any $x \in \mathbb{E}$ having spectral decomposition  $x = \sum_{i=1}^r \lambda_i c_i$ as in  Proposition \ref{prop:spectral}, we also define $\|x\|_\infty := \max \{|\lambda_1| , \dots , |\lambda_r| \}$.

It is known that the set of squares $\mathcal{K} = \{ x^2 : x \in \mathbb{E} \}$ is the symmetric cone of  $\mathbb{E}$  (cf. Theorems III.2.1 and III.3.1 of \cite{Faraut1994}).
The following properties can be derived from the results in \cite{Faraut1994}, as in  Corollary 2.3 of \cite{Yoshise2007}:
\begin{proposition}
\label{prop:lambda-Jordan}
Let $x \in \mathbb{E}$ and let $\sum_{j=1}^r \lambda_j c_j$ be a decomposition of $x$ given by Propositoin \ref{prop:spectral}. Then
\begin{description}
\item[(i)]
$x \in \mathcal{K}$ if and only if  $\lambda_j \geq 0 \ (j=1,2,\ldots,r)$, 
\item[(ii)]
$x \in {\rm int} \hspace{0.75mm} \mathcal{K}$ if and only if $\lambda_j > 0 \ (j=1,2,\ldots,r)$.
\end{description}
\end{proposition}

A Euclidean Jordan algebra $(\mathbb{E}, \circ)$ is called {\em simple} if it cannot be written as any Cartesian product of non-zero Euclidean Jordan algebras.
If the Euclidean Jordan algebra $(\mathbb{E}, \circ)$ associated with a symmetric cone $\mathcal{K}$ is simple, then we say that $\mathcal{K}$ is {\em simple}.
In this study, we will consider that $\mathcal{K}$ is given by a Cartesian product of $p$ simple positive semidefinite cones $\mathbb{S}^{r_\ell}_+$,
$\mathcal{K} := \mathbb{S}^{r_1}_+ \times \dots \times \mathbb{S}^{r_p}_+$, whose rank and identity element are $r_\ell$ and $e_\ell$ $(\ell=1, \ldots, p)$.
The rank $r$ and the identity element of  $\mathcal{K}$ are given by
\begin{equation}
\notag
r = \sum_{\ell=1}^p r_\ell, \ \ e = (e_1 , \dots , e_p).
\end{equation}
Furthermore, for any $x \in \mathbb{E}_1 \times \dots \times \mathbb{E}_p$, $\det(x) = \prod_{\ell=1}^p \prod_{i=1}^{r_\ell} \lambda_i$.

In what follows, $x_\ell$ stands for the $\ell$-th block element of $x \in \mathbb{E}$, i.e., $x = (x_1, \dots, x_p) \in \mathbb{E}_1 \times \dots \times \mathbb{E}_p$.
For each $\ell=1, \cdots, p$, we define $\lambda_{\min}(x_\ell) := \min\{ \lambda_1, \cdots, \lambda_{r_\ell} \}$ and $\lambda_{\max}(x_\ell) := \max\{ \lambda_1, \cdots, \lambda_{r_\ell} \}$ where $\lambda_1, \cdots,  \lambda_{r_\ell}$ are eigenvalues of $x_\ell$.
The minimum and maximum eigenvalues of $x \in \mathcal{K}$ are given by $\lambda_{\min}(x) = \min \{ \lambda_{\min}(x_1), \cdots, \lambda_{\min}(x_p) \}$ and $\lambda_{\max}(x) = \max \{ \lambda_{\max}(x_1), \cdots, \lambda_{\max}(x_p) \}$, respectively.

Next, we consider the {\em quadratic representation}  $Q_v(x)$ defined by $Q_v(x) := 2 v \circ ( v \circ x ) - v^2 \circ x$.
For the Euclidean Jordan algebra $(\mathbb{E} , \circ)$ such as $\mathbb{E} = \mathbb{E}_1 \times \dots \times \mathbb{E}_p$, the quadratic representation $Q_v(x)$ of $x \in \mathbb{E}$ is denoted by $Q_v(x) = \left(Q_{v_1} (x_1) , \dots , Q_{v_p}(x_p) \right)$.
Letting $\mathcal{I}_\ell$ be the identity operator of the Euclidean Jordan algebra $(\mathbb{E}_\ell, \circ_\ell)$ associated with the cone $\mathcal{K}_\ell$, we have $Q_{e_\ell} = \mathcal{I}_\ell$ for $\ell=1, \ldots, p$.
The following properties can also be retrieved from the results in \cite{Faraut1994} as in Proposition 3 of \cite{Lourenco2019}:

\begin{proposition}
\label{ptop:quadratic}
For any  $v \in {\rm int}\mathcal{K}$, $ Q_v (\mathcal{K}) = \mathcal{K}$.
\end{proposition}

More detailed descriptions, including concrete examples of symmetric cone optimization, can be found in, e.g., \cite{Faraut1994,Faybusovich1997,Schmieta2003}. 
Here, we will explain the bilinear operation, the identity element, the inner product, the eigenvalues, the primitive idempotents, and the quadratic representation of the cone when the cone is a positive semidefinite cone.

\begin{example}[$\mathcal{K}$ is the semidefinite cone $\mathbb{S}^n_+$]
{\rm 
Let $\mathbb{S}^n$ be the set of symmetric matrices of $n \times n$.
The semidefinite cone $\mathbb{S}^n_+$ is given by $\mathbb{S}^n_+ = \{ X \in \mathbb{S}^n : X \succeq O \}$.
For any symmetric matrices $X , Y \in \mathbb{S}^n$, define the bilinear operation $\circ$ and inner product as $X \circ Y = \frac{ XY + YX } {2}$ and $\langle X , Y \rangle  = \mbox{tr}(XY) = \sum_{i=1}^n \sum_{j=1}^n X_{ij} Y_{ij}$, respectively.
For any $X \in \mathbb{S}^n$, perform the eigenvalue decomposition and let $u_1 , \dots , u_n$ be the corresponding normalized eigenvectors for the eigenvalues $\lambda_1 , \dots , \lambda_n$: $X = \sum_{i=1}^n \lambda_i u_i u_i^T$.
The eigenvalues of $X$ in the Jordan algebra are $\lambda_1 , \dots , \lambda_n$ and the primitive idempotents are $c_1 = u_1 u_1^T , \dots , c_n = u_n u_n^T$, which implies that the rank of the semidefinite cone $\mathbb{S}^n_+$ is $r=n$.
The identity element is the identity matrix $I$.
The quadratic representation of $V \in \mathbb{S}^n$ is given by $Q_V(X) = V X V$.
}
\end{example}

\subsection{Projection and rescaling algorithm}
\label{sec2-3}
For two sets $C_1$ and $C_2$, we denote by $\mbox{FP} (C_1, C_2)$ the feasibility problem
\begin{equation}
\notag
\begin{array}{ll}
\mbox{find}   &x \in C_1 \cap C_2.
\end{array}
\end{equation}
The projection and rescaling algorithms \cite{Pena2017, Lourenco2019, Kanoh2023} solve $\mbox{FP} (\mathcal{L}, \mbox{int } \mathcal{K})$, where $\mathcal{K} \subseteq \mathbb{E}$ is a symmetric cone and $\mathcal{L} \subseteq \mathbb{E}$ is a linear subspace.
The feasibility of $\mbox{FP} (\mathcal{L}, \mbox{int } \mathcal{K})$ is closely related to the feasibility of $\mbox{FP} (\mathcal{L}^\bot, \mathcal{K} \setminus \{0\})$, where $\mathcal{L}^\bot$ is the orthogonal complement of $\mathcal{L}$.
The proof of Proposition \ref{Proposition: alternative relation} is straightforward using Theorem 20.2 of \cite{Rockafellar1997} and therefore omitted.
\begin{proposition}
\label{Proposition: alternative relation}
$\mbox{FP} (\mathcal{L}, \mbox{int } \mathcal{K})$ is infeasible if and only if $\mbox{FP} (\mathcal{L}^\bot, \mathcal{K} \setminus \{0\})$ is feasible.
\end{proposition}

For any $k >0$ and feasible solution $x$ of $\mbox{FP} (\mathcal{L}, \mbox{int } \mathcal{K})$, since $kx$ is a feasible solution of $\mbox{FP} (\mathcal{L}, \mbox{int } \mathcal{K})$, it makes sense to consider the positive scaled version of $\mbox{FP} (\mathcal{L}, \mbox{int } \mathcal{K})$. 
In \cite{Kanoh2023}, they consider the following feasibility problem.
We will denote it by $\mbox{FP}_{S_\infty} (\mathcal{L}, \mbox{int } \mathcal{K})$ in this study. 
\begin{equation}
\notag
\begin{array}{ll}
\mbox{find}   &x \in \mathcal{L} \cap \mbox{int } \mathcal{K} \ \ \ \mbox{s.t.} \ \ \|x\|_\infty \leq 1.
\end{array}
\end{equation}
 
\begin{proposition}
$\mbox{FP} (\mathcal{L}, \mbox{int } \mathcal{K})$ is feasible if and only if $\mbox{FP}_{S_\infty} (\mathcal{L}, \mbox{int } \mathcal{K})$ is feasible.
\end{proposition}
\begin{proof}
If $\mbox{FP}_{S_\infty} (\mathcal{L}, \mbox{int } \mathcal{K})$ is feasible, it is clear that $\mbox{FP} (\mathcal{L}, \mbox{int } \mathcal{K})$ is feasible. 

Let $x$ be a feasible solution of $\mbox{FP} (\mathcal{L}, \mbox{int } \mathcal{K})$ and let $\lambda_{\max} (x)$ and $\lambda_{\min} (x)$ be a maximum eigenvalue and a minimum eigenvalue of $x$, respectively. 
Since $x \in \mbox{int } \mathcal{K}$, we can see that $\lambda_{\max} (x) \geq \lambda_{\min} (x)  > 0$ and hence $\frac{1}{\lambda_{\max} (x)} x \in \mbox{int } \mathcal{K}$. 
In addition, $\| \frac{1}{\lambda_{\max} (x)} x \|_\infty = 1$ holds. 
Thus, $\frac{1}{\lambda_{\max} (x)} x$ is a feasible solution of $\mbox{FP}_{S_\infty} (\mathcal{L}, \mbox{int } \mathcal{K})$. 
\end{proof}

The projection and rescaling algorithms consist of two ingredients: the ``main algorithm'' and the ``basic procedure.''
The structure of the method is as follows: In the outer iteration, the main algorithm calls the basic procedure with $\mathcal{L}$ and $\mathcal{K}$.
The basic procedure proposed in \cite{Kanoh2023} generates a sequence in $\mathbb{E}$ using projection to $\mathcal{L}$ and terminates in a finite number of iterations returning one of the following:
(i). a solution of problem $\mbox{FP} (\mathcal{L}, \mbox{int } \mathcal{K})$, 
(ii). a solution of problem $\mbox{FP} (\mathcal{L}^\bot, \mathcal{K} \setminus \{0\})$, or
(iii). a cut of $\mbox{FP}_{S_\infty} (\mathcal{L}, \mbox{int } \mathcal{K})$, i.e., a Jordan frame $\{c_1, c_2, \dots, c_r \}$ such that $\langle c_i, x \rangle \leq \xi$ holds for any feasible solution $x$ of problem $\mbox{FP}_{S_\infty} (\mathcal{L}, \mbox{int } \mathcal{K})$ and for some $i \in \{1, 2, \dots, r\}$, where $r$ is a rank of $\mathcal{K}$ and $\xi$ is a parameter specified by the user such that $0 < \xi < 1$.
If the result (i) or (ii) is returned by the basic procedure, then the feasibility of problem $\mbox{FP} (\mathcal{L}, \mbox{int } \mathcal{K})$ can be determined, and the main procedure stops.
If the result (iii) is returned, then the main procedure scales the problem $\mbox{FP}_{S_\infty} (\mathcal{L}, \mbox{int } \mathcal{K})$ as $\mbox{FP}_{S_\infty} (Q_v (\mathcal{L}), \mbox{int } \mathcal{K})$, where $v = \frac{1}{\sqrt{\xi}}  \sum_{h \in H}  c_h + \sum_{h \notin H} c_h$ and $H = \{ i: \langle c_i, x \rangle \leq \xi \}$.
Then, the main procedure calls the basic procedure with $Q_v (\mathcal{L})$ and $\mathcal{K}$.
Noting that $v \in {\rm int} \mathcal{K}$ and Proposition \ref{ptop:quadratic}, the feasibility of $\mbox{FP} (\mathcal{L}, \mbox{int } \mathcal{K})$ can be checked by solving $\mbox{FP}_{S_\infty} (Q_v (\mathcal{L}), \mbox{int } \mathcal{K})$.
Thus, the projection and rescaling algorithm of \cite{Kanoh2023} checks the feasibility of $\mbox{FP} (\mathcal{L}, \mbox{int } \mathcal{K})$ by repeating the above procedures.

Their algorithm has a feature that the main algorithm works while keeping information about the minimum eigenvalues of any feasible solution of $\mbox{FP}_{S_\infty} (\mathcal{L}, \mbox{int } \mathcal{K})$.
That is, their method can determine whether there exists a feasible solution of $\mbox{FP}_{S_\infty} (\mathcal{L}, \mbox{int } \mathcal{K})$ whose minimum eigenvalue is greater than $\varepsilon$, where $\varepsilon >0$ is a parameter specified by the user.
In \cite{Kanoh2023}, a feasible solution of $\mbox{FP}_{S_\infty} (\mathcal{L}, \mbox{int } \mathcal{K})$ whose minimum eigenvalue is greater than or equal to $\varepsilon$ is called an $\varepsilon$-feasible solution of $\mbox{FP}_{S_\infty} (\mathcal{L}, \mbox{int } \mathcal{K})$.

 \begin{algorithm}[H]
 \caption{The projection and rescaling algorithm of \cite{Kanoh2023}}
 \label{pr method}
 \begin{algorithmic}[1]
 \renewcommand{\algorithmicrequire}{\textbf{Input: }}
 \renewcommand{\algorithmicensure}{\textbf{Output: }}
 \renewcommand{\stop}{\textbf{stop }}
 \renewcommand{\return}{\textbf{return }}

 \STATE \algorithmicrequire $\mathcal{L}$, $\mathcal{K}$, $\varepsilon > 0$ and a constant $\xi$ such that $0 < \xi < 1$.
 \STATE \algorithmicensure A solution to $\mbox{FP} (\mathcal{L}, \mbox{int } \mathcal{K})$ or $\mbox{FP} (\mathcal{L}^\bot, \mathcal{K} \setminus \{0\})$ or a certificate that there is no $\varepsilon$ feasible solution to problem $\mbox{FP}_{S_\infty} (\mathcal{L}, \mbox{int } \mathcal{K})$.

 \STATE initialization: $k \leftarrow 1, \ \mathcal{L}^k \leftarrow \mathcal{L}$
 \STATE Call the basic procedure with $\mathcal{L}^k$, $\mathcal{K}$ and $\xi$.
 \IF {a solution to $\mbox{FP} (\mathcal{L}^k, \mbox{int } \mathcal{K})$ is obtained}
 \STATE Rescale the obtained solution to the solution of $\mbox{FP} (\mathcal{L}, \mbox{int } \mathcal{K})$
 \STATE \return the solution to $\mbox{FP} (\mathcal{L}, \mbox{int } \mathcal{K})$
 \ELSIF {a solution to $\mbox{FP} ({\mathcal{L}^k}^\bot, \mathcal{K} \setminus \{0\})$ is obtained}
 \STATE Rescale the obtained solution to the solution of $\mbox{FP} (\mathcal{L}^\bot, \mathcal{K} \setminus \{0\})$
 \STATE \return the solution to $\mbox{FP} (\mathcal{L}^\bot, \mathcal{K} \setminus \{0\})$
 \ELSE
 \STATE Compute an upper bound for the minimum eigenvalue of any feasible solution of $\mbox{FP}_{S_\infty} (\mathcal{L}, \mbox{int } \mathcal{K})$
 \IF {the computed upper bound is less than $\varepsilon$}
 \STATE \stop Algorithm \ref{pr method} (There is no $\varepsilon$ feasible solution to $\mbox{FP}_{S_\infty} (\mathcal{L}, \mbox{int } \mathcal{K})$.)
 \ENDIF
 \STATE Compute the vector $v$ used to scale the problem
 \ENDIF
 \STATE Scale the linear subspace, i.e., $\mathcal{L}^{k+1} \leftarrow Q_v (\mathcal{L}^k)$
 \STATE $k \leftarrow k + 1$. Go back to line 4.
 \end{algorithmic} 
 \end{algorithm}

\section{Proposed algorithms}
\label{sec3}
Since projection and rescaling algorithms can only solve the special case of SCP, i.e., $\mbox{FP} (\mathcal{L}, \mbox{int } \mathcal{K})$ for a linear subspace $\mathcal{L} \subseteq \mathbb{E}$, we consider how to use projection and rescaling algorithms to obtain an approximate optimal solution to (P) or (D). 
In Section \ref{sec: modeling}, we introduce two types of formulations to which projection and rescaling methods can be applied. 
These formulations require a real value $\theta \in \mathbb{R}$, and their feasibilities depend on the value of $\theta$.  
In Section \ref{sec: modeling1}, we introduce the formulation that gives the interior feasible solution $x$ of (P) such that $\langle c, x \rangle < \theta$ when (P) is strongly feasible and $\theta > \theta_p$.  
In Section \ref{sec: modeling2}, we introduce the formulation that gives the interior feasible solution $(y, z)$ of (D) such that $b^\top y > \theta$ when (D) is strongly feasible and $\theta < \theta_d$. 
Then, we present a basic idea of our algorithm in Section \ref{sec: concept}.
We also discuss some implementation strategies of our algorithm in Section \ref{sec: tech}.
We then develop a practical version of our algorithm in Section \ref{sec: Algorithm}.

\subsection{Theoretical foundations}
\label{sec: modeling}

\subsubsection{Primal model}
\label{sec: modeling1}
Let $\theta \in \mathbb{R}$. 
For $\theta$, let us define the linear operator $\mathcal{A}(\theta)$ as follows: 
\begin{equation}
\notag
\mathcal{A}(\theta) =
\begin{pmatrix}
\mathcal{A}	&-b	&\bm{0}	\\
c^\top 	&-\theta	&1
\end{pmatrix}
.
\end{equation}

We define the symmetric cone $\bar{\mathcal{K}} = \mathcal{K} \times \mathbb{R}^2_+$ and consider the feasibility problem $\mbox{FP} (\mbox{ker} \mathcal{A}(\theta), \mbox{int } \bar{\mathcal{K}})$.

\begin{proposition}
\label{pro: model1-sol}
Suppose that $\mbox{FP} (\mbox{ker} \mathcal{A}(\theta), \mbox{int } \bar{\mathcal{K}})$ is feasible and $(x, \tau, \rho)$ is a feasible solution of it. 
Then, $\frac{1}{\tau} x$ is an interior feasible solution to (P), and $\langle c, \frac{1}{\tau} x \rangle < \theta$ holds. 
\end{proposition}
\begin{proof}
Since $(x, \tau, \rho)$ is a feasible solution of $\mbox{FP} (\mbox{ker} \mathcal{A}(\theta), \mbox{int } \bar{\mathcal{K}})$, $\mathcal{A}x-\tau b = 0$ and $\langle c,x \rangle - \tau \theta + \rho = 0$ hold.
Noting that $\tau>0$, we have $\mathcal{A} \frac{1}{\tau} x - b = 0$ and $\langle c, \frac{1}{\tau} x \rangle - \theta + \frac{\rho}{\tau} = 0$ and hence, $\frac{1}{\tau}x$ is an interior feasible solution of (P) such that $\langle c, \frac{1}{\tau} x \rangle < \theta$ because $x \in \mbox{int } \mathcal{K}$ and $\rho >0$.
\end{proof}

From Proposition \ref{pro: model1-sol}, if we obtain the feasible solution of $\mbox{FP} (\mbox{ker} \mathcal{A}(\theta), \mbox{int } \bar{\mathcal{K}})$, then we can construct the interior feasible solution to (P) whose objective value is smaller than $\theta$.
The following proposition gives us a necessary and sufficient condition for $\mbox{FP} (\mbox{ker} \mathcal{A}(\theta), \mbox{int } \bar{\mathcal{K}})$ to be feasible. 

\begin{proposition}
\label{pro: model1-feas}
$\mbox{FP} (\mbox{ker} \mathcal{A}(\theta), \mbox{int } \bar{\mathcal{K}})$ is feasible if and only if (P) is strongly feasible and $\theta_p < \theta$.
\end{proposition}
\begin{proof}
If $\mbox{FP} (\mbox{ker} \mathcal{A}(\theta), \mbox{int } \bar{\mathcal{K}})$ is feasible, then there exists a feasible solution $(x, \tau, \rho)$ to $\mbox{FP} (\mbox{ker} \mathcal{A}(\theta), \mbox{int } \bar{\mathcal{K}})$.
For $x$ and $\tau$, $\frac{1}{\tau} x$ is an interior feasible solution to (P) satisfying $\langle c, \frac{1}{\tau} x \rangle < \theta$ from Proposition \ref{pro: model1-sol}, which implies that (P) has an interior feasible solution and $\theta_p \leq \langle c, \frac{1}{\tau} x \rangle < \theta$.

Conversely, if (P) is stronlgy feasible and $\theta_p < \theta$, then there exists an interior feasible solution $x$ to (P) such that $\theta_p < \langle c, x \rangle < \theta$.
We can easily see that $(x, 1, \theta - \langle c,x \rangle)$ is a feasible solution for $\mbox{FP} (\mbox{ker} \mathcal{A}(\theta), \mbox{int } \bar{\mathcal{K}})$.
\end{proof}

Combining Proposition \ref{pro: model1-feas} with Proposition \ref{Proposition: alternative relation}, we have a necessary and sufficient condition for alternative problem $\mbox{FP} (\mbox{range} {\mathcal{A}(\theta)}^*, \bar{\mathcal{K}} \setminus \{0\})$ of $\mbox{FP} (\mbox{ker} \mathcal{A}(\theta), \mbox{int } \bar{\mathcal{K}})$ to be feasible. 
\begin{corollary}
\label{pro: model1-alt-feas}
$\mbox{FP} (\mbox{range} {\mathcal{A}(\theta)}^*, \bar{\mathcal{K}} \setminus \{0\})$ is feasible if and only if (P) is not strongly feasible or $\theta_p \geq \theta$.
\end{corollary}

While feasible solutions of $\mbox{FP} (\mbox{ker} \mathcal{A}(\theta), \mbox{int } \bar{\mathcal{K}})$ give us interior feasible solutions to (P) whose objective value is smaller than $\theta$, feasible solutions of $\mbox{FP} (\mbox{range} {\mathcal{A}(\theta)}^*, \bar{\mathcal{K}} \setminus \{0\})$ give us information about the feasibility of (P) or a feasible solution to (D) whose objective value is greater than or equal to $\theta$.
\begin{proposition}
\label{pro: model1-alt-sol}
Suppose that $\mbox{FP} (\mbox{range} {\mathcal{A}(\theta)}^*, \bar{\mathcal{K}} \setminus \{0\})$ is feasible and $(z, \omega, \kappa)$ is a feasible solution of $\mbox{FP} (\mbox{range} {\mathcal{A}(\theta)}^*, \bar{\mathcal{K}} \setminus \{0\})$. 
Then, there exists $y \in \mathbb{R}^m$ such that $z = \mathcal{A}^*y + \kappa c$ and $\omega = -b^\top y - \kappa \theta$.
For $(y, \kappa)$, one of the following three cases holds:
\begin{enumerate}
\item $\kappa >0$ meaning that $(- \frac{1}{\kappa} y, c-\mathcal{A}^*(\frac{-1}{\kappa} y))$ is a feasible solution to (D) and its objective value is greater than or equal to $\theta$, 
\item $\kappa = 0$ and $\omega > 0$ meaning that $-y$ is an improving ray of (D), i.e., $-\mathcal{A}^*(-y) \in \mathcal{K}$ and $b^\top (-y) > 0$, or 
\item $\kappa=\omega=0$ meaning that $-y$ is a reucing direction for (P), i.e., $-\mathcal{A}^*(-y) \in \mathcal{K} \setminus \{0\}$ and $b^\top (-y) = 0$.
\end{enumerate}
\end{proposition}

\begin{proof}
Since $(z, \omega, \kappa) \in \mbox{range} {\mathcal{A}(\theta)}^*$, there exists $(y, \gamma) \in \mathbb{R}^{m+1}$ such that
\begin{equation}
\notag
\begin{pmatrix}
z\\
\omega \\
\kappa
\end{pmatrix}
=
\begin{pmatrix}
\mathcal{A}^* & c \\
-b^\top & -\theta \\
\bm{0}^\top & 1
\end{pmatrix}
\begin{pmatrix}
y \\
\gamma
\end{pmatrix}
.
\end{equation}
From this equation, we can easily see that $\kappa = \gamma$, and hence $z = \mathcal{A}^*y + \kappa c$ and $\omega = -b^\top y - \kappa \theta$ hold for $(y, \kappa)$.
\medskip \\
\noindent (1): If $\kappa >0$, then $-\mathcal{A}^*(\frac{-1}{\kappa} y) + c \in \mathcal{K}$ and $b^\top (\frac{-1}{\kappa} y) - \theta \geq 0$ hold since $(z, \omega, \kappa) \in \bar{\mathcal{K}} \setminus \{0\}$.
\medskip \\
\noindent (2) \& (3): If $\kappa =0$, then $z = -\mathcal{A}^*(-y) \in \mathcal{K}$ and $\omega = b^\top (-y) \geq 0$ hold for $y$ since $(z, \omega, \kappa) \in \bar{\mathcal{K}} \setminus \{0\}$.

If $\omega>0$, $y$ satisfies $-\mathcal{A}^*(-y) \in \mathcal{K}$ and $b^\top (-y) > 0$ and hence, $-y$ is an improving ray of (D).

If $\omega = 0$, we can easily see that $b^\top (-y) = 0$.
In addition, $-\mathcal{A}^*(-y) \in \mathcal{K} \setminus \{0\}$ holds since $(z, \omega, \kappa) \neq (0,0,0)$.
Thus, $-y$ is a reducing direction for (P).
\end{proof}

Proposition \ref{pro: model1-alt-sol} ensures that feasible solutions of $\mbox{FP} (\mbox{range} {\mathcal{A}(\theta)}^*, \bar{\mathcal{K}} \setminus \{0\})$ give feasible solutions to (D) if (P) is strongly feasible and (D) is feasible.
By adding one more assumption, we can guarantee that the feasible solution of $\mbox{FP} (\mbox{range} {\mathcal{A}(\theta)}^*, \bar{\mathcal{K}} \setminus \{0\})$ gives the optimal solution of (D). 
\begin{corollary}
\label{coro: model1-alt-opt-sol}
Suppose that (P) is strongly feasible, (D) is feasible, and $\theta = \theta_p$. 
Then, $\mbox{FP} (\mbox{range} {\mathcal{A}(\theta)}^*, \bar{\mathcal{K}} \setminus \{0\})$ is feasible, i.e., there exists $y \in \mathbb{R}^m$ such that $z = \mathcal{A}^*y + \kappa c$ and $\omega = -b^\top y - \kappa \theta$ for any feasible solution $(z, \omega, \kappa)$ of $\mbox{FP} (\mbox{range} {\mathcal{A}(\theta)}^*, \bar{\mathcal{K}} \setminus \{0\})$. 
Furthermore, $(-\frac{1}{\kappa} y, \frac{1}{\kappa} z)$ is an optimal solution for (D). 
\end{corollary}
\begin{proof}
Since $\theta = \theta_p$, $\mbox{FP} (\mbox{range} {\mathcal{A}(\theta)}^*, \bar{\mathcal{K}} \setminus \{0\})$ is feasible by Corollary \ref{pro: model1-alt-feas}.
In addition, for any feasible solution  $(z, \omega, \kappa)$ for $\mbox{FP} (\mbox{range} {\mathcal{A}(\theta)}^*, \bar{\mathcal{K}} \setminus \{0\})$, there exists $y \in \mathbb{R}^m$ satisfying $z = \mathcal{A}^*y + \kappa c$ and $\omega = -b^\top y - \kappa \theta$ by Proposition \ref{pro: model1-alt-sol}.

Noting that (P) is strongly feasible and (D) is feasible, we can see that $\kappa > 0$ holds for any feasible solution  $(z, \omega, \kappa)$ for $\mbox{FP} (\mbox{range} {\mathcal{A}(\theta)}^*, \bar{\mathcal{K}} \setminus \{0\})$ from Proposition \ref{pro: model1-alt-sol}.
Thus, $(- \frac{1}{\kappa} y, c-\mathcal{A}^*(\frac{-1}{\kappa} y))$ is a feasible solution for (D) and $b^\top (\frac{-1}{\kappa} y) \geq \theta_p$ holds.

Since (P) is strongly feasible and (D) is feasible, by Proposition \ref{pro: strong duality}, (D) has an optimal solution and $\theta_p = \theta_d$ and hence, $(- \frac{1}{\kappa} y, c-\mathcal{A}^*(\frac{-1}{\kappa} y))$ is an optimal solution for (D).
\end{proof}

\subsubsection{Dual model}
\label{sec: modeling2}

Let $\theta \in \mathbb{R}$. 
We define the linear operator $\mathcal{A}(\theta)$ and the symmetric cone $\bar{\mathcal{K}}$ in the same way as defined in Section \ref{sec: modeling1}.
In this section, we consider the feasibility problem $\mbox{FP} (\mbox{range} \mathcal{A}(\theta)^*, \mbox{int } \bar{\mathcal{K}})$.
\begin{proposition}
\label{pro: model2-sol}
Suppose that $\mbox{FP} (\mbox{range} \mathcal{A}(\theta)^*, \mbox{int } \bar{\mathcal{K}})$ is feasible and $(z, \omega, \kappa)$ is a feasible solution of it.
Then, there exists $y \in \mathbb{R}^m$ such that $z = \mathcal{A}^*y + \kappa c$ and $\omega = -b^\top y - \kappa \theta$.
In addition, $(- \frac{1}{\kappa} y, c-\mathcal{A}^*(\frac{-1}{\kappa} y))$ is an interior feasible solution to (D), and $b^\top (\frac{-1}{\kappa} y) > \theta$ holds. 
\end{proposition}
\begin{proof}
Since this proposition can be proved in the same way as the proof of Proposition \ref{pro: model1-sol}, we omit the proof. 
%
\end{proof}

Similar to Proposition \ref{pro: model1-sol}, Proposition \ref{pro: model2-sol} implies that we can construct the interior feasible solution of (D) whose objective value is greater than $\theta$ using the feasible solutions of $\mbox{FP} (\mbox{range} \mathcal{A}(\theta)^*, \mbox{int } \bar{\mathcal{K}})$
The following proposition gives us a necessary and sufficient condition for $\mbox{FP} (\mbox{range} \mathcal{A}(\theta)^*, \mbox{int } \bar{\mathcal{K}})$ to be feasible. 
\begin{proposition}
\label{pro: model2-feas}
$\mbox{FP} (\mbox{range} \mathcal{A}(\theta)^*, \mbox{int } \bar{\mathcal{K}})$ is feasible if and only if (D) is strongly feasible and $\theta_d > \theta$.
\end{proposition}
\begin{proof}
Since this proposition can be proved in the same way as the proof of Proposition \ref{pro: model1-feas}, we omit the proof. 
%
\end{proof}

Combining Proposition \ref{pro: model2-feas} with Proposition \ref{Proposition: alternative relation}, we have a necessary and sufficient condition for alternative problem $\mbox{FP} (\mbox{ker} \mathcal{A}(\theta), \bar{\mathcal{K}} \setminus \{0\})$ of $\mbox{FP} (\mbox{range} \mathcal{A}(\theta)^*, \mbox{int } \bar{\mathcal{K}})$ to be feasible. 
\begin{corollary}
\label{pro: model2-alt-feas}
$\mbox{FP} (\mbox{ker} \mathcal{A}(\theta), \bar{\mathcal{K}} \setminus \{0\})$ is feasible if and only if (D) is not strongly feasible or $\theta_d \leq \theta$.
\end{corollary}

While feasible solutions for $\mbox{FP} (\mbox{range} {\mathcal{A}(\theta)}^*, \mbox{int } \bar{\mathcal{K}})$ give us interior feasible solutions to (D) whose objective value is greater than $\theta$, feasible solutions for $\mbox{FP} (\mbox{ker} \mathcal{A}(\theta), \bar{\mathcal{K}} \setminus \{0\})$ give us information about the feasibility of (D) or a feasible solution to (P) whose objective value is smaller than or equal to $\theta$.
\begin{proposition}
\label{pro: model2-alt-sol}
Suppose that $\mbox{FP} (\mbox{ker} \mathcal{A}(\theta), \bar{\mathcal{K}} \setminus \{0\})$ is feasible and $(x, \tau, \rho)$ is a feasible solution of $\mbox{FP} (\mbox{ker} \mathcal{A}(\theta), \bar{\mathcal{K}} \setminus \{0\})$. 
Then, for $(x, \tau, \rho)$, one of the following three cases holds:
\begin{enumerate}
\item $\tau >0$ meaning that $\frac{1}{\tau} x$ is a feasible solution to (P) and its objective value is smaller than or equal to $\theta$, 
\item $\tau = 0$ and $\rho > 0$, meaning that $x$ is an improving ray of (P), i.e., $x \in \mathcal{K}$, $\mathcal{A} x = 0$ and $\langle c,x \rangle < 0$, or 
\item $\tau=\rho=0$ meaning that $x$ is a reucing direction for (D), i.e., $x \in \mathcal{K} \setminus \{0\}$, $\mathcal{A} x = 0$ and $\langle c,x \rangle = 0$.
\end{enumerate}
\end{proposition}

\begin{proof}
Since this proposition can be proved in the same way as the proof of Proposition \ref{pro: model1-alt-sol}, we omit the proof. 
\end{proof}

Proposition \ref{pro: model2-alt-sol} ensures that feasible solutions for $\mbox{FP} (\mbox{ker} \mathcal{A}(\theta), \bar{\mathcal{K}} \setminus \{0\})$ provide feasible solutions for (P) if (D) is strongly feasible and (P) is feasible.
Similar to Corollary \ref{coro: model1-alt-opt-sol}, by adding one more assumption, we can guarantee that the feasible solution of $\mbox{FP} (\mbox{ker} \mathcal{A}(\theta), \bar{\mathcal{K}} \setminus \{0\})$ gives the optimal solution of (P). 

\begin{corollary}
\label{coro: model2-alt-opt-sol}
Suppose that (D) is strongly feasible, (P) is feasible, and $\theta = \theta_d$. 
Then, $\mbox{FP} (\mbox{ker} \mathcal{A}(\theta), \bar{\mathcal{K}} \setminus \{0\})$ is feasible.
In addition, for any feasible solution $(x, \tau, \rho)$ for $\mbox{FP} (\mbox{ker} \mathcal{A}(\theta), \bar{\mathcal{K}} \setminus \{0\})$, $\frac{1}{\tau} x$ is an optimal solution for (P). 
\end{corollary}
\begin{proof}
Since this proposition can be proved in the same way as the proof of Corollary \ref{coro: model1-alt-opt-sol}, we omit the proof. 
\end{proof}

\subsection{Basic concept of the proposed algorithm }
\label{sec: concept}
To briefly illustrate the concept of the proposed method, suppose that (P) is strongly feasible.
Then, the feasibility of $\mbox{FP} (\mbox{ker} \mathcal{A}(\theta), \mbox{int } \bar{\mathcal{K}})$ only depends on the value of $\theta$ by Proposition \ref{pro: model1-feas}.
If $\theta > \theta_p$, $\mbox{FP} (\mbox{ker} \mathcal{A}(\theta), \mbox{int } \bar{\mathcal{K}})$ is feasible and its solution gives an interior feasible solution $x$ to (P) such that $\langle c, x \rangle < \theta$ by Proposition \ref{pro: model1-sol}.
If $\theta \leq \theta_p$, $\mbox{FP} (\mbox{ker} \mathcal{A}(\theta), \mbox{int } \bar{\mathcal{K}})$ is infeasible and the infeasibility certificates for $\mbox{FP} (\mbox{ker} \mathcal{A}(\theta), \mbox{int } \bar{\mathcal{K}})$, i.e., feasible solutions for $\mbox{FP} (\mbox{range} {\mathcal{A}(\theta)}^*, \bar{\mathcal{K}} \setminus \{0\})$, give the dual feasible solution $y$ such that $b^\top y \geq \theta$ by Proposition \ref{pro: model1-alt-sol}. 
Therefore, the closer the value of $\theta$ is to the optimal value of (P) and (D), the more accurate the approximate optimal solution of (P) and (D) obtained by solving $\mbox{FP} (\mbox{ker} \mathcal{A}(\theta), \mbox{int } \bar{\mathcal{K}})$. 
In addition, we can know whether $\theta > \theta_p$ or $\theta \leq \theta_p$ by solving $\mbox{FP} (\mbox{ker} \mathcal{A}(\theta), \mbox{int } \bar{\mathcal{K}})$. 

We present our algorithm for finding an approximate optimal interior feasible solution to (P). 
 \begin{algorithm}
 \caption{Projection and rescaling algorithm with the primal model}
 \label{p alg}
 \begin{algorithmic}[1]
 \renewcommand{\algorithmicrequire}{\textbf{Input: }}
 \renewcommand{\algorithmicensure}{\textbf{Output: }}
  \renewcommand{\stop}{\textbf{stop }}
 \renewcommand{\return}{\textbf{return }}

 \STATE \algorithmicrequire $\mathcal{A}$, $b$, $c$, $\mathcal{K}$, and $\theta_{acc} > 0$.
 \STATE \algorithmicensure A feasible solution to (P) or a vector that determines the feasibility status of (P).
 \STATE initialization: $k \leftarrow 0, \ LB \leftarrow -\infty, UB \leftarrow \infty$, $\bar{\mathcal{K}} \leftarrow \mathcal{K} \times \mathbb{R}^{2}_+$
 \STATE Choose $\theta^k \in (LB, UB)$ and construct
\begin{equation}
\notag
\mathcal{A}(\theta^k) =
\begin{pmatrix}
\mathcal{A}	&-b	&\bm{0}	\\
c^\top 	&-\theta^k	&1
\end{pmatrix}
.
\end{equation}
 \STATE Let $\varepsilon$ be a sufficiently small positive value and $\xi$ be a constant such that $0<\xi<1$.
 \WHILE {$UB-LB > \theta_{acc}$}
 \STATE Call the projection and rescaling algorithm of \cite{Kanoh2023} with $\mbox{ker} \mathcal{A}(\theta^k)$, $\bar{\mathcal{K}}$, $\varepsilon$, and $\xi$.
 \IF {a solution $(x, \tau, \rho)$ to $\mbox{FP} (\mbox{ker} \mathcal{A}(\theta^k), \mbox{int } \bar{\mathcal{K}})$ is obtained}
 \STATE $x_{tmp} \leftarrow \frac{1}{\tau} x$, $UB \leftarrow \theta^k$
 \ELSIF {a solution $(z, \omega, \kappa)$ to $\mbox{FP} (\mbox{range} {\mathcal{A}(\theta^k)}^*, \bar{\mathcal{K}} \setminus \{0\})$ is obtained}
 \STATE Compute $(y, \gamma) \in \mathbb{R}^{m+1}$ such that $z = \mathcal{A}^*y + \gamma c$, $\omega = -b^\top y - \gamma \theta$ and $\kappa = \gamma$.
 \IF {$-y$ is an improving ray of (D) or a reducing direction for (P)}
 \STATE \stop Algorithm \ref{p alg} and \return $(-y, \mathcal{A}^*y)$
 \ENDIF
 \STATE $y_{tmp} \leftarrow \frac{-1}{\kappa} y$, \ $z_{tmp} \leftarrow c - \mathcal{A}^*y_{tmp}$, \ $LB \leftarrow \theta^k$ 
 \ELSE
 \STATE $LB \leftarrow \theta^k$
 \ENDIF
 \STATE Choose $\theta^{k+1} \in (LB, UB)$, \ $k \leftarrow k+1$
 \ENDWHILE
 \STATE \return $x_{tmp}$
 \end{algorithmic} 
 \end{algorithm}

Algorithm \ref{p alg} works as follows.
First, Algorithm \ref{p alg} chooses the input value $\theta^k \in (LB, UB)$ and execute a projection and rescaling algorithm with the corresponding problem $\mbox{FP} (\mbox{ker} \mathcal{A}(\theta^k), \mbox{int } \bar{\mathcal{K}})$.
Next, Algorithm \ref{p alg} performs the operations according to the returned result, as follows:
\begin{enumerate}
\item If a solution to (P) is obtained, the current primal solution $x_{tmp}$ and $UB$ are updated. 
\item If a solution to (D) is obtained, the current dual solution $y_{tmp}$ and $LB$ are updated. 
\item If a reducing direction for (P) or an improving ray of (D) is obtained, then Algorithm \ref{p alg} terminates. 
\item If a projection and rescaling algorithm determines that the minimum eigenvalue of any feasible solution of the input problem, i.e., $\mbox{FP}_{S_\infty}  (\mbox{ker } \mathcal{A}(\theta^k), \mbox{int } \bar{\mathcal{K}})$, is less than $\varepsilon$, then $LB$ is updated. 
\end{enumerate}
Algorithm \ref{p alg} repeats the above operations until $UB - LB > \theta_{acc}$ holds.
$UB$ and $LB$ play the role of upper and lower bounds on $\theta_p$, respectively.
Thus, if $\theta_{acc}$ is sufficiently small, the output $x_{tmp}$ from Algorithm \ref{p alg} can be seen as an approximate optimal solution to (P). 

Here, we note that Algorithm \ref{p alg} updates $LB$ on line 17. 
Suppose that Algorithm \ref{p alg} reaches line 17 at the $k$-th iteration. 
In this case, it can be inferred that the projection and rescaling algorithm terminates with the same termination criteria with $\mbox{FP}_{S_\infty} (\mbox{ker } \mathcal{A}(\theta), \mbox{int } \bar{\mathcal{K}})$ for any $\theta$ such that $\theta < \theta^k$, unless $\theta^k$ is not too large compared to the maximum value of the objective function in (P). 
Noting that Algorithm \ref{p alg} is used as a post-processing step, it is reasonable to update $LB$ as on line 17.
We also note that Algorithm \ref{p alg} can find an approximate optimal solution to (D).
Suppose that Algorithm \ref{p alg} reaches line 15 at the $k$-th iteration and $\theta^k$ is very close to the optimal value of (D).
Then, $(y_{tmp}, z_{tmp})$ satisfies $b^\top y_{tmp} \geq \theta^k$ and $c-\mathcal{A}^* y_{tmp} = z_{tmp} \in \mathcal{K}$ by Proposition \ref{pro: model1-alt-sol}.
However, even if $\mbox{FP} (\mbox{ker} \mathcal{A}(\theta^k), \mbox{int } \bar{\mathcal{K}})$ is infeasible, the projection and rescaling algorithm does not necessarily return a solution to its alternative problem. 
Thus, Algorithm \ref{p alg} is just a method to find the approximate optimal interior feasible solution of (P).

From the contents of Section \ref{sec: modeling2}, the algorithm for finding an approximate optimal interior feasible solution to (D) can be considered similarly. (See Algorithm \ref{d alg}.)
 \begin{algorithm}
 \caption{Projection and rescaling algorithm with the dual model}
 \label{d alg}
 \begin{algorithmic}[1]
 \renewcommand{\algorithmicrequire}{\textbf{Input: }}
 \renewcommand{\algorithmicensure}{\textbf{Output: }}
  \renewcommand{\stop}{\textbf{stop }}
 \renewcommand{\return}{\textbf{return }}

 \STATE \algorithmicrequire $\mathcal{A}$, $b$, $c$, $\mathcal{K}$, and $\theta_{acc} > 0$.
 \STATE \algorithmicensure A feasible solution to (D) or a vector that determines the feasibility status of (D).
 \STATE Same as lines 3-5 of Algorithm \ref{p alg}
 \WHILE {$UB-LB > \theta_{acc}$}
 \STATE Call the projection and rescaling algorithm of \cite{Kanoh2023} with $\mbox{range} {\mathcal{A}(\theta^k)}^*$, $\bar{\mathcal{K}}$, $\varepsilon$, and $\xi$.
 \IF {a solution $(z, \omega, \kappa)$ to $\mbox{FP} (\mbox{range} {\mathcal{A}(\theta^k)}^*, \mbox{int } \bar{\mathcal{K}})$ is obtained}
 \STATE Compute $(y, \gamma) \in \mathbb{R}^{m+1}$ such that $z = \mathcal{A}^*y + \gamma c$, $\omega = -b^\top y - \gamma \theta$ and $\kappa = \gamma$.
 \STATE $y_{tmp} \leftarrow \frac{-1}{\kappa} y$, \ $z_{tmp} \leftarrow c - \mathcal{A}^*y_{tmp}$, \ $LB \leftarrow \theta^k$
 \ELSIF  {a solution $(x, \tau, \rho)$ to $\mbox{FP} (\mbox{ker } \mathcal{A}(\theta^k), \bar{\mathcal{K}} \setminus \{0\} )$ is obtained}
 \IF {$x$ is an improving ray of (P) or a reducing direction for (D)}
 \STATE \stop Algorithm \ref{d alg} and \return $x$
 \ENDIF
 \STATE $x_{tmp} \leftarrow \frac{1}{\tau} x$, $UB \leftarrow \theta^k$
 \ELSE
 \STATE $UB \leftarrow \theta^k$
 \ENDIF
 \STATE Choose $\theta^{k+1} \in (LB, UB)$, \ $k \leftarrow k+1$
 \ENDWHILE
 \STATE \return $y_{tmp}$
 \end{algorithmic} 
 \end{algorithm}

\subsection{Implementation strategies}
\label{sec: tech}
Algorithms \ref{p alg} and \ref{d alg} can operate more efficiently by fully utilizing the information obtained at each iteration. 
In this section, we modify Algorithm \ref{p alg} to make it more practical.
These modifications can also be employed in Algorithm \ref{d alg}.
%

\subsubsection{Updating LB using a dual feasible solution}
\label{sec: tech1}
Algorithm \ref{p alg} keeps $LB$ as the lower bound on the $\theta_p$. 
If a dual feasible solution $(y_{tmp}, z_{tmp})$ is obtained at the $k$-th iteration of Algorithm \ref{p alg}, we update $LB$ with $\theta^k$. 
Here, we note that $b^\top y_{tmp} \geq \theta^k$ by Proposition \ref{pro: model1-alt-sol} and all dual feasible solutions give the lower bound on the $\theta_p$. 
Therefore, $LB$ can be updated with $b^\top y_{tmp}$ instead of $\theta^k$, which reduces the number of iterations of Algorithm \ref{p alg}.
However, Algorithm \ref{p alg} obtains approximate feasible solutions in practice.
That is, $z_{tmp} \notin \mathcal{K}$ or $b^\top y_{tmp} < \theta^k$ might hold.
Considering that $(y_{tmp}, z_{tmp})$ might be an approximate dual feasible solution, we update $LB$ as follows:
\begin{equation}
\label{LB update1}
LB := 
\begin{cases}
\max \{ \theta^k, b^\top y_{tmp} \} & z_{tmp} \in \mathcal{K} \\
\theta^k & \mbox{otherwise}
\end{cases}
.
\end{equation}

\subsubsection{Keeping and updating a current dual feasible solution}
\label{sec: tech2}
The modification in this section has the same spirit as Section \ref{sec: tech1}.
That is, we update $LB$ using dual feasible solutions.
In the previous section, we propose to update $LB$ as in (\ref{LB update1}) only if $(y_{tmp}, z_{tmp})$ is a feasible solution for (D).
However, even if $(y_{tmp}, z_{tmp})$ is an approximate feasible solution, it can be used to update $LB$ as long as we have a dual feasible solution $(\bar{y}, \bar{z})$. 
By considering the linear combination of $y_{tmp}$ and $\bar{y}$, $(y_{new}, z_{new})$ might be obtained such that $z_{new} = c- \mathcal{A}^* y_{new} \in \mathcal{K}$ and $b^\top y_{new} \geq \max \{ \theta^k, b^\top \bar{y} \}$, which leads to reducing the number of iterations of Algorithm \ref{p alg}.
In addition, this modification allows Algorithm \ref{p alg} to return the current dual feasible solution $(\bar{y}, \bar{z})$.

Therefore, the following operations are added to Algorithm \ref{p alg}. 
\begin{enumerate}
\item Initialize $\bar{y}$ as $\bar{y} \leftarrow \emptyset$. ( If a dual feasible solution $y$ is known in advance, initialize $\bar{y}$ and $LB$ as $\bar{y} \leftarrow y$ and $LB \leftarrow b^\top y$, respectively.)
\item Suppose that $(y_{tmp}, z_{tmp})$ is obtained at the $k$-th iteration. Then, we perform the following operations. 
\begin{itemize}
\item If $z_{tmp} \in \mathcal{K}$ and $\bar{y}$ is empty, then $\bar{y} = y_{tmp}$. 
\item If $\bar{y}$ is not empty, compute $y_{new}$ such that $b^\top y_{new} \geq b^\top \bar{y}$ and $c-\mathcal{A}^* y_{new} \in \mathcal{K}$, and then update $LB$ and $\bar{y}$ as $LB = \max \{ \theta^k, b^\top y_{new} \}$ and $\bar{y} = y_{new}$, respectively.
\end{itemize}
\item Return $\bar{y}$ on line 21 if $\bar{y}$ is not empty. 
\end{enumerate}

In Algorithm \ref{p alg}, $\bar{y}$ plays the role of the current dual feasible solution.
If we have a dual feasible solution before running Algorithm \ref{p alg}, it can be used to initialize $\bar{y}$.
Since Algorithm \ref{p alg} is used as a post-processing step for the interior point method, $\bar{y}$ can be initialized with the output of the interior point method.
The method for computing $y_{new}$ is described in the Appendix \ref{Appendix A}.

\subsubsection{Use of scaling information from previous iterations}
\label{sec: tech5}
Algorithm \ref{p alg} calls the projection and rescaling algorithm of \cite{Kanoh2023} iteratively.
It is natural to consider ways to reduce the computational time of the projection and rescaling algorithm at the $k+1$-th iteration using the information obtained by the $k$-th iteration. 
Recall that the projection and rescaling algorithm of \cite{Kanoh2023} solves the feasibility problem $\mbox{FP}_{S_\infty} (\mathcal{L}, \mbox{int } \mathcal{K})$ by repeating two steps: 
(i). find a cut for $\mbox{FP}_{S_\infty} (\mathcal{L}, \mbox{int } \mathcal{K})$,
(ii). scale the problem to an isomorphic problem equivalent to $\mbox{FP}_{S_\infty} (\mathcal{L}, \mbox{int } \mathcal{K})$ such that the region narrowed by the cut is efficiently explored. 
Therefore, in Algorithm \ref{p alg}, if the cuts obtained by a projection and rescaling algorithm at the $i(<k)$-th iteration hold for any feasible solution of the feasibility problem considered at the $k$-th iteration, such cuts can be used to reduce the execution time of the projection and rescaling algorithm at the $k$-th iteration. 
Proposition \ref{pro: p scale info inherit 2} provides the sufficient condition for a cut to $\mbox{FP}_{S_\infty} (\mbox{ker} \mathcal{A}(\bar{\theta}), \mbox{int } \bar{\mathcal{K}})$ to be valid for $\mbox{FP}_{S_\infty} (\mbox{ker} \mathcal{A}(\theta), \mbox{int } \bar{\mathcal{K}})$ for two real values $\bar{\theta} \in \mathbb{R}$ and $\theta \in \mathbb{R}$.

\begin{proposition}
\label{pro: p scale info inherit 2}
Suppose that (P) is strongly feasible and that $\bar{\theta} \in \mathbb{R}$ satisfies $\bar{\theta} > \theta_p$ and $\bar{\theta} - \theta_p \leq 1$.
Then, $\mbox{FP}_{S_\infty} (\mbox{ker} \mathcal{A}(\theta), \mbox{int } \bar{\mathcal{K}})$ is feasible for any $\theta \in \mathbb{R}$ such that $\theta_p < \theta < \bar{\theta}$.
Furthermore, if $\langle v,\bar{x} \rangle \leq \xi_x$, $\bar{\tau} \leq \xi_\tau$ and $\bar{\rho} \leq \xi_\rho$ hold for any feasible solution $(\bar{x}, \bar{\tau}, \bar{\rho})$ of $\mbox{FP}_{S_\infty} (\mbox{ker} \mathcal{A}(\bar{\theta}), \mbox{int } \bar{\mathcal{K}})$ and for some $\xi_x <1, \xi_\tau<1, \xi_\rho < 1$ and $v \in \mathcal{K}$, then, for any $\theta \in \mathbb{R}$ such that $\theta_p < \theta < \bar{\theta}$ and for any feasible solution $(x, \tau, \rho)$ of $\mbox{FP}_{S_\infty} (\mbox{ker} \mathcal{A}(\theta), \mbox{int } \bar{\mathcal{K}})$, $\langle v,x \rangle \leq \xi_x$, $\tau \leq \xi_\tau$ and $\rho \leq \xi_\rho$ hold.
\end{proposition}

\begin{proof}
For any $\theta \in \mathbb{R}$ such that $\theta_p < \theta < \bar{\theta}$, $\mbox{FP}_{S_\infty} (\mbox{ker} \mathcal{A}(\theta), \mbox{int } \bar{\mathcal{K}})$ is feasible from Proposition \ref{pro: model1-feas}.
Let $(x, \tau, \rho)$ be a feasible solution for $\mbox{FP}_{S_\infty} (\mbox{ker} \mathcal{A}(\theta), \mbox{int } \bar{\mathcal{K}})$.
Then, we have
\begin{equation}
\label{proof: eq1}
\langle c, x \rangle - \tau \theta + \rho = 0.
\end{equation}
Noting that $\tau > 0$, we have $\langle c, \frac{1}{\tau} x \rangle + \frac{1}{\tau} \rho = \theta$.
Since $\rho >0$ and $\frac{1}{\tau} x$ is an interior feasible solution for (P), $\theta_p < \langle c, \frac{1}{\tau} x \rangle < \theta < \bar{\theta}$.
Thus, we find 
\begin{equation}
\label{proof: eq2}
\bar{\theta} - \langle c, \frac{1}{\tau} x \rangle <  \bar{\theta} - \theta_p.
\end{equation}
Let $\alpha = \bar{\theta} - \theta$.
By substituting $\theta = \bar{\theta} - \alpha$ into (\ref{proof: eq1}), we have
\begin{equation}
\label{proof: eq3}
\langle c, x \rangle - \tau \bar{\theta} + \tau \alpha + \rho = 0.
\end{equation}
Since $\bar{\theta} - \theta_p \leq 1$ holds from the assumption, we find
\begin{align*}
\tau \alpha + \rho
&= \tau \bar{\theta} - \langle c,x \rangle \ \ \ \mbox{(by  (\ref{proof: eq3}))}	\\
&< \tau (\bar{\theta} - \theta_p)	\ \ \ \mbox{(by  (\ref{proof: eq2}))}	\\
&\leq \tau \leq 1.
\end{align*}

Therefore, $\tau \alpha + \rho < 1$ holds, and hence $(x, \tau, \tau \alpha + \rho)$ is a feasible solution for $\mbox{FP}_{S_\infty} (\mbox{ker} \mathcal{A}(\bar{\theta}), \mbox{int } \bar{\mathcal{K}})$.
From the assumption, $\langle v,x \rangle \leq \xi_x$, $\tau \leq \xi_\tau$ and $\tau \alpha + \rho \leq \xi_\rho$ hold, which completes the proof.
\end{proof}

Based on Proposition \ref{pro: p scale info inherit 2}, we add the following operations to Algorithm \ref{p alg}. 
\begin{enumerate}
\item Initialize $\bar{v}$ as $\bar{v} \leftarrow (e,1,1)$.
\item Call the projection and rescaling algorithm with $Q_{\bar{v}} \left( \mbox{ker} \mathcal{A}(\theta^k) \right)$ and $\bar{\mathcal{K}}$ to solve $\mbox{FP} ( \mbox{ker} \mathcal{A}(\theta^k), \mbox{int } \bar{\mathcal{K}})$.
\begin{itemize}
\item Suppose that the projection and rescaling algorithm returns the solution of the input problem by finding the solution of the scaled problem $\mbox{FP}_{S_\infty} ( Q_v \left( \mbox{ker} \mathcal{A}(\theta^k) \right), \mbox{int } \bar{\mathcal{K}})$, where $v = (v_1, v_2, v_3) \in \mbox{int } \bar{\mathcal{K}}$. 
\begin{itemize}
\item After updating $UB$ as $UB \leftarrow \theta^k$, if $UB - LB \leq 1$ holds, then preserve the scaling information as $\bar{v} \leftarrow v$. 
\end{itemize}
\end{itemize}
\end{enumerate}

\subsubsection{Use of approximate optimal solutions}
\label{sec: tech7}
Let us consider the indicator $\delta_\infty (\mathcal{L} \cap \mbox{int } \mathcal{K}) := \max_x \left \{ {\rm det}(x) \mid x \in \mathcal{L} \cap \mbox{int } \mathcal{K}, \|x\|_\infty = 1 \right\}$, which is equivalent to the indicators used in  \cite{Pena2017} and Section 6.2 of \cite{Kanoh2023}.
If $\mathcal{L} \cap \mbox{int } \mathcal{K} \neq \emptyset$, then $\delta_\infty (\mathcal{L} \cap \mbox{int } \mathcal{K}) \in (0,1]$ holds, and if $e \in \mathcal{L} \cap \mbox{int } \mathcal{K}$, then $\delta_\infty (\mathcal{L} \cap \mbox{int } \mathcal{K}) =1$ holds.
The larger the value of the indicator $\delta_\infty (\mathcal{L} \cap \mbox{int } \mathcal{K})$, the sooner projection and rescaling algorithms will find a solution for $\mbox{FP} (\mathcal{L}, \mbox{int } \mathcal{K})$.
Therefore, scaling with $v \in \mbox{int } \bar{\mathcal{K}}$ such that $\delta_\infty (Q_v (\mbox{ker} \mathcal{A} (\theta)) \cap \mbox{int } \bar{\mathcal{K}}) \geq \delta_\infty (\mbox{ker} \mathcal{A} (\theta) \cap \mbox{int } \bar{\mathcal{K}})$ holds can reduce the computational time of Algorithm \ref{p alg}. 

Let us introduce Proposition \ref{pro: basic idea}, which gives us the basic idea of obtaining $v \in \mbox{int } \bar{\mathcal{K}}$ such that $\delta_\infty (Q_v (\mbox{ker} \mathcal{A} (\theta)) \cap \mbox{int } \bar{\mathcal{K}}) \geq \delta_\infty (\mbox{ker} \mathcal{A} (\theta) \cap \mbox{int } \bar{\mathcal{K}})$.
\begin{proposition}
\label{pro: basic idea}
Suppose that (P) is strongly feasible and for a given $\theta \in \mathbb{R}$ and an interior feasible solution $x \in \mathbb{E}$ for (P), $\max \{ \theta_p, \theta - 1 \} < \langle c,x \rangle < \theta$ holds. 
Define $v \in \mathbb{E} \times \mathbb{R}^2$ as $v := (x^{-\frac{1}{2}}, 1, 1)$.
Then, the following inequality holds:
\begin{equation}
\notag
\delta_\infty (Q_v (\mbox{ker} \mathcal{A} (\theta)) \cap \mbox{int } \bar{\mathcal{K}}) \geq \theta - \langle c,x \rangle.
\end{equation}
\end{proposition}
\begin{proof}
It is obvious that $(x, 1, \theta - \langle c,x \rangle) \in \mbox{ker} \mathcal{A} (\theta) \cap \mbox{int } \bar{\mathcal{K}}$.
Noting that $Q_v \left( (x, 1, \theta - \langle c,x \rangle) \right) = (e, 1, \theta - \langle c,x \rangle) \in Q_v (\mbox{ker} \mathcal{A} (\theta)) \cap \mbox{int } \bar{\mathcal{K}}$ and $\theta - \langle c,x \rangle < 1$, we find $\delta_\infty (Q_v (\mbox{ker} \mathcal{A} (\theta)) \cap \mbox{int } \bar{\mathcal{K}}) \geq \det(e) \times 1 \times (\theta - \langle c,x \rangle) = \theta - \langle c,x \rangle$.
\end{proof}

The next corollary follows similarly to Proposition \ref{pro: basic idea}.
\begin{corollary}
\label{coro: basic idea}
Suppose that (P) is strongly feasible and for a given $\theta \in \mathbb{R}$, $\theta_p < \theta - 1$ holds. 
For any interior feasible solution $x \in \mathbb{E}$ for (P) such that $\langle c,x \rangle = \theta-1$, define $v \in \mathbb{E} \times \mathbb{R}^2$ as $v := (x^{-\frac{1}{2}}, 1, 1)$.
Then, $\delta_\infty (Q_v (\mbox{ker} \mathcal{A} (\theta)) \cap \mbox{int } \bar{\mathcal{K}}) = 1$ holds. 
\end{corollary}

Proposition \ref{pro: basic idea} and Corollary \ref{coro: basic idea} raise interest in whether the following assumption holds or not. 
\begin{assumption}
\label{assumption}
Suppose that (P) is strongly feasible.
Let $x \in \mathbb{E}$ and $s \in \mathbb{E}$ be interior feasible solutions for (P) such that $\max \{ \theta_p, \theta - 1 \} < \langle c,x \rangle < \langle c, s \rangle < \theta$ for a given $\theta \in \mathbb{R}$. 
Define $v_x \in \mathbb{E} \times \mathbb{R}^2$ and $v_s \in \mathbb{E} \times \mathbb{R}^2$ as $v_x  := (x^{-\frac{1}{2}}, 1, 1)$ and $v_s  := (s^{-\frac{1}{2}}, 1, 1)$, respectively. 
Then, the following relation holds.
\begin{equation}
\notag
\delta_\infty (Q_{v_x} (\mbox{ker} \mathcal{A} (\theta)) \cap \mbox{int } \bar{\mathcal{K}}) > \delta_\infty (Q_{v_s} (\mbox{ker} \mathcal{A} (\theta)) \cap \mbox{int } \bar{\mathcal{K}}).
\end{equation}
\end{assumption}

Unfortunately, the authors could not prove whether Assumption \ref{assumption} holds. 
What prevents us from proving this assumption is revealed by the following proposition. 

\begin{proposition}
\label{pro: assumption}
Suppose that (P) is strongly feasible and that for a given $\theta \in \mathbb{R}$ and an interior feasible solution $s \in \mathbb{E}$ for (P), $\max \{ \theta_p,\theta-1\} < \langle c,s \rangle < \theta$ holds.
Let $v_s  := (s^{-\frac{1}{2}}, 1, 1)$ and $(s_1, s_2, s_3)$ be the point giving the maximum value of $\delta_\infty (Q_{v_s} (\mbox{ker} \mathcal{A} (\theta)) \cap \mbox{int } \bar{\mathcal{K}})$.
If $s_1 = e$, then for any interior feasible solution $x \in \mathbb{E}$ for (P) such that $\max \{ \theta_p, \theta - 1 \} < \langle c,x \rangle < \langle c, s \rangle$, 
\begin{equation}
\notag
\delta_\infty (Q_{v_x} (\mbox{ker} \mathcal{A} (\theta)) \cap \mbox{int } \bar{\mathcal{K}}) > \delta_\infty (Q_{v_s} (\mbox{ker} \mathcal{A} (\theta)) \cap \mbox{int } \bar{\mathcal{K}})
\end{equation}
holds, where $v_x  := (x^{-\frac{1}{2}}, 1, 1)$.

\end{proposition}

\begin{proof}
It is obvious that $(e, 1, \theta - \langle c,s \rangle ) \in Q_{v_s} (\mbox{ker} \mathcal{A} (\theta)) \cap \mbox{int } \bar{\mathcal{K}}$ and $\left( Q_{s^\frac{1}{2}} (s_1), s_2, s_3 \right) \in \mbox{ker} \mathcal{A} (\theta) \cap \mbox{int } \bar{\mathcal{K}}$ holds. 

If $s_1 = e$, then $Q_{s^\frac{1}{2}} (s_1) = s$ holds. 
In addition, we can easily see that $s_2 = 1$ and $s_3 = \theta - \langle c,s \rangle$ since $s$ is a feasible solution for (P) and $(s, s_2, s_3) \in \mbox{ker} \mathcal{A} (\theta) \cap \mbox{int } \bar{\mathcal{K}}$. 
Thus, we have
\begin{equation}
\notag
\delta_\infty (Q_{v_s} (\mbox{ker} \mathcal{A} (\theta)) \cap \mbox{int } \bar{\mathcal{K}}) = \det(e) \times 1 \times ( \theta -  \langle c,s \rangle) = \theta - \langle c,s \rangle.
\end{equation}
By Proposition \ref{pro: basic idea}, for any interior feasible solution $x$ for (P) such that $\max \{ \theta_p, \theta-1 \} < \langle c,x \rangle  < \langle c,s \rangle$, we find $\delta_\infty (Q_{v_x} (\mbox{ker} \mathcal{A} (\theta)) \cap \mbox{int } \bar{\mathcal{K}}) \geq \theta - \langle c,x \rangle$, and hence we have $\delta_\infty (Q_{v_x} (\mbox{ker} \mathcal{A} (\theta)) \cap \mbox{int } \bar{\mathcal{K}}) > \delta_\infty (Q_{v_s} (\mbox{ker} \mathcal{A} (\theta)) \cap \mbox{int } \bar{\mathcal{K}})$.
%
\end{proof}

As in the proof of Proposition \ref{pro: assumption}, Assumption \ref{assumption} holds in the case $s_1 = e$. 
However, proving whether Assumption \ref{assumption} holds when $s_1 \neq e$ is difficult. 
In this case, the specific value or upper bound of $\delta_\infty (Q_{v_s} (\mbox{ker} \mathcal{A} (\theta)) \cap \mbox{int } \bar{\mathcal{K}})$ is challenging to obtain. 
Even if $(s_1, s_2, s_3)$ were obtained, it is unclear how the existence of $(s_1, s_2, s_3)$ such that $s_1 \neq e$ affects the value of $\delta_\infty (Q_{v_x} (\mbox{ker} \mathcal{A} (\theta)) \cap \mbox{int } \bar{\mathcal{K}})$ for any interior feasible solution $x$ satisfying $\max \{ \theta_p, \theta - 1 \} < \langle c,x \rangle < \langle c, s \rangle$.
These obstacles prevent us from proving Assumption \ref{assumption}.

So far, we have focused on scaling with an interior feasible solution $x \in \mathbb{E}$ such that $\max \{ \theta_p, \theta - 1 \} < \langle c,x \rangle < \theta$. 
The critical concern for our algorithm is whether scaling with such $x$ can reduce the execution time of Algorithm \ref{p alg}.
In other words, our algorithm needs to determine whether the following relation holds for any interior feasible solution $x$ satisfying $\max \{ \theta_p, \theta - 1 \} < \langle c,x \rangle < \theta$ and $v = (x^{-\frac{1}{2}}, 1, 1)$.
\begin{equation}
\notag
\delta_\infty (Q_v (\mbox{ker} \mathcal{A} (\theta)) \cap \mbox{int } \bar{\mathcal{K}}) \geq \delta_\infty (\mbox{ker} \mathcal{A} (\theta) \cap \mbox{int } \bar{\mathcal{K}}).
\end{equation}

Unfortunately, the above relation does not always hold.
The easiest counterexample is when $e$ is a feasible solution for (P) and $\langle c, e \rangle = \theta-1$, i.e., $\delta_\infty (\mbox{ker} \mathcal{A} (\theta) \cap \mbox{int } \bar{\mathcal{K}}) = 1$. 
Even if $e$ is not a feasible solution for (P), for the same reason that it is challenging to prove Assumption \ref{assumption}, it is also difficult to ascertain whether this relation holds.
It isn't easy to see whether this relation holds when $x$ is an approximate feasible solution. 
However, the authors believe that Assumption \ref{assumption} holds and scaling with an arbitrary interior feasible solution $x$ such that $\max \{ \theta_p, \theta - 1 \} < \langle c,x \rangle < \theta$ can reduce the computational time of Algorithm \ref{p alg} based on the observation of a simple example discussed in Appendix \ref{Appendix B}. 

The observations in Appendix \ref{Appendix B} imply that even if $x$ is an approximate solution for (P), we can expect that the scaling with such $x$ reduces the computational time of Algorithm \ref{p alg}, as long as $\langle c, x \rangle \simeq \max \{ \theta_p, \theta - 1 \}$ and $e$ is not an approximate or feasible solution for (P) such that $\langle c, e \rangle \simeq \max \{ \theta_p, \theta - 1 \}$. 
Noting that Algorithm \ref{p alg} will be used as a post-processing step of interior point methods, there is no problem in supposing that we have an approximate optimal solution $x \in \mbox{int } \mathcal{K}$ and $\theta$ such that $\theta \simeq \theta_p$, i.e., $\max \{ \theta_p, \theta - 1 \} = \theta_p$, before running Algorithm \ref{p alg}.

Thus, the following operations are added to Algorithm \ref{p alg}.
\begin{enumerate}
\item Initialize $\bar{v}$ as $\bar{v} \leftarrow (e,1,1)$. ( If an approximate primal solution $x \in \mbox{int } \mathcal{K}$ is known in advance, initialize $\bar{v}$ as $\bar{v} \leftarrow (x^{-\frac{1}{2}} ,1,1)$.)
\item Call the projection and rescaling algorithm with $Q_{\bar{v}} \left( \mbox{ker} \mathcal{A}(\theta^k) \right)$ and $\bar{\mathcal{K}}$ to solve $\mbox{FP} ( \mbox{ker} \mathcal{A}(\theta^k), \mbox{int } \bar{\mathcal{K}})$.
\begin{itemize}
\item Suppose that the projection and rescaling algorithm returns the solution of the input problem, and we obtain the solution $(x,\tau,\rho)$ for $\mbox{FP} (\mbox{ker} \mathcal{A}(\theta^k), \mbox{int } \bar{\mathcal{K}})$.
\begin{itemize}
\item After updating $UB$ as $UB \leftarrow \theta^k$, if $UB - LB > 1$ holds, then update $\bar{v}$ as $\bar{v} \leftarrow ((\frac{1}{\tau}x)^{-\frac{1}{2}} ,1,1)$. 
\end{itemize}
\end{itemize}
%
\end{enumerate}
Note that Algorithm \ref{p alg} with this modification does not update $\bar{v}$ in the item (2) above if $UB - LB \leq 1$.
This is because the technique in Section \ref{sec: tech5} updates $\bar{v}$ when $UB - LB \leq 1$.
This modification is not theoretically guaranteed to work.
Our numerical experiments in Section \ref{sec: numerical results} confirm whether these techniques work well.

\subsubsection{Extracting a highly accurate solution}
\label{sec: tech3}
Algorithm \ref{p alg} will obtain approximate solutions many times before they terminate. 
Such solutions can be used to make the accuracy of the outputs from Algorithm \ref{p alg} robust.

The following operations are added to Algorithm \ref{p alg}.
Since our algorithm is intended to be used as a post-processing step, it is assumed that an approximate optimal solution $(x^0, y^0, z^0)$ is known in advance.
\begin{enumerate}
\item Initialize $Sol_{\rm P}$ and $Sol_{\rm D}$ as $Sol_{\rm P} \leftarrow x^0$ and $Sol_{\rm D} \leftarrow y^0$, respectively.
\item If $x_{tmp}$ is obtained at the $k$-th iteration, then we add $x_{tmp}$ to $Sol_{\rm P}$.
\item If $y_{tmp}$ is obtained at the $k$-th iteration, then we add $y_{tmp}$ to $Sol_{\rm D}$.
\item Choose $y^*$ as in (\ref{extract d sol}) and then compute $x^*$ as in (\ref{extract p sol}) using $y^*$ before line 21
\item Return $x^*$ instead of $x_{tmp}$.
\end{enumerate}

Since the set $Sol_{\rm P}$ will contain many approximate primal solutions, Algorithm \ref{p alg} can extract the highly accurate approximate optimal solution $x^*$ from this set as in
\begin{equation}
\label{extract p sol}
x^* \leftarrow \argmin_{x \in Sol_{\rm P}} f(x,y^*,c-\mathcal{A}^* y^*),
\end{equation}
where
\begin{equation}
\notag
f(x,y,z) = \frac{\|\mathcal{A}x - b\|_2}{1+ \displaystyle \max_{i =1, \dots ,m}| b_i |} + 
\max \left\{ 0, \frac{-\lambda_{\min} (x)}{1+\displaystyle \max_{i =1, \dots ,m}| b_i |} \right \} + 
\frac{|\langle c, x \rangle - b^\top y|}{1+|\langle c,x \rangle| + |b^\top y|} + 
\frac{|\langle x, z\rangle|}{1+|\langle c,x \rangle| + |b^\top y|}.
\end{equation}
Each term of the function $f(x,y,z)$ is based on the DIMACS Error \cite{Mittelmann2003}, which is a measure of accuracy as an optimal solution for (P) and (D). 
Note that a dual solution $(y,z)$ is required to extract $x^*$ as in (\ref{extract p sol}).
If the dual optimal solution were known, $f(x,y,z)$ could be used to accurately evaluate the accuracy of $x$ as the primal optimal solution, but such cases would be sporadic. 
Thus, we choose $y^*$ from $Sol_{\rm D}$ and use it as the approximate optimal dual solution to extract $x^*$.
Algorithm \ref{p alg} chooses $y^*$ as in
\begin{equation}
\label{extract d sol}
y^* \leftarrow \argmax_{y \in Sol} b^\top y,
\end{equation}
where $Sol := \left \{ y \in Sol_{\rm D} : \lambda_{\min} (c-\mathcal{A}^*y) \geq \min \{ \lambda_{\min} (c-\mathcal{A}^*y^0), 0 \} \right\}$.
With the modification of Algorithm \ref{p alg} proposed in Section \ref{sec: tech2}, we can use a current dual solution $(\bar{y}, c-\mathcal{A}^* \bar{y})$ to extract $x^*$ as long as $\bar{y}$ is not empty.

\subsection{Practical versions of Algorithms \ref{p alg} and \ref{d alg}}
\label{sec: Algorithm}
We now describe Algorithms \ref{p alg} and \ref{d alg} employing the modifications proposed in the previous section.
Both algorithms are designed to use the approximate optimal solutions $(x^0, y^0, z^0)$ of (P) and (D).
Algorithms \ref{Practical p alg} and \ref{Practical d alg} are practical versions of Algorithms \ref{p alg} and \ref{d alg}, respectively.
These algorithms terminate when $UB-LB \leq \theta_{acc}$ holds or when a vector that determines the feasibility status of (P) or (D) is found.
We consider that a reducing direction for (P) is obtained when Algorithm \ref{Practical p alg} finds a vector $(y, \gamma)$ such that
\begin{equation}
\notag
|\gamma| \leq \mbox{1e-12}, \ |-b^\top y| \leq \mbox{1e-12}, \ \lambda_{\min} (\mathcal{A}^*y) \geq \mbox{-1e-12}, \ \mbox{and} \ \|\mathcal{A}^*y\| > \mbox{1e-12}.
\end{equation}
In addition, we consider that an improving ray of (D) is obtained when Algorithm \ref{Practical p alg} finds a vector $(y, \gamma)$ such that
\begin{equation}
\notag
|\gamma| \leq \mbox{1e-12}, \ -b^\top y > \mbox{1e-12}, \ \mbox{and} \ \frac{1}{-b^\top y} \lambda_{\min} (\mathcal{A}^*y) \geq \mbox{-1e-12}.
\end{equation}
Similarly, when Algorithm \ref{Practical d alg} finds a vector $(x, \tau, \rho)$ satisfying
\begin{equation}
\notag
|\tau| \leq \mbox{1e-12}, \ | \langle c,x \rangle | \leq \mbox{1e-12}, \ \lambda_{\min} (x) \geq \mbox{-1e-12}, \ \mbox{and} \ \|x\| > \mbox{1e-12},
\end{equation}
we consider that a reducing direction for (D) is obtained, and when Algorithm \ref{Practical d alg} finds a vector satisfying
\begin{equation}
\notag
|\tau| \leq \mbox{1e-12}, \ \langle c,x \rangle < \mbox{-1e-12}, \ \mbox{and} \ \frac{1}{- \langle c,x \rangle} \lambda_{\min} (x) \geq \mbox{-1e-12},
\end{equation}
we consider that an improving ray of (P) is obtained. 

 \begin{algorithm}[H]
 \caption{Practical version of Algorithm \ref{p alg}}
 \label{Practical p alg}
 \begin{algorithmic}[1]
 \renewcommand{\algorithmicrequire}{\textbf{Input: }}
 \renewcommand{\algorithmicensure}{\textbf{Output: }}
  \renewcommand{\stop}{\textbf{stop }}
 \renewcommand{\return}{\textbf{return }}

 \STATE \algorithmicrequire $\mathcal{A}$, $b$, $c$, $\mathcal{K}$, $(x^0,y^0,z^0)$ and $\theta_{acc} > 0$.
 \STATE \algorithmicensure A primal solution $x$ (and a dual solution $y$) or a vector that determines the feasibility status of (P).
 \STATE initialization: $k \leftarrow 0, \ LB \leftarrow -\infty, UB \leftarrow \infty, \bar{\mathcal{K}} \leftarrow \mathcal{K} \times \mathbb{R}^{2}_+, \bar{y} \leftarrow \emptyset, \bar{v} \leftarrow (e,1,1), Sol_{\rm P} \leftarrow x^0, Sol_{\rm D} \leftarrow y^0$
 \\
 \ \ \ \ \ \ \ \ \ \ \ \ \ \ \ \ \ \ \  ( If $c-\mathcal{A}^* y^0 \in \mathcal{K}$, $\bar{y} \leftarrow y^0$, $LB \leftarrow b^\top \bar{y}$. If $x^0 \in {\rm int} \mathcal{K}$, $\bar{x} \leftarrow x^0$, \ $\bar{v} \leftarrow (\bar{x}^{- \frac{1}{2}},1,1)$. )
 \STATE Choose $\theta^k \in (LB, UB)$ and construct
$
\mathcal{A}(\theta^k) =
\begin{pmatrix}
\mathcal{A}	&-b	&\bm{0}	\\
c^\top 	&-\theta^k	&1
\end{pmatrix}
.
$
 \STATE Let $\varepsilon$ be a sufficiently small positive value and $\xi$ be a constant such that $0<\xi<1$.
 \WHILE {$UB-LB > \theta_{acc}$}
 \STATE Call the projection and rescaling algorithm of \cite{Kanoh2023} with $Q_{\bar{v}}(\mbox{ker} \mathcal{A}(\theta^k))$, $\bar{\mathcal{K}}$, $\varepsilon$, and $\xi$  \\ \ \ \ \ \ to solve $\mbox{FP} ( \mbox{ker} \mathcal{A}(\theta^k), \mbox{int } \bar{\mathcal{K}})$ and then obtain the scaling information $v = (v_1, v_2, v_3)$.
 \IF {a solution $(x, \tau, \rho)$ to $\mbox{FP} (\mbox{ker} \mathcal{A}(\theta^k), \mbox{int } \bar{\mathcal{K}})$ is obtained}
 \STATE $x_{tmp} \leftarrow \frac{1}{\tau} x$, $Sol_{\rm P} \leftarrow Sol_{\rm P} \cup \{ x_{tmp} \}$, $UB \leftarrow \theta^k$
 \IF {$UB-LB \leq 1$}
 \STATE $\bar{v} \leftarrow (v_1, v_2, v_3)$
 \ELSE
 \IF{$x_{tmp} \in {\rm int} \mathcal{K}$}
 \STATE $\bar{v} \leftarrow (x_{tmp}^{- \frac{1}{2}},1,1)$
 \ENDIF
 \ENDIF
 \ELSIF {a solution $(z, \omega, \kappa)$ to $\mbox{FP} (\mbox{range} {\mathcal{A}(\theta^k)}^*, \bar{\mathcal{K}} \setminus \{0\})$ is obtained}
 \STATE Compute $(y, \gamma) \in \mathbb{R}^{m+1}$ such that $z = \mathcal{A}^*y + \gamma c$, $\omega = -b^\top y - \gamma \theta$ and $\kappa = \gamma$.

 \IF {$-y$ is an improving ray of (D) or a reducing direction for (P)}
 \STATE \stop Algorithm \ref{Practical p alg} and \return $(-y, \mathcal{A}^*y)$
 \ELSE
 \STATE $y_{tmp} \leftarrow \frac{-1}{\kappa} y$, \ $z_{tmp} \leftarrow c - \mathcal{A}^*y_{tmp}$, $Sol_{\rm D} \leftarrow Sol_{\rm D} \cup \{ y_{tmp} \}$, $LB \leftarrow \theta^k$

 \IF {$\bar{y} \neq \emptyset$}
 \STATE Compute $y_{new}$ such that $b^\top y_{new} \geq b^\top \bar{y}$ and $c-\mathcal{A}^* y_{new} \in \mathcal{K}$ using $\bar{y}$ and $y_{tmp}$.
 \STATE $\bar{y} \leftarrow y_{new}$, $LB \leftarrow \max \{ \theta^k, b^\top \bar{y} \}$
 \ELSE
 \IF {$z_{tmp} \in \mathcal{K}$}
 \STATE $\bar{y} \leftarrow y_{tmp}$, $LB \leftarrow \max \{ \theta^k, b^\top \bar{y} \}$
 \ENDIF

 \ENDIF
 \ENDIF
 \ELSE
 \STATE $LB \leftarrow \theta^k$
 \ENDIF
 \STATE Choose $\theta^{k+1} \in (LB, UB)$, \ $k \leftarrow k+1$
 \ENDWHILE
 \STATE Choose $y^*$ from $Sol_{\rm D}$ as in (\ref{extract d sol})  \ \ \ \ \ \ \ \ \ \ \ \ \ //  If $\bar{y} = \emptyset$, $y^* \leftarrow \bar{y}$.
 \STATE Choose $x^*$ from $Sol_{\rm P}$ as in (\ref{extract p sol}) 
 \STATE \return $x^*$ (and $y^*$)
 \end{algorithmic} 
 \end{algorithm}

 \begin{algorithm}[H]
 \caption{Practical version of Algorithm \ref{d alg}}
 \label{Practical d alg}
 \begin{algorithmic}[1]
 \renewcommand{\algorithmicrequire}{\textbf{Input: }}
 \renewcommand{\algorithmicensure}{\textbf{Output: }}
  \renewcommand{\stop}{\textbf{stop }}
 \renewcommand{\return}{\textbf{return }}

 \STATE \algorithmicrequire $\mathcal{A}$, $b$, $c$, $\mathcal{K}$, $(x^0,y^0,z^0)$ and $\theta_{acc} > 0$.
 \STATE \algorithmicensure A dual solution $y$ (and a primal solution $x$) or a vector that determines the feasibility status of (D).
 \STATE initialization: $k \leftarrow 0, \ LB \leftarrow -\infty, UB \leftarrow \infty, \bar{\mathcal{K}} \leftarrow \mathcal{K} \times \mathbb{R}^{2}_+, \bar{y} \leftarrow \emptyset, \bar{v} \leftarrow (e,1,1), Sol_{\rm P} \leftarrow x^0, Sol_{\rm D} \leftarrow y^0$
\\
 \ \ \ \ \ \ \ \ \ \ \ \ \ \ \ \ \ \ \  ( If $c-\mathcal{A}^* y^0 \in \mathcal{K}$, $\bar{y} \leftarrow y^0$, $LB \leftarrow b^\top \bar{y}$. If $z^0 \in {\rm int} \mathcal{K}$, $\bar{z} \leftarrow z^0$, \ $\bar{v} \leftarrow (\bar{z}^{- \frac{1}{2}},1,1)$. )
 \STATE Choose $\theta^k \in (LB, UB)$ and construct
$
\mathcal{A}(\theta^k) =
\begin{pmatrix}
\mathcal{A}	&-b	&\bm{0}	\\
c^\top 	&-\theta^k	&1
\end{pmatrix}
.
$
 \STATE Let $\varepsilon$ be a sufficiently small positive value and $\xi$ be a constant such that $0<\xi<1$.
 \WHILE {$UB-LB > \theta_{acc}$}
 \STATE Call the projection and rescaling algorithm of \cite{Kanoh2023} with $Q_{\bar{v}}(\mbox{range} {\mathcal{A}(\theta^k)}^*)$, $\bar{\mathcal{K}}$, $\varepsilon$, and $\xi$  \\ \ \ \ \ \ to solve $\mbox{FP} ( \mbox{range} {\mathcal{A}(\theta^k)}^*, \mbox{int } \bar{\mathcal{K}})$ and then obtain the scaling information $v = (v_1, v_2, v_3)$.

 \IF {a solution $(z, \omega, \kappa)$ to $\mbox{FP} (\mbox{range} {\mathcal{A}(\theta^k)}^*, \mbox{int } \bar{\mathcal{K}})$ is obtained}
 \STATE Compute $(y, \gamma) \in \mathbb{R}^{m+1}$ such that $z = \mathcal{A}^*y + \gamma c$, $\omega = -b^\top y - \gamma \theta$ and $\kappa = \gamma$.
 \STATE $y_{tmp} \leftarrow \frac{-1}{\kappa} y$, \ $z_{tmp} \leftarrow c - \mathcal{A}^*y_{tmp}$, $Sol_{\rm D} \leftarrow Sol_{\rm D} \cup \{ y_{tmp} \}$, $LB \leftarrow \theta^k$
 \IF {$\bar{y} \neq \emptyset$}
 \STATE Compute $y_{new}$ such that $b^\top y_{new} \geq b^\top \bar{y}$ and $c-\mathcal{A}^* y_{new} \in \mathcal{K}$ using $\bar{y}$ and $y_{tmp}$.
 \STATE $\bar{y} \leftarrow y_{new}$, $LB \leftarrow \max \{ \theta^k, b^\top \bar{y} \}$
 \ELSE
 \IF {$z_{tmp} \in \mathcal{K}$}
 \STATE $\bar{y} \leftarrow y_{tmp}$, $LB \leftarrow \max \{ \theta^k, b^\top \bar{y} \}$
 \ENDIF
 \ENDIF
 \IF {$UB-LB \leq 1$}
 \STATE $\bar{v} \leftarrow (v_1, v_2, v_3)$
 \ELSE
 \IF{$z_{tmp} \in {\rm int} \mathcal{K}$}
 \STATE $\bar{v} \leftarrow (z_{tmp}^{- \frac{1}{2}},1,1)$
 \ENDIF
 \ENDIF
 \ELSIF  {a solution $(x, \tau, \rho)$ to $\mbox{FP} (\mbox{ker } \mathcal{A}(\theta^k), \bar{\mathcal{K}} \setminus \{0\})$ is obtained}
 \IF {$x$ is an improving ray of (P) or a reducing direction for (D)}
 \STATE \stop Algorithm \ref{Practical d alg} and \return $x$
 \ENDIF
 \STATE $x_{tmp} \leftarrow \frac{1}{\tau} x$, $Sol \leftarrow Sol \cup \{ x_{tmp} \}$, $UB \leftarrow \theta^k$
 \ELSE 
 \STATE $UB \leftarrow \theta^k$
 \ENDIF
 \STATE Choose $\theta^{k+1} \in (LB, UB)$, \ $k \leftarrow k+1$
 \ENDWHILE
 \STATE Choose $y^*$ from $Sol_{\rm D}$ as in (\ref{extract d sol})  \ \ \ \ \ \ \ \ \ \ \ \ \ //  If $\bar{y} = \emptyset$, $y^* \leftarrow \bar{y}$.
 \STATE Choose $x^*$ from $Sol_{\rm P}$ as in (\ref{extract p sol}) 
 \STATE \return $y^*$ (and $x^*$)
 \end{algorithmic} 
 \end{algorithm}

\section{Numerical results}
\label{sec: numerical results}
In this section, we apply our algorithm to some instances to show the numerical performance of the proposed methods. 
Figure \ref{figure: Flow} shows the flow of our numerical experiments.
First, we solve the instances with SDP solvers to obtain approximate optimal solutions $(x^0, y^0, z^0)$ to (P) and (D).
SDPA \cite{SDPA}, SDPT3 \cite{SDPT3}, and Mosek \cite{Mosek} were used in our experiment.
Next, we call Algorithm \ref{postpro alg} with $(x^0, y^0, z^0)$. 
Algorithm \ref{postpro alg} calls Algorithm \ref{Practical d alg} and Algorithm \ref{Practical p alg} in that order if $\lambda_{\min} (z^0) >$-1e-12. 
Algorithm \ref{Practical d alg} is called first to take advantage of the modification in Section \ref{sec: tech2}. 
Since Algorithm \ref{Practical d alg} can find an approximate optimal interior feasible solution to (D), a highly accurate optimal dual solution $y$ will likely be obtained after running Algorithm \ref{Practical d alg}. 
Such a vector $y$ can reduce the execution time of Algorithm \ref{Practical p alg} via the modification proposed in Section \ref{sec: tech2}.
In our experiment, SDPA and SPDT3 returned $z^0$ such that $\lambda_{\min} (z^0) >$-1e-14 for all instances. 
However, Mosek returned $z^0$ such that $\lambda_{\min} (z^0) <$-1e-10 for some instances in the experiment of Section \ref{sec: numerical results post-pro}. 
Algorithm \ref{postpro alg} processed these problems in the order of Algorithm \ref{Practical p alg} and Algorithm \ref{Practical d alg}. 
This is because we preferred to use the modifications proposed in Section \ref{sec: tech7} rather than find a highly accurate optimal dual solution before executing Algorithm \ref{Practical p alg}.
(The next section will show the effect of the modifications proposed in Section \ref{sec: tech7}.) 
Even if $\lambda_{\min} (z^0) < $-1e-12, we can generate an interior point of the cone $\mathcal{K}$ by adding a slight perturbation to $z^0$.
However, the interior points obtained in this way are not likely to be accurate dual feasible solutions.
On the other hand, Algorithm \ref{Practical p alg} can take advantage of the modifications proposed in Section \ref{sec: tech7} using $x^0$.
If Algorithm \ref{Practical p alg} can find a dual feasible solution, then Algorithm \ref{Practical d alg}, called after Algorithm \ref{Practical p alg}, can also use the modifications proposed in Section \ref{sec: tech7} from the first iteration. 
We note that Mosek did not yield $(x^0, y^0, z^0)$ such that $\lambda_{\min} (z^0) <$ -1e-12 and $\lambda_{\min} (x^0) <$ -1e-12 for all instances in our experiment. 
Thus, Algorithm \ref{postpro alg} did not reach step 30 in our experiment.
As long as $(x^0, y^0, z^0)$ is an approximate interior feasible solution, step 30 of Algorithm \ref{postpro alg} is not reached. 

We set the upper limit for the execution time of Algorithms \ref{Practical p alg} and \ref{Practical d alg} to 30 minutes and $\theta_{acc}$ = 1e-12.
In addition, we added practical termination conditions to Algorithms \ref{Practical p alg} and \ref{Practical d alg} to account for numerical errors.
(See Appendix \ref{Appendix C1}.) 
For detailed settings of our algorithm, see Appendix \ref{Appendix C}.

\begin{figure}[htb]
\begin{center}
\includegraphics[keepaspectratio, scale=0.4]{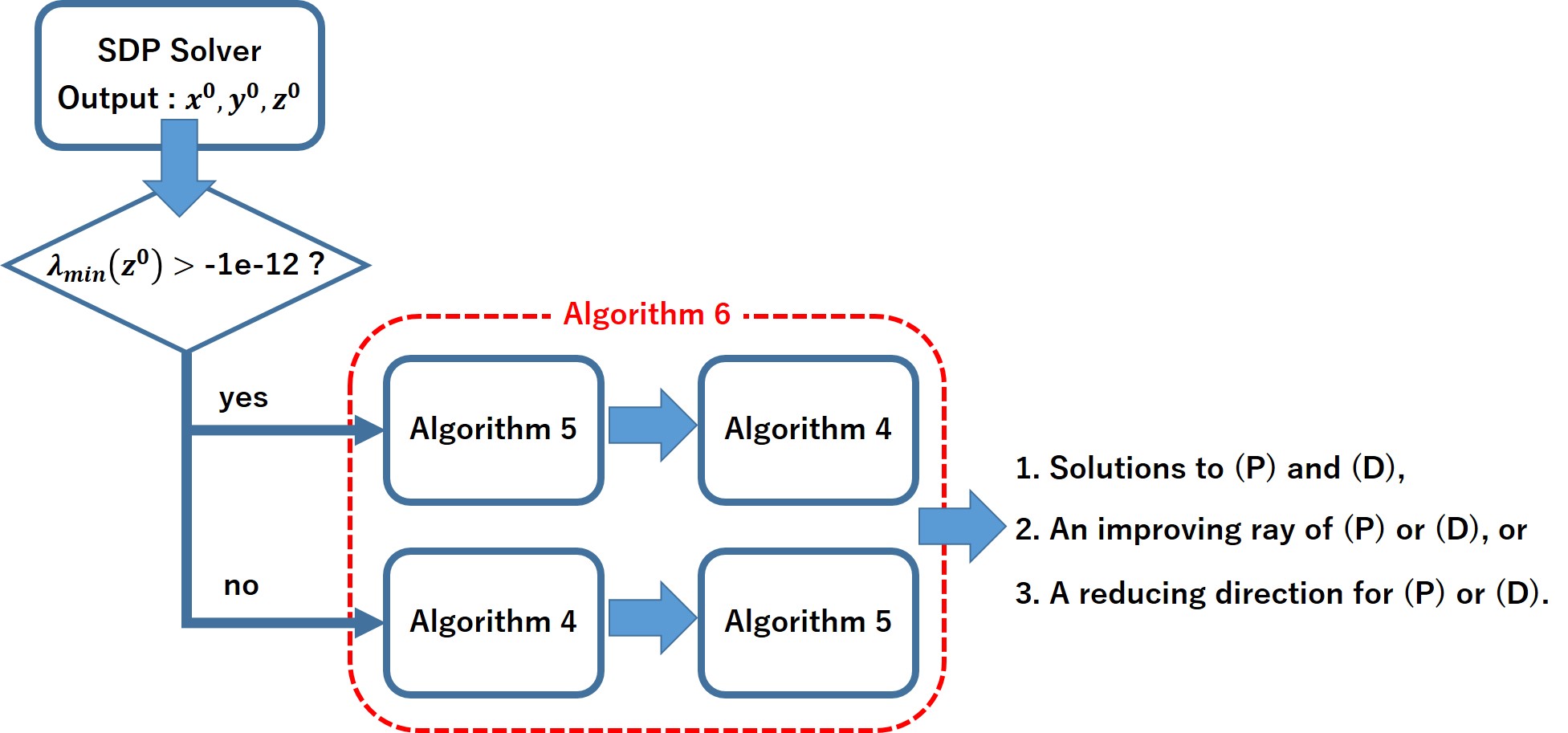}
\end{center}
\caption{Experiment flow}
\label{figure: Flow}
\end{figure}

\begin{algorithm}
 \caption{Proposed Algorithm}
 \label{postpro alg}
 \begin{algorithmic}[1]
 \renewcommand{\algorithmicrequire}{\textbf{Input: }}
 \renewcommand{\algorithmicensure}{\textbf{Output: }}
  \renewcommand{\stop}{\textbf{stop }}
 \renewcommand{\return}{\textbf{return }}

 \STATE \algorithmicrequire $\mathcal{A}$, $b$, $c$, $\mathcal{K}$, $(x^0,y^0,z^0)$ and $\theta_{acc} > 0$.
 \STATE \algorithmicensure Solutions to (P) and (D) or a vector that determines the feasibility status of (P) or (D).
 \IF {$\lambda_{\min} (z^0) \geq$ $-1$e-12}
 \STATE $(x^*,y^*,z^*) \leftarrow (x^0,y^0,z^0)$
 \IF{$z^* \notin \mbox{int} \mathcal{K}$}
 \STATE $z^* \leftarrow z^* + ( -\lambda_{\min} (z^*) +$1e-15$) e$
 \ENDIF
 \STATE Call Algorithm \ref{Practical d alg} with $\mathcal{A}$, $b$, $c$, $\mathcal{K}$, $(x^*,y^*,z^*) $ and $\theta_{acc} > 0$.
 \IF {a solution $(x, y)$ to (P) and (D) is obtained}
 \STATE $x^* \leftarrow x$, $y^* \leftarrow y$, $z^* \leftarrow c-\mathcal{A}^* y$
 \ELSIF {$x$, a reducing direction for (D) or an improving ray of (P), is obtained}
 \STATE \stop Algorithm \ref{postpro alg} and \return $x$
 \ENDIF
 \STATE Call Algorithm \ref{Practical p alg} with $\mathcal{A}$, $b$, $c$, $\mathcal{K}$, $(x^*,y^*,z^*) $ and $\theta_{acc} > 0$.
 \IF {a solution $(x, y)$ to (P) and (D) is obtained}
 \STATE $x^* \leftarrow x$, $y* \leftarrow y$, $z^* \leftarrow c-\mathcal{A}^* y$
 \ELSIF {$(y, z)$, a reducing direction for (P) or an improving ray of (D), is obtained}
 \STATE \stop Algorithm \ref{postpro alg} and \return $(y, z)$
 \ENDIF
 \ELSIF {$\lambda_{\min} (x^0) \geq -1$e-12}
 \STATE $(x^*,y^*,z^*) \leftarrow (x^0,y^0,z^0)$
 \IF {$x^* \notin \mbox{int} \mathcal{K}$}
 \STATE $x^* \leftarrow x^* + (-\lambda_{\min} (x^*) +$1e-15$) e$
 \ENDIF
 \STATE Call Algorithm \ref{Practical p alg} with $\mathcal{A}$, $b$, $c$, $\mathcal{K}$, $(x^*,y^*,z^*)$ and $\theta_{acc} > 0$.
 \STATE Same as lines 15 - 19
 \STATE Call Algorithm \ref{Practical d alg} with $\mathcal{A}$, $b$, $c$, $\mathcal{K}$, $(x^*,y^*,z^*)$ and $\theta_{acc} > 0$.
 \STATE Same as lines 9 - 13
 \ELSE
 \STATE $(x^*,y^*,z^*) \leftarrow (x^0,y^0,z^0)$
 \STATE Same as lines 8 - 19
 \ENDIF
 \STATE \return $(x^*, y^*, z^*)$
 \end{algorithmic} 
 \end{algorithm}

We conducted three types of numerical experiments.
In Section \ref{sec: numerical results heuristic}, we verified whether the modification proposed in Section \ref{sec: tech7} contributes to reducing the execution time of Algorithm \ref{postpro alg}.
In Section \ref{sec: numerical results post-pro}, we checked whether Algorithm \ref{postpro alg} can obtain approximate optimal solutions with higher accuracy than the SDP solvers.
In Section \ref{sec: numerical results check status}, we tested whether Algorithm \ref{postpro alg} detects the feasibility status of SDP more accurately than the SDP solvers.
All executions were performed using MATLAB R2022a on an Intel(R) Core(TM) i9-10980XE CPU @ 3.00GHz machine with 128GB of RAM. 

In Section \ref{sec: numerical results heuristic} and Section \ref{sec: numerical results post-pro}, we measured the accuracy of the output solution $(x,y,z)$ to (P) and (D) using the DIMACS errors \cite{Mittelmann2003}.
Letting $d$ be the dimension of the Euclidean space $\mathbb{E}$ corresponding to $\mathcal{K}$, the DIMACS errors consists of six measures as follows: 
\begin{equation}
\notag
\begin{array}{ll}
{\rm err_1}(x, y, z) = \cfrac{\|\mathcal{A}x - b\|_2}{1+ \displaystyle \max_{i =1, \dots ,m}| b_i |}, &{\rm err_2}(x, y, z) = \max \left\{ 0, \cfrac{-\lambda_{\min} (x)}{1+\displaystyle \max_{i =1, \dots ,m}| b_i |} \right \}, \\
{\rm err_3}(x, y, z) = \cfrac{\| c - \mathcal{A}^*y - z\|_J}{1+ \displaystyle \max_{i =1, \dots ,d}| c_i |},	&{\rm err_4}(x, y, z) = \max \left\{ 0, \cfrac{-\lambda_{\min} (z)}{1+ \displaystyle \max_{i =1, \dots ,d}| c_i | } \right\},	\\
{\rm err_5}(x, y, z) =  \cfrac{\langle c, x \rangle - b^\top y}{1+|\langle c,x \rangle| + |b^\top y|},	&{\rm err_6}(x, y, z) = \cfrac{\langle x, z\rangle}{1+|\langle c,x \rangle| + |b^\top y|}.
\end{array}
\end{equation}

Primal feasibility is measured by ${\rm err_1}(x, y, z)$ and ${\rm err_2}(x, y, z)$ and dual feasibility is measured by ${\rm err_3}(x, y, z)$ and ${\rm err_4}(x, y, z)$.
Optimality is measured by ${\rm err_5}(x, y, z)$ and ${\rm err_6}(x, y, z)$. 
We note that the values of ${\rm err_5}(x, y, z)$ and ${\rm err_6}(x, y, z)$ might be negative as long as $(x,y,z)$ is an approximate solution.
If ${\rm err_5}(x, y, z) = -\delta$ and $\delta > 0$, then ${\rm err_5}(x, y, z)$ corresponds to a ``worse'' solution than if ${\rm err_5}(x, y, z) = \delta$.

\subsection{The effectiveness of the modification in Section \ref{sec: tech7}}
\label{sec: numerical results heuristic}
The instances used in this experiment are picked from SDPLIB \cite{Borchers1999}.
The feasibility status of SDPLIB instances can be inferred based on Freund et al.'s study \cite{Freund2007}.
Since this experiment was only to evaluate the modification in Section \ref{sec: tech7}, we chose only small size instances where both the primal and dual problems are strongly feasible, i.e., control1, control2, control3, truss1, truss3, and truss4.
We first solved these instances with Mosek to obtain an approximate optimal interior feasible solution $(x^0, y^0, z^0)$.
However, Mosek returned dual solutions $(y^0,z^0)$ such that $c-\mathcal{A}^*y^0 \notin \mathcal{K}$ and $z^0 \notin \mathcal{K}$ to truss1, truss3, and truss4.
Thus, we only checked for six instances how the computation time of Algorithm \ref{Practical p alg} changed with and without the modification proposed in Section \ref{sec: tech7}.
In this section, we referred to Algorithm \ref{Practical p alg} without the modification proposed in Section \ref{sec: tech7} as the ``Naive" method. 

Table \ref{Table: heuristic tech} summarizes the results of our experiments.
The ``times(s)" column shows the average CPU time of the corresponding method, and the other columns show the average value of the DIMACS errors.
Table \ref{Table: heuristic tech} shows that the computational time of Algorithm \ref{Practical p alg} is significantly smaller than that of the ``Naive" method. 
Algorithm \ref{Practical p alg} scales the problem with an approximate optimal interior feasible solution before executing the projection and rescaling algorithm. 
Such scaling reduced the number of iterations of the projection and rescaling algorithm in this experiment. 
Thus, the computational time of Algorithm \ref{Practical p alg} is much smaller than that of the ``Naive" method. 
\begin{table}
\caption{Experimental results testing the effectiveness of the modification proposed in Section \ref{sec: tech7}}
\label{Table: heuristic tech}
\begin{center}
\setlength{\tabcolsep}{5pt}
\footnotesize
\begin{tabular}{cccccccr} \toprule
Method	&${\rm err_1}$&${\rm err_2}$	&${\rm err_3}$	&${\rm err_4}$	&${\rm err_5}$	&${\rm err_6}$	&time(s)	\\	\midrule
Naive		&7.15e-14	&0	&0	&0	&3.89e-10	&3.89e-10	&13.36		\\
Algorithm \ref{Practical p alg}		&1.50e-14	&4.81e-17	&0	&0	&1.81e-13	&1.81e-13	&2.66		\\ \bottomrule
\end{tabular}
\end{center}
\end{table}

\subsection{The effectiveness of our post-processing algorithm}
\label{sec: numerical results post-pro}
The outline of our numerical experiment in this section is as follows:
\begin{enumerate}
\item Solve the instances with the SDP solvers using the default settings to obtain approximate optimal solutions $(x_{\rm def}, y_{\rm def}, z_{\rm def})$. 
\item Solve the instances using $(x_{\rm def}, y_{\rm def}, z_{\rm def})$ and Algorithm \ref{postpro alg} to obtain approximate optimal solutions $(x_{\rm pro}, y_{\rm pro}, z_{\rm pro})$. 
\item Solve the instances with the SDP solvers using the tight settings to obtain approximate optimal solutions $(x_{\rm tight}, y_{\rm tight}, z_{\rm tight})$. 
\item Compare the above three results regarding accuracy and computational time. 
\end{enumerate}

In this experiment, we used SDPA, SDPT3, and Mosek.
These solvers have parameters that allow us to tune the feasibility and the optimality tolerances. 
For example, SDPA terminates if an approximate optimal solution $(x, y, z)$ is obtained for parameters $\epsilon_{\rm primal}$, $\epsilon_{\rm dual}$ and $\epsilon_{\rm gap}$ such that
\begin{equation}
\notag
\begin{array}{lll}
\displaystyle \max_{i =1, \dots ,m} | (\mathcal{A}x-b)_i | \leq \epsilon_{\rm primal},
&\displaystyle \max_{i =1, \dots ,d} | (c - \mathcal{A}^*y - z)_i | \leq \epsilon_{\rm dual}, \ \mbox{and} 
&\cfrac{|\langle c, x \rangle - b^\top y|}{ \max \left\{ \left( |\langle c,x \rangle| + |b^\top y| \right)/2,  1 \right\} } \leq \epsilon_{\rm gap}
\end{array}
\end{equation}
are satisfied.
That is, the parameters $\epsilon_{\rm primal}$, $\epsilon_{\rm dual}$ and $\epsilon_{\rm gap}$ are the tolerances for primal feasibility, dual feasibility, and optimality measures, respectively.
The definitions of feasibility and optimality measures are slightly different for each solver, but the other solvers have parameters that play the same role. 
Table \ref{Table: para setting} summarizes the values of the parameters used in our experiment. 
The ``default tolerances" column shows the default values of the parameters for each solver.
The ``tight tolerances" column shows the values of the parameters used as the tight setting in our experiment.

\begin{table}
\caption{Solver setting in our numerical experiments}
\label{Table: para setting}
\begin{center}
\begin{tabular}{cccc|ccc} \toprule
	&\multicolumn{3}{c|}{default tolerances}	&\multicolumn{3}{c}{tight tolerances}	\\	
Solver	&$\epsilon_{\rm primal}$	&$\epsilon_{\rm dual}$	&$\epsilon_{\rm gap}$	&$\epsilon_{\rm primal}$	&$\epsilon_{\rm dual}$	&$\epsilon_{\rm gap}$ \\ \midrule
Mosek	&\multicolumn{3}{c|}{1.0e-08}		&\multicolumn{3}{c}{\multirow{3}{*}{1.0e-12}}	\\
SDPT3	&\multicolumn{3}{c|}{1.0e-08}		\\
SDPA		&\multicolumn{3}{c|}{1.0e-07}		\\ \bottomrule
\end{tabular}
\end{center}
\end{table}

The instances of SDPLIB were used in this experiment.
Note that the instances for which the SDP solver using the default setting returned infeasibility criteria were excluded.
In addition, we excluded some instances due to memory limitations and then conducted numerical experiments with 57 instances in total.
To appropriately observe the results, we classified these instances into two groups, the well-posed and the ill-posed groups, based on Freund et al.'s study \cite{Freund2007}.
The well-posed group includes 34 instances where both the primal and dual problems are expected to be strongly feasible.
On the other hand, the ill-posed group includes 23 instances where either the primal or dual problem is not expected to be strongly feasible.
According to \cite{Freund2007}, SDPLIB does not include the instances whose primal problem is strongly feasible, but the dual is not. 
Thus, the ill-posed group includes only instances where the dual problem is expected to be strongly feasible, but the primal problem is not.

The numerical results for the well-posed group and ill-posed group are summarized in Figures \ref{fig: Mosek-well}-\ref{fig: SDPT3-well} and \ref{fig: Mosek-ill}-\ref{fig: SDPT3-ill}, respectively. 
Figures \ref{fig: Mosek-well}-\ref{fig: SDPT3-ill} summarize the values of DIMACS errors for $(x_{\rm def}, y_{\rm def}, z_{\rm def})$, $(x_{\rm tight}, y_{\rm tight}, z_{\rm tight})$ and $(x_{\rm pro}, y_{\rm pro}, z_{\rm pro})$. 
In this experiment, no output $(x_{\rm pro}, y_{\rm pro}, z_{\rm pro})$ was obtained such that ${\rm err_2} (x_{\rm pro}, y_{\rm pro}, z_{\rm pro})>$1e-11 or ${\rm err_4}(x_{\rm pro}, y_{\rm pro}, z_{\rm pro})>0$.
Therefore, we omitted figures summarizing the values of ${\rm err_2}$ and ${\rm err_4}$.
In Figures \ref{fig: Mosek-well}-\ref{fig: SDPT3-ill}, the black solid line, the blue circle and the inverted orange triangle show the values of the DIMACS errors logarithmized by the base $10$ for $(x_{\rm def}, y_{\rm def}, z_{\rm def})$, $(x_{\rm tight}, y_{\rm tight}, z_{\rm tight})$ and $(x_{\rm pro}, y_{\rm pro}, z_{\rm pro})$, respectively. 
It is not plotted if the corresponding DIMACS error value is $0$.  
We note that the values of $\log_{10} | {\rm err_5}(x,y,z) |$ and $\log_{10} | {\rm err_6}(x,y,z) |$ are plotted because the values of ${\rm err_5}(x,y,z)$ and ${\rm err_6}(x,y,z)$ can be negative.
In addition, we plot the points where the corresponding DIMACS error values are negative by filling them in.
In each graph of Figures \ref{fig: Mosek-well}-\ref{fig: SDPT3-ill}, the horizontal axis indicates the instances, sorted so that the corresponding DIMACS error values for $(x_{\rm def}, y_{\rm def}, z_{\rm def})$ are in ascending order. 

First, let us compare the results for the well-posed group. 
From Figures \ref{fig: Mosek-well}-\ref{fig: SDPT3-well}, we can observe that Algorithm \ref{postpro alg} returned an approximate optimal solution $(x_{\rm pro}, y_{\rm pro}, z_{\rm pro})$ that was superior to the $(x_{\rm def}, y_{\rm def}, z_{\rm def})$ and $(x_{\rm tight}, y_{\rm tight}, z_{\rm tight})$ in terms of feasibility and optimality for almost all instances.
Algorithm \ref{postpro alg} returns $(x_{\rm pro}, y_{\rm pro}, z_{\rm pro})$ such that $z_{\rm pro} = c- \mathcal{A}^* y_{\rm pro}$, thus the value of ${\rm err_3} (x_{\rm pro}, y_{\rm pro}, z_{\rm pro})$ should be $0$ for all instances. 
However, in (b) of Figures \ref{fig: Mosek-well}-\ref{fig: SDPT3-well}, four inverted triangles representing the results for four instances, theta1, theta2, theta3, and theta4, were plotted. 
This is probably due to slight numerical errors in computing $c- \mathcal{A}^* y_{\rm pro}$ for these instances. 
Note that $(x_{\rm pro}, y_{\rm pro}, z_{\rm pro})$ satisfies $\langle c, x_{\rm pro} \rangle – b^\top y_{\rm pro} > 0$ for almost all instances. 

We can also see that the solvers with the tight setting tend to return more accurate approximate optimal solutions $(x_{\rm tight}, y_{\rm tight}, z_{\rm tight})$ than $(x_{\rm def}, y_{\rm def}, z_{\rm def})$.  
However, Mosek and SDPA, with the tight setting, returned strange outputs for some instances. 
Mosek, with the default setting, determined the primal and dual problems are expected to be feasible and returned an approximate optimal solution for all instances. 
The error message ``rescode = 100006" was returned for only one instance. 
This error message means that ``the optimizer is terminated due to slow progress".
On the other hand, when using Mosek with the tight settings, it returned the error message ``rescode = 100006" for 21 instances. 
Among these instances, 4 instances were determined that their primal problems were expected to be infeasible, and 13 instances were not determined their feasibility. 
For the aforementioned 4 instances and 13 instances, Mosek, with the tight setting, provided incorrect infeasibility certificates and inaccurate dual feasible solutions, respectively.
SDPA, with the default setting, determined the primal and/or dual problems are feasible and returned an approximate optimal solution for all instances. 
However, when using SDPA with tight settings, it was determined that at least one of the primal or dual problems was expected to be infeasible for 13 instances, and incorrect primal solutions were returned for 11 instances of them.
A possible reason for these strange outputs is that tightening the tolerances might cause numerical instability. 
Figure \ref{fig: SDPT3-well} shows that SDPT3 with the tight setting worked stably and obtained more accurate approximate optimal solutions $(x_{\rm tight}, y_{\rm tight}, z_{\rm tight})$ than $(x_{\rm def}, y_{\rm def}, z_{\rm def})$ for almost all instances. 
We also solved all instances using SDPT3 with the tolerances $\epsilon_{\rm primal}$, $\epsilon_{\rm dual}$ and $\epsilon_{\rm gap}$ set to 1e-13 and obtained the same results as Figure \ref{fig: SDPT3-well}.

Next, let us compare the results for the ill-posed group. 
Since Mosek, using the default setting, returned the reducing direction for hinf12, Figure \ref{fig: Mosek-ill} excludes the plot representing the result for this instance. 
On the other hand, Algorithm \ref{postpro alg} did not find a reducing direction for (P) in this experiment. 
Thus, Algorithm \ref{postpro alg} terminated by returning an approximate optimal solution for all instances. 
From Figures \ref{fig: Mosek-ill}-\ref{fig: SDPT3-ill}, we can observe that Algorithm \ref{postpro alg} returned an approximate optimal solution $(x_{\rm pro}, y_{\rm pro}, z_{\rm pro})$ that was superior to the $(x_{\rm def}, y_{\rm def}, z_{\rm def})$ and $(x_{\rm tight}, y_{\rm tight}, z_{\rm tight})$ in terms of feasibility for almost all instances.
In addition, we can see that the value of $| {\rm err_5}(x_{\rm pro}, y_{\rm pro}, z_{\rm pro}) |$ or $|{\rm err_6}(x_{\rm pro}, y_{\rm pro}, z_{\rm pro})|$ was larger than the value of $| {\rm err_5}(x_{\rm def}, y_{\rm def}, z_{\rm def}) |$ or $|{\rm err_6}(x_{\rm def}, y_{\rm def}, z_{\rm def})|$ for some instances. 
However, this observation does not imply that $(x_{\rm pro}, y_{\rm pro}, z_{\rm pro})$ was inferior to $(x_{\rm def}, y_{\rm def}, z_{\rm def})$ and $(x_{\rm tight}, y_{\rm tight}, z_{\rm tight})$ in terms of optimality. 

Tables \ref{Table: compare obj value Mosek}-\ref{Table: compare obj value SDPT3} compare the dual objective value of $(x_{\rm pro}, y_{\rm pro}, z_{\rm pro})$ with the primal or dual objective values of $(x_{\rm def}, y_{\rm def}, z_{\rm def})$ and $(x_{\rm tight}, y_{\rm tight}, z_{\rm tight})$. 
From Tables \ref{Table: compare obj value Mosek}-\ref{Table: compare obj value SDPT3}, we can see that
\begin{equation}
\notag
b^\top y_{\rm pro} > \max \{ b^\top y_{\rm def}, b^\top y_{\rm tight} \}
\end{equation}
holds for almost all instances of the ill-posed group. 
Noting that $c - \mathcal{A}^* y_{\rm pro} = z_{\rm pro} \in \mathcal{K}$ holds for all instances, $(y_{\rm pro}, z_{\rm pro})$ can be regarded as a more accurate optimal dual solution than $(y_{\rm def}, z_{\rm def})$ and $(y_{\rm tight}, z_{\rm tight})$.
Therefore, the reason why the value of $| {\rm err_5}(x_{\rm pro}, y_{\rm pro}, z_{\rm pro}) |$ or $|{\rm err_6}(x_{\rm pro}, y_{\rm pro}, z_{\rm pro})|$ was sometimes greater than the corresponding DIMACS error values for $(x_{\rm def}, y_{\rm def}, z_{\rm def})$ and $(x_{\rm tight}, y_{\rm tight}, z_{\rm tight})$ is that Algorithm \ref{postpro alg} did not obtain a sufficient accurate primal solution $x_{\rm pro}$ such that $\langle c, x_{\rm pro} \rangle \simeq b^\top y_{\rm pro}$ for some instances. 
Here, we note that the primal problem of all instances in the ill-posed group is not expected to be strongly feasible, which might prevent the projection and rescaling methods in Algorithm \ref{postpro alg} from working stably and obtaining accurate optimal primal solutions.

\begin{table}
\caption{Comparison of the objective values of  $(x_{\rm pro}, y_{\rm pro}, z_{\rm pro})$, $(x_{\rm def}, y_{\rm def}, z_{\rm def})$ and $(x_{\rm tight}, y_{\rm tight}, z_{\rm tight})$ for the ill-posed group : Mosek}
\label{Table: compare obj value Mosek}
\begin{center}
\setlength{\tabcolsep}{5pt}
\footnotesize
\begin{tabular}{ccccc} \toprule
instance	&$b^\top y_{\rm pro} - \langle c, x_{\rm def} \rangle$	&$b^\top y_{\rm pro} - b^\top y_{\rm def}$		&$b^\top y_{\rm pro} - \langle c, x_{\rm tight} \rangle$		&$b^\top y_{\rm pro} - b^\top y_{\rm tight}$		\\	\midrule
gpp100   &7.38e-06   &7.37e-06   &7.38e-06   &7.37e-06  \\
gpp124-1   &4.88e-06   &4.86e-06   &1.24e-06   &1.23e-06  \\
gpp124-2   &5.75e-05   &5.81e-05   &5.75e-05   &5.81e-05  \\
gpp124-3   &4.76e-05   &4.94e-05   &4.76e-05   &4.94e-05  \\
gpp124-4   &5.10e-05   &5.09e-05   &5.10e-05   &5.09e-05  \\
hinf1   &6.65e-07   &6.01e-07   &1.30e-06   &1.18e-06  \\
hinf3   &3.79e-05   &3.59e-05   &3.85e-05   &3.64e-05  \\
hinf4   &6.85e-05   &6.61e-05   &1.62e-05   &1.53e-05  \\
hinf5   &8.33e+00   &8.33e+00   &8.33e+00   &8.33e+00  \\
hinf6   &1.86e-03   &1.75e-03   &7.51e-03   &7.21e-03  \\
hinf7   &7.17e+00   &7.17e+00   &7.17e+00   &7.17e+00  \\
hinf8   &4.63e-03   &4.46e-03   &4.63e-03   &4.46e-03  \\
hinf10   &7.17e-05   &4.00e-05   &3.09e-04   &2.43e-04  \\
hinf11   &2.23e-04   &1.93e-04   &1.60e-03   &1.48e-03  \\
hinf13   &3.40e-04   &2.59e-04   &3.40e-04   &2.59e-04  \\
hinf14   &5.49e-04   &5.25e-04   &5.58e-04   &5.34e-04  \\
hinf15   &8.12e-04   &6.42e-04   &2.24e-03   &2.08e-03  \\
qap5   &1.29e-06   &1.22e-06   &1.90e-09   &1.79e-09  \\
qap6   &4.38e-04   &4.36e-04   &2.68e-03   &2.61e-03  \\
qap7   &4.81e-05   &4.43e-05   &8.71e-04   &8.53e-04  \\
qap8   &5.26e-03   &5.25e-03   &4.50e-03   &4.49e-03  \\
qap9   &1.14e-03   &1.13e-03   &4.89e-03   &4.84e-03  \\
qap10   &1.56e-02   &1.55e-02   &1.56e-02   &1.55e-02  \\
\bottomrule
\end{tabular}
\end{center}
\end{table}

\begin{table}
\caption{Comparison of the objective values of  $(x_{\rm pro}, y_{\rm pro}, z_{\rm pro})$, $(x_{\rm def}, y_{\rm def}, z_{\rm def})$ and $(x_{\rm tight}, y_{\rm tight}, z_{\rm tight})$ for the ill-posed group : SDPA}
\label{Table: compare obj value SDPA}
\begin{center}
\setlength{\tabcolsep}{5pt}
\footnotesize
\begin{tabular}{ccccc} \toprule
instance	&$b^\top y_{\rm pro} - \langle c, x_{\rm def} \rangle$	&$b^\top y_{\rm pro} - b^\top y_{\rm def}$		&$b^\top y_{\rm pro} - \langle c, x_{\rm tight} \rangle$		&$b^\top y_{\rm pro} - b^\top y_{\rm tight}$		\\	\midrule
gpp100   &-1.88e-07   &2.07e-07   &-1.55e-03   &3.00e-05  \\
gpp124-1   &-1.29e-07   &5.94e-08   &-1.04e-03   &1.28e-05  \\
gpp124-2   &-4.51e-06   &1.03e-07   &-1.21e-03   &9.34e-06  \\
gpp124-3   &5.07e-05   &6.27e-05   &-2.05e-03   &1.04e-04  \\
gpp124-4   &-3.14e-05   &1.88e-06   &-4.54e-03   &1.03e-04  \\
hinf1   &1.18e-05   &1.27e-05   &-2.06e-04   &1.09e-04  \\
hinf3   &-5.02e-03   &9.78e-03   &4.19e-03   &5.67e-03  \\
hinf4   &3.46e-04   &3.46e-04   &5.96e-04   &4.22e-04  \\
hinf5   &2.05e+01   &2.00e+01   &2.06e+01   &2.01e+01  \\
hinf6   &4.65e-02   &2.41e-02   &2.86e-01   &2.40e-01  \\
hinf7   &-3.08e+00   &4.77e+00   &-3.08e+00   &4.77e+00  \\
hinf8   &1.69e-01   &1.69e-01   &1.68e-01   &1.68e-01  \\
hinf10   &1.07e-01   &5.37e-02   &9.72e+00   &6.91e+00  \\
hinf11   &1.25e-01   &6.30e-02   &3.11e+00   &2.58e+00  \\
hinf12   &2.63e+01   &3.21e+00   &4.93e+01   &4.17e+01  \\
hinf13   &4.39e+00   &2.88e+00   &4.39e+00   &2.88e+00  \\
hinf14   &-2.90e-03   &3.14e-03   &-2.03e-02   &1.84e-02  \\
hinf15   &3.63e+00   &2.64e+00   &3.63e+00   &2.64e+00  \\
qap5   &-2.24e-02   &3.36e-04   &-1.60e-03   &1.08e-04  \\
qap6   &-5.04e-02   &1.21e-02   &-1.93e-02   &5.86e-03  \\
qap7   &-1.01e-01   &2.34e-02   &-4.57e-02   &1.58e-02  \\
qap8   &-1.53e-01   &2.11e-02   &-4.95e-01   &1.62e-01  \\
qap9   &-3.95e-01   &8.01e-02   &-3.01e-01   &5.72e-02  \\
qap10   &-7.21e-01   &8.16e-02   &-6.57e-01   &7.75e-02  \\
\bottomrule
\end{tabular}
\end{center}
\end{table}

\begin{table}
\caption{Comparison of the objective values of  $(x_{\rm pro}, y_{\rm pro}, z_{\rm pro})$, $(x_{\rm def}, y_{\rm def}, z_{\rm def})$ and $(x_{\rm tight}, y_{\rm tight}, z_{\rm tight})$ for the ill-posed group : SDPT3}
\label{Table: compare obj value SDPT3}
\begin{center}
\setlength{\tabcolsep}{5pt}
\footnotesize
\begin{tabular}{ccccc} \toprule
instance	&$b^\top y_{\rm pro} - \langle c, x_{\rm def} \rangle$	&$b^\top y_{\rm pro} - b^\top y_{\rm def}$		&$b^\top y_{\rm pro} - \langle c, x_{\rm tight} \rangle$		&$b^\top y_{\rm pro} - b^\top y_{\rm tight}$		\\	\midrule
gpp100   &2.88e-06   &1.88e-06   &3.70e-06   &1.84e-06  \\
gpp124-1   &9.15e-07   &3.71e-07   &9.20e-07   &3.71e-07  \\
gpp124-2   &1.71e-06   &1.11e-06   &2.25e-06   &1.09e-06  \\
gpp124-3   &4.48e-06   &2.94e-06   &5.65e-06   &2.88e-06  \\
gpp124-4   &2.72e-05   &1.20e-05   &2.78e-05   &1.20e-05  \\
hinf1   &1.03e-04   &5.19e-05   &1.03e-04   &5.19e-05  \\
hinf3   &2.71e-02   &1.36e-02   &2.71e-02   &1.36e-02  \\
hinf4   &1.92e-03   &9.68e-04   &1.92e-03   &9.68e-04  \\
hinf5   &4.81e+00   &4.47e+00   &4.81e+00   &4.47e+00  \\
hinf6   &1.60e-02   &1.73e-02   &2.89e-02   &1.46e-02  \\
hinf7   &1.42e-02   &7.48e-03   &1.42e-02   &7.48e-03  \\
hinf8   &4.06e-02   &2.03e-02   &4.06e-02   &2.03e-02  \\
hinf10   &6.14e-02   &5.28e-02   &1.35e-01   &6.73e-02  \\
hinf11   &7.37e-02   &3.68e-02   &7.37e-02   &3.68e-02  \\
hinf12   &2.97e-05   &1.45e-05   &-8.09e-06   &-1.29e-05  \\
hinf13   &1.03e-02   &4.33e-03   &6.81e-03   &2.03e-03  \\
hinf14   &6.37e-05   &8.15e-05   &5.90e-05   &3.54e-05  \\
hinf15   &1.80e-02   &2.76e-02   &1.80e-02   &2.76e-02  \\
qap5   &-3.08e-07   &8.53e-08   &-4.55e-09   &1.58e-09  \\
qap6   &4.52e-02   &2.22e-02   &4.52e-02   &2.22e-02  \\
qap7   &3.15e-02   &1.55e-02   &3.15e-02   &1.55e-02  \\
qap8   &1.13e-01   &5.54e-02   &1.13e-01   &5.54e-02  \\
qap9   &2.19e-02   &1.06e-02   &2.19e-02   &1.06e-02  \\
qap10   &5.11e-02   &1.69e-02   &5.11e-02   &1.69e-02  \\
\bottomrule
\end{tabular}
\end{center}
\end{table}

From the results of the solvers for the ill-posed group, we can observe the same thing as for the well-posed group. 
When using SDPA and Mosek with the tight setting, they sometimes returned inaccurate primal solutions and inaccurate dual solutions, respectively. 
In addition, SDPT3 with the tight setting returned more accurate optimal solutions $(x_{\rm tight}, y_{\rm tight}, z_{\rm tight})$ than $(x_{\rm def}, y_{\rm def}, z_{\rm def})$ for almost all instances.
We note that SDPT3 obtained the same results when $\epsilon_{\rm primal} = \epsilon_{\rm dual} = \epsilon_{\rm gap}=$ 1e-13 as when $\epsilon_{\rm primal} = \epsilon_{\rm dual} = \epsilon_{\rm gap}=$ 1e-12.

Let us compare the results for the well-posed and ill-posed groups. 
Comparing Figures \ref{fig: Mosek-well}-\ref{fig: SDPT3-well} and \ref{fig: Mosek-ill}-\ref{fig: SDPT3-ill}, we can observe the following: 
\begin{itemize}
\item Algorithm \ref{postpro alg} did not consistently obtain an accurate optimal solution for the ill-posed group compared to the results for the well-posed group. 
\item Algorithm \ref{postpro alg} returned $(x_{\rm pro}, y_{\rm pro}, z_{\rm pro})$ such that $|{\rm err_5} (x_{\rm pro}, y_{\rm pro}, z_{\rm pro})| \simeq |{\rm err_6} (x_{\rm pro}, y_{\rm pro}, z_{\rm pro})|$ for almost all instances of the well-posed group, but such a relation did not hold for the ill-posed group. 
\end{itemize}

The first observation is evident in Figures \ref{fig: Mosek-well}-\ref{fig: SDPT3-ill}, (c). 
To clarify the second observation, see Table \ref{Table: compare err5 and err6}. 
Table \ref{Table: compare err5 and err6} summarizes the average value of $| |{\rm err_5} (x_{\rm pro}, y_{\rm pro}, z_{\rm pro})| - |{\rm err_6} (x_{\rm pro}, y_{\rm pro}, z_{\rm pro})| |$ for each group. 
Table \ref{Table: compare err5 and err6} shows that for all combinations of Algorithm \ref{postpro alg} and the solvers, the average value of $| |{\rm err_5} (x_{\rm pro}, y_{\rm pro}, z_{\rm pro})| - |{\rm err_6} (x_{\rm pro}, y_{\rm pro}, z_{\rm pro})| |$ for the ill-posed group is greater than for the well-posed group.
From this table, we can see that the relation $|{\rm err_5} (x_{\rm pro}, y_{\rm pro}, z_{\rm pro})| \simeq |{\rm err_6} (x_{\rm pro}, y_{\rm pro}, z_{\rm pro})|$ holds for the well-posed group but not for the ill-posed group. 

\begin{table}
\caption{Comparison of the average values of $\left| |{\rm err_5} (x_{\rm pro}, y_{\rm pro}, z_{\rm pro})| - |{\rm err_6} (x_{\rm pro}, y_{\rm pro}, z_{\rm pro})| \right|$}
\label{Table: compare err5 and err6}
\begin{center}
\setlength{\tabcolsep}{5pt}
\footnotesize
\begin{tabular}{ccc} \toprule
Method					&Well-posed group	&Ill-posed group	\\	\midrule
Mosek + Algorithm \ref{postpro alg}	&5.34e-11	&3.45e-03	\\
SDPA + Algorithm \ref{postpro alg}	&1.93e-11	&3.75e-03	\\
SDPT3 + Algorithm \ref{postpro alg}	&2.51e-12	&3.51e-04	\\
\bottomrule
\end{tabular}
\end{center}
\end{table}

There are two reasons for this.
The first one is Algorithm \ref{postpro alg} did not obtain accurate optimal primal solutions for some instances of the ill-posed group. 
As mentioned above, the projection and rescaling methods might not work stably due to the feasibility status of the primal problem for the ill-posed group, which results in the instability of Algorithm \ref{postpro alg} in obtaining accurate optimal primal solutions. 
Therefore, the relation $|{\rm err_5} (x_{\rm pro}, y_{\rm pro}, z_{\rm pro})| \simeq |{\rm err_6} (x_{\rm pro}, y_{\rm pro}, z_{\rm pro})|$ did not hold for the ill-posed group compared to the results for the well-posed group. 
The second one is the existence of a reducing direction for (P). 
Recall that $(f, -\mathcal{A}^* f) \in \mathbb{R}^m \times \mathbb{E}$ is called a reducing direction for (P) if $(f, -\mathcal{A}^* f)$ satisfies $b^\top f = 0$ and $-\mathcal{A}^* f \in \mathcal{K} \setminus \{0\}$. 
Thus, if (D) has an optimal solution and a reducing direction for (P) exists, then the optimal solution set of (D) is unbounded. 
Suppose that optimal dual solutions and reducing directions for (P) exist for all instances of the ill-posed group.
Then, for any optimal dual solution $(y, z)$ and any positive value $k > 0$, there exists a dual optima solution $(y_{opt}, z_{opt})$ and a reducing direction $(f, -\mathcal{A}^* f)$ such that $y = y_{opt} + f$ and $z_{opt} + k \mathcal{A}^* f \notin \mathcal{K}$, and we have
\begin{align*}
{\rm err_5} (x,y,z)
&= \cfrac{\langle c, x \rangle - b^\top (y_{opt} + f)}{1+|\langle c,x \rangle| + |b^\top y|} \\
&= \cfrac{\langle c, x \rangle - b^\top y_{opt}}{1+|\langle c,x \rangle| + |b^\top y|}, \ \ \ \mbox{(since $b^\top f = 0$)}
\end{align*}
and
\begin{align*}
{\rm err_6} (x,y,z) 
&= \cfrac{\langle x,  c - \mathcal{A}^* y_{opt} - \mathcal{A}^* f \rangle}{1+|\langle c,x \rangle| + |b^\top y|} \\
&= \cfrac{\langle x,  c \rangle - (\mathcal{A}x)^\top y_{opt} + \langle x, - \mathcal{A}^* f \rangle}{1+|\langle c,x \rangle| + |b^\top y|}.
\end{align*}
Therefore, even if $\|\mathcal{A} x_{\rm pro} - b\|_2$ is sufficiently small, the value of $\langle x_{pro}, - \mathcal{A}^* f \rangle$ might not be negligibly small as long as $y_{\rm pro} \simeq y_{opt} + f$ holds for some reducing direction $(f, -\mathcal{A}^* f)$ such that the value of $\|-\mathcal{A}^* f\|$ is very large.

Table \ref{Table: compare normalized dual var norm} summarizes the average values of $f(z) = \frac{\|z\|}{1+ \displaystyle \max_{i}| c_i | + \displaystyle \max_{i, j}| A_{ij} |}$ for $z_{\rm def}$ and $z_{\rm pro}$ by the well-posed and the ill-posed groups, where $A \in \mathbb{R}^{m \times d}$ is a matrix representation of the linear operator $\mathcal{A}$.
Since Mosek and SDPA with the tight settings did not work stably in this experiment, the average values of $f(z_{\rm tight})$ for each solver are omitted. 
In Table \ref{Table: compare normalized dual var norm}, the first, third, and fifth rows show the average values of $f(z)$ for $z_{\rm def}$ obtained from Mosek, SDPA, and SDPT3, respectively. 
The second, fourth, and sixth rows show the average values of $f(z)$ for $z_{\rm pro}$ obtained from Algorithm \ref{postpro alg} using $(x_{\rm def}, y_{\rm def}, z_{\rm def})$ returned from Mosek, SDPA, and SDPT3, respectively. 
We note that Algorithm \ref{postpro alg} returned $z_{\rm pro}$ such that $\| z_{\rm pro} \| >> \| z_{\rm def} \|$ for hinf2. 
Since hinf2 is included in the well-posed group, in Table \ref{Table: compare normalized dual var norm}, the values in rows 2, 4, and 6 of the ``Well-posed group" column are significantly different from the values in rows 1, 3 and 5 of the same column, respectively. 
The average values of $f(z_{\rm def})$ and $f(z_{\rm pro})$ for the well-posed group, excluding hinf2, are summarized in the second column. 
From Table \ref{Table: compare normalized dual var norm}, we can see that the average values of $f(z_{\rm def})$ and $f(z_{\rm pro})$ in the third column are greater than in the second column. 
Thus, the dual solutions $z_{\rm def}$ and $z_{\rm pro}$ obtained by the solvers and Algorithm \ref{postpro alg} for the ill-posed group are likely to include a reducing direction $(f, -\mathcal{A}^* f)$ such that the value of $\|-\mathcal{A}^* f\|$ is very large.
Furthermore, we can expect that the dual optimal solution set of hinf2 includes solutions $(y_{min}, z_{min})$ and $(y_{max}, z_{max})$ such that $\|z_{min}\| << \|z_{max}\|$.

\begin{table}
\caption{Comparison of the average values of $f(z) = \frac{\|z\|}{1+ \displaystyle \max_{i}| c_i | + \displaystyle \max_{i, j}| A_{ij} |}$ for $z_{\rm def}$ and $z_{\rm pro}$}
\label{Table: compare normalized dual var norm}
\begin{center}
\setlength{\tabcolsep}{5pt}
\footnotesize
\begin{tabular}{cccc} \toprule
Method					&Well-posed group	&Well-posed group excluding hinf2	&Ill-posed group	\\	\midrule
Mosek					&1.04e+03	&6.85e+02	&6.75e+05	\\
Mosek + Algorithm \ref{postpro alg}	&3.32e+03	&6.85e+02	&2.63e+10	\\
SDPA						&1.04e+03	&6.85e+02	&8.81e+04	\\
SDPA + Algorithm \ref{postpro alg}	&3.32e+03	&6.85e+02	&2.74e+11	\\
SDPT3					&6.70e+02	&6.85e+02	&4.66e+08	\\
SDPT3 + Algorithm \ref{postpro alg}	&3.33e+03	&6.85e+02	&7.87e+09	\\
\bottomrule
\end{tabular}
\end{center}
\end{table}

\begin{figure}[htbp]
    \begin{tabular}{cc}
      \begin{minipage}[t]{0.475\hsize}
        \centering
        \includegraphics[keepaspectratio, scale=0.425]{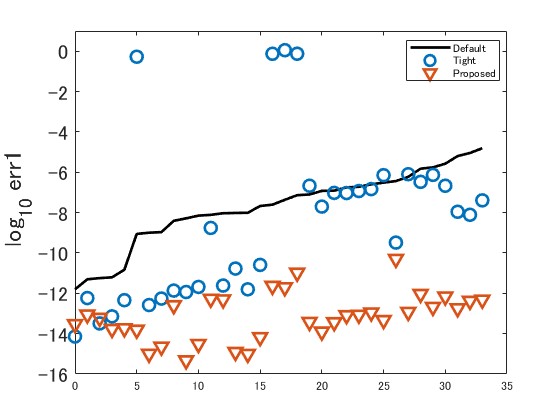}
        \subcaption{The values of $\log_{10} {\rm err_1}(x,y,z)$}
        \label{fig: Mosek-well-err1}
      \end{minipage} &
      \begin{minipage}[t]{0.475\hsize}
        \centering
        \includegraphics[keepaspectratio, scale=0.425]{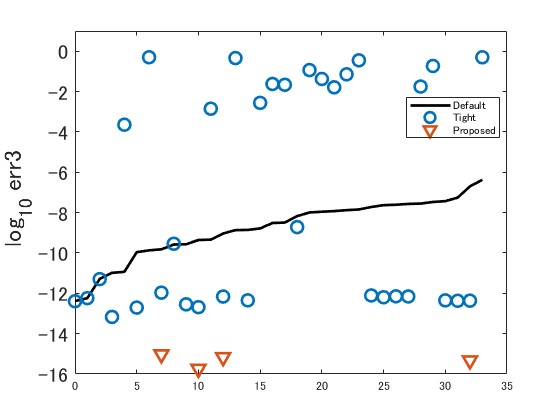}
	\subcaption{The values of $\log_{10} {\rm err_3}(x,y,z)$}
        \label{fig: Mosek-well-err3}
      \end{minipage} \\

      \begin{minipage}[t]{0.475\hsize}
        \centering
        \includegraphics[keepaspectratio, scale=0.425]{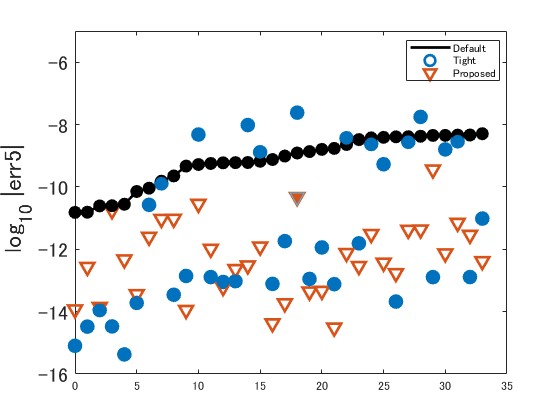}
	\subcaption{The values of $\log_{10} |{\rm err_5}(x,y,z)|$}
        \label{fig: Mosek-well-err5}
      \end{minipage} &
   
      \begin{minipage}[t]{0.475\hsize}
        \centering
        \includegraphics[keepaspectratio, scale=0.425]{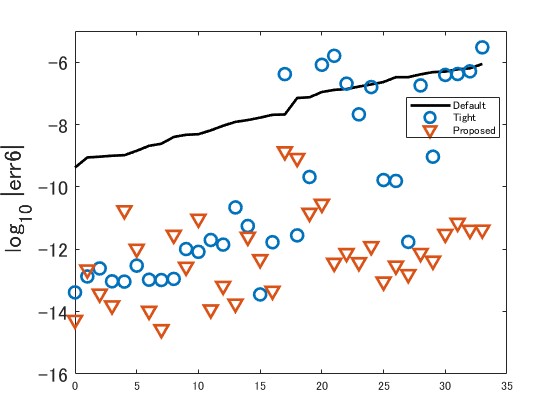}
	\subcaption{The values of $\log_{10} |{\rm err_6}(x,y,z)|$}
        \label{fig: Mosek-well-err6}
      \end{minipage}
   \end{tabular}
\caption{Numreical results for the well-posed group with Mosek}
\label{fig: Mosek-well}
\end{figure}

\begin{figure}[htbp]
    \begin{tabular}{cc}
      \begin{minipage}[t]{0.475\hsize}
        \centering
        \includegraphics[keepaspectratio, scale=0.425]{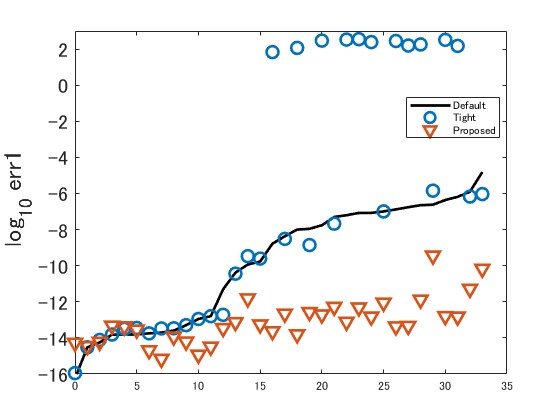}
        \subcaption{The values of $\log_{10} {\rm err_1}(x,y,z)$}
        \label{fig: SDPA-well-err1}
      \end{minipage} &
      \begin{minipage}[t]{0.475\hsize}
        \centering
        \includegraphics[keepaspectratio, scale=0.425]{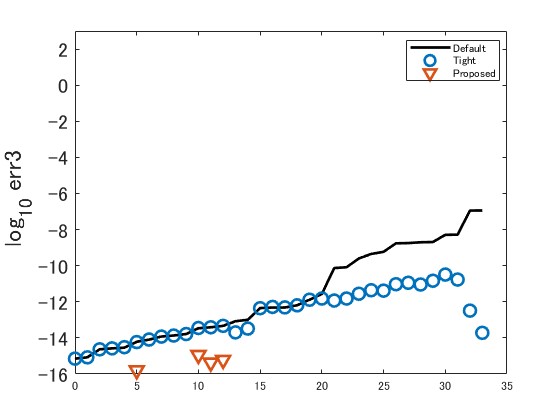}
	\subcaption{The values of $\log_{10} {\rm err_3}(x,y,z)$}
        \label{fig: SDPA-well-err3}
      \end{minipage} \\

      \begin{minipage}[t]{0.475\hsize}
        \centering
        \includegraphics[keepaspectratio, scale=0.425]{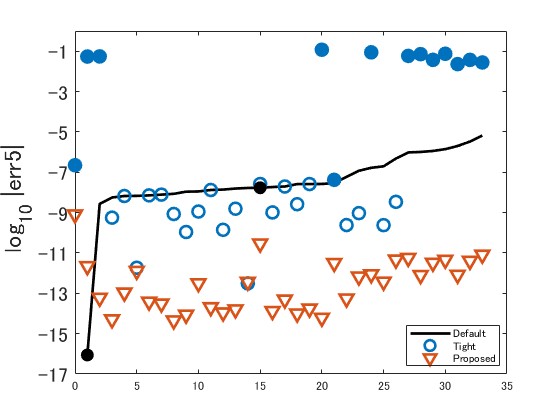}
	\subcaption{The values of $\log_{10} |{\rm err_5}(x,y,z)|$}
        \label{fig: SDPA-well-err5}
      \end{minipage} &
   
      \begin{minipage}[t]{0.475\hsize}
        \centering
        \includegraphics[keepaspectratio, scale=0.425]{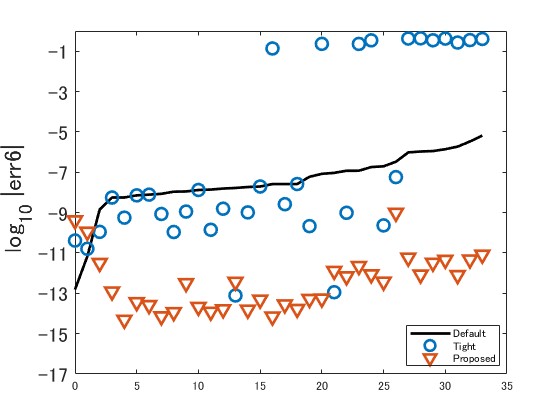}
	\subcaption{The values of $\log_{10} |{\rm err_6}(x,y,z)|$}
        \label{fig: SDPA-well-err6}
      \end{minipage}
   \end{tabular}
\caption{Numreical results for the well-posed group with SDPA}
\label{fig: SDPA-well}
\end{figure}

\begin{figure}[htbp]
    \begin{tabular}{cc}
      \begin{minipage}[t]{0.475\hsize}
        \centering
        \includegraphics[keepaspectratio, scale=0.425]{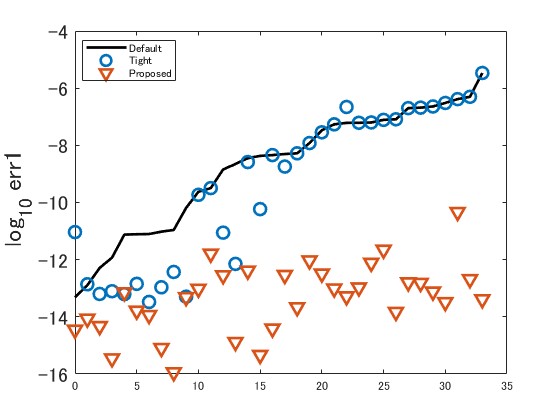}
        \subcaption{The values of $\log_{10} {\rm err_1}(x,y,z)$}
        \label{fig: SDPT3-well-err1}
      \end{minipage} &
      \begin{minipage}[t]{0.475\hsize}
        \centering
        \includegraphics[keepaspectratio, scale=0.425]{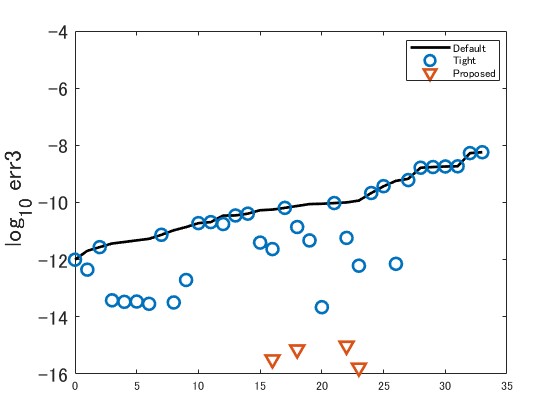}
	\subcaption{The values of $\log_{10} {\rm err_3}(x,y,z)$}
        \label{fig: SDPT3-well-err3}
      \end{minipage} \\

      \begin{minipage}[t]{0.475\hsize}
        \centering
        \includegraphics[keepaspectratio, scale=0.425]{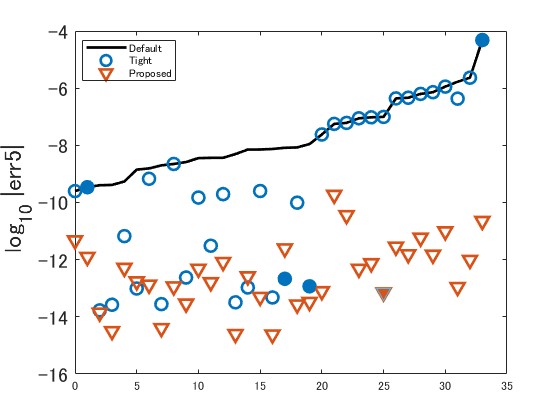}
	\subcaption{The values of $\log_{10} |{\rm err_5}(x,y,z)|$}
        \label{fig: SDPT3-well-err5}
      \end{minipage} &
   
      \begin{minipage}[t]{0.475\hsize}
        \centering
        \includegraphics[keepaspectratio, scale=0.425]{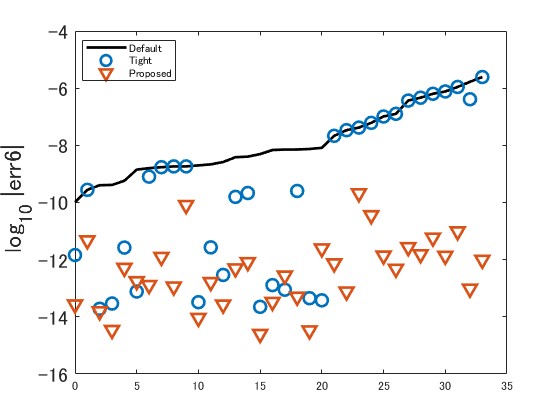}
	\subcaption{The values of $\log_{10} |{\rm err_6}(x,y,z)|$}
        \label{fig: SDPT3-well-err6}
      \end{minipage}
   \end{tabular}
\caption{Numreical results for the well-posed group with SDPT3}
\label{fig: SDPT3-well}
\end{figure}

\begin{figure}[htbp]
    \begin{tabular}{cc}
      \begin{minipage}[t]{0.475\hsize}
        \centering
        \includegraphics[keepaspectratio, scale=0.425]{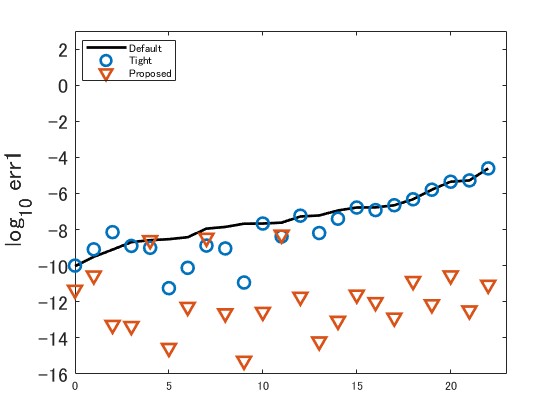}
        \subcaption{The values of $\log_{10} {\rm err_1}(x,y,z)$}
        \label{fig: Mosek-ill-err1}
      \end{minipage} &
      \begin{minipage}[t]{0.475\hsize}
        \centering
        \includegraphics[keepaspectratio, scale=0.425]{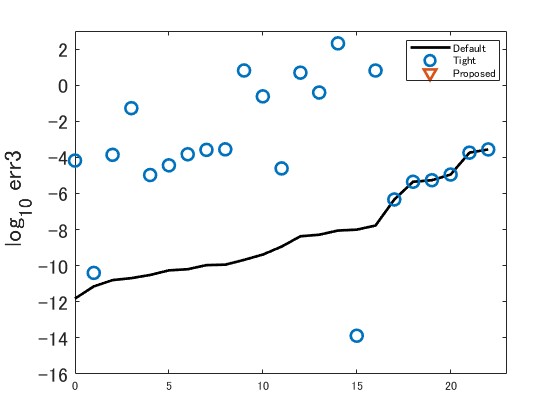}
	\subcaption{The values of $\log_{10} {\rm err_3}(x,y,z)$}
        \label{fig: Mosek-ill-err3}
      \end{minipage} \\

      \begin{minipage}[t]{0.475\hsize}
        \centering
        \includegraphics[keepaspectratio, scale=0.425]{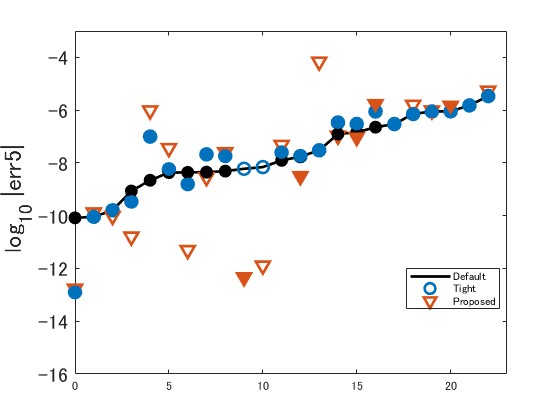}
	\subcaption{The values of $\log_{10} |{\rm err_5}(x,y,z)|$}
        \label{fig: Mosek-ill-err5}
      \end{minipage} &
   
      \begin{minipage}[t]{0.475\hsize}
        \centering
        \includegraphics[keepaspectratio, scale=0.425]{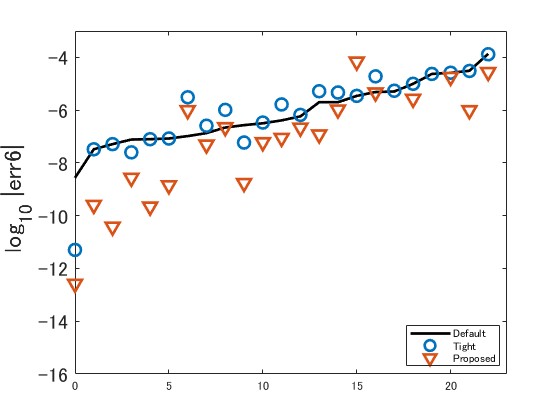}
	\subcaption{The values of $\log_{10} |{\rm err_6}(x,y,z)|$}
        \label{fig: Mosek-ill-err6}
      \end{minipage}
   \end{tabular}
\caption{Numreical results for the ill-posed group with Mosek}
\label{fig: Mosek-ill}
\end{figure}

\begin{figure}[htbp]
    \begin{tabular}{cc}
      \begin{minipage}[t]{0.475\hsize}
        \centering
        \includegraphics[keepaspectratio, scale=0.425]{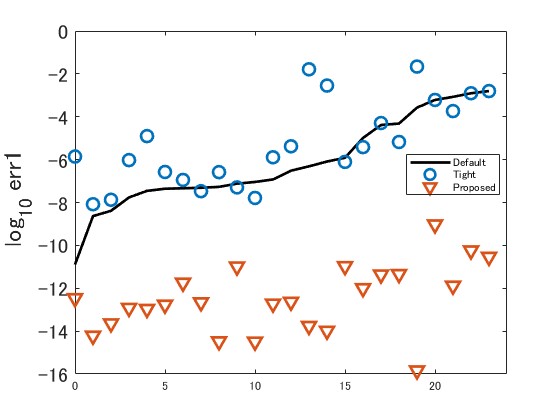}
        \subcaption{The values of $\log_{10} {\rm err_1}(x,y,z)$}
        \label{fig: SDPA-ill-err1}
      \end{minipage} &
      \begin{minipage}[t]{0.475\hsize}
        \centering
        \includegraphics[keepaspectratio, scale=0.425]{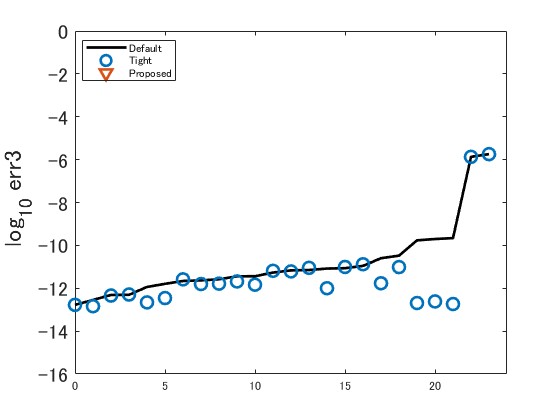}
	\subcaption{The values of $\log_{10} {\rm err_3}(x,y,z)$}
        \label{fig: SDPA-ill-err3}
      \end{minipage} \\

      \begin{minipage}[t]{0.475\hsize}
        \centering
        \includegraphics[keepaspectratio, scale=0.425]{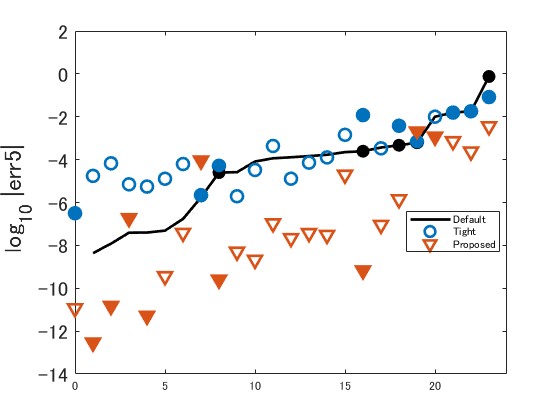}
	\subcaption{The values of $\log_{10} |{\rm err_5}(x,y,z)|$}
        \label{fig: SDPA-ill-err5}
      \end{minipage} &
   
      \begin{minipage}[t]{0.475\hsize}
        \centering
        \includegraphics[keepaspectratio, scale=0.425]{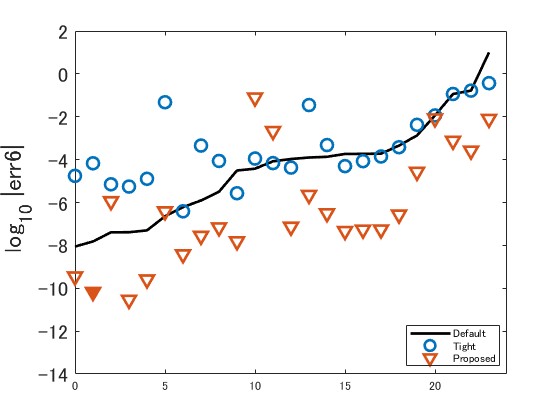}
	\subcaption{The values of $\log_{10} |{\rm err_6}(x,y,z)|$}
        \label{fig: SDPA-ill-err6}
      \end{minipage}
   \end{tabular}
\caption{Numreical results for the ill-posed group with SDPA}
\label{fig: SDPA-ill}
\end{figure}

\begin{figure}[htbp]
    \begin{tabular}{cc}
      \begin{minipage}[t]{0.475\hsize}
        \centering
        \includegraphics[keepaspectratio, scale=0.425]{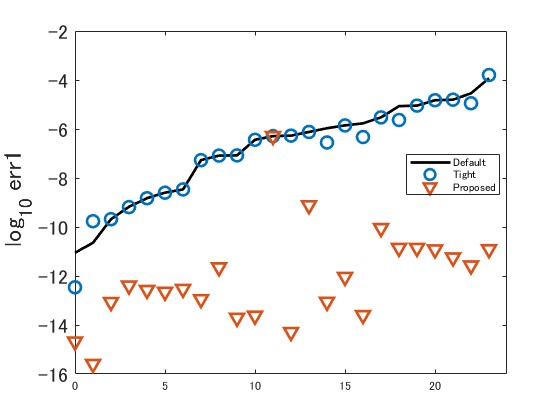}
        \subcaption{The values of $\log_{10} {\rm err_1}(x,y,z)$}
        \label{fig: SDPT3-ill-err1}
      \end{minipage} &
      \begin{minipage}[t]{0.475\hsize}
        \centering
        \includegraphics[keepaspectratio, scale=0.425]{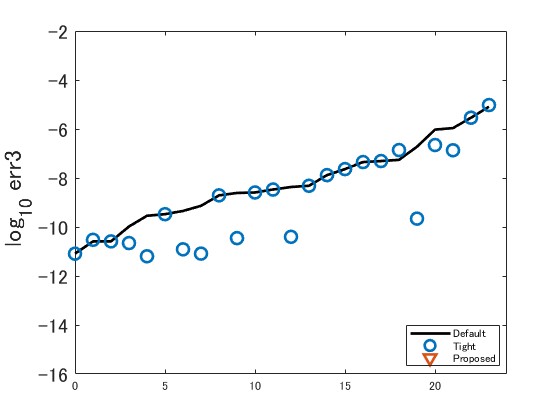}
	\subcaption{The values of $\log_{10} {\rm err_3}(x,y,z)$}
        \label{fig: SDPT3-ill-err3}
      \end{minipage} \\

      \begin{minipage}[t]{0.475\hsize}
        \centering
        \includegraphics[keepaspectratio, scale=0.425]{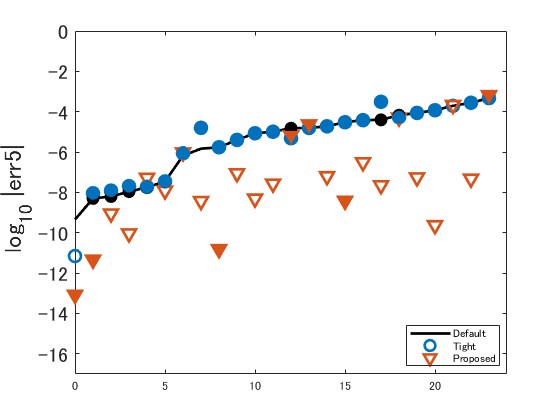}
	\subcaption{The values of $\log_{10} |{\rm err_5}(x,y,z)|$}
        \label{fig: SDPT3-ill-err5}
      \end{minipage} &
   
      \begin{minipage}[t]{0.475\hsize}
        \centering
        \includegraphics[keepaspectratio, scale=0.425]{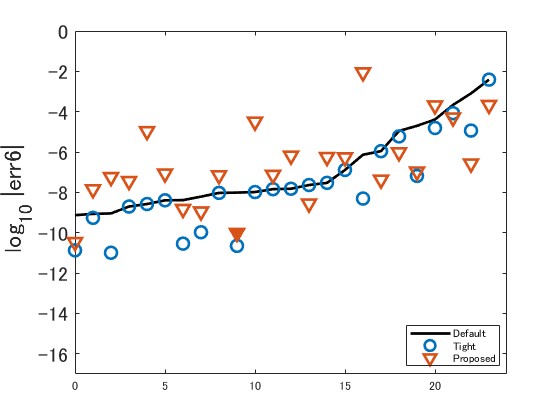}
	\subcaption{The values of $\log_{10} |{\rm err_6}(x,y,z)|$}
        \label{fig: SDPT3-ill-err6}
      \end{minipage}
   \end{tabular}
\caption{Numreical results for the ill-posed group with SDPT3}
\label{fig: SDPT3-ill}
\end{figure}


Finally, let us compare the execution time of Algorithm \ref{postpro alg} for each instance. 
Figures \ref{fig: TimeSize-Mosek}, \ref{fig: TimeSize-SDPA} and \ref{fig: TimeSize-SDPT3} show the execution time of Algorithm \ref{postpro alg} using $(x_{\rm def}, y_{\rm def}, z_{\rm def})$ returned from Mosek, SDPA and SDPT3, respectively.
In these figures, the horizontal axis shows instances sorted in ascending order of problem size.
In this study, we defined the problem size as $m \times d$, where $m$ is the number of constraints and $d$ is the dimension of the Euclidean space $\mathbb{E}$ corresponding to $\mathcal{K}$. 
The line graph shows the execution time of Algorithm \ref{postpro alg}, and the bar graph shows the problem size.
From Figures \ref{fig: TimeSize-Mosek}-\ref{fig: TimeSize-SDPT3}, we can observe the following: 
\begin{itemize}
\item As the problem size increases, the execution time of Algorithm \ref{postpro alg} tends to increase.
\item There was a variation in the execution time of Algorithm \ref{postpro alg}, even for instances of similar problem size.
\end{itemize}
 
The above results were observed due to the structure of the projective rescaling method. 
The computational cost of the projection and rescaling algorithm proposed in \cite{Kanoh2023} is
\begin{equation}
\label{comp cost}
\mathcal{O} \left( -\frac{r}{\log \xi} \log \left( \frac{1}{\varepsilon}\right) \left( m^3+m^2d +  \frac{1}{\xi^2} p^2 r_{\max}^2  \left( \max \left( C^{\rm sd} , md \right) \right) \right) \right)
\end{equation}
where $C^{\rm sd}$ is the computational cost of the spectral decomposition.
In (\ref{comp cost}), the maximum number of iterations of the main algorithm, the computational cost of the projection matrix, the maximum number of iterations of the basic procedure and the computational cost per iteration of the basic procedure corresponds to $-\frac{r}{\log \xi} \log \left( \frac{1}{\varepsilon}\right)$, $m^3+m^2d$, $\frac{1}{\xi^2} p^2 r_{\max}^2$ and $\max \left( C^{\rm sd}, md \right)$, respectively.
Since the increase in problem size will increase the time required to compute the projection matrix, the execution time of Algorithm \ref{postpro alg} in Figures \ref{fig: TimeSize-Mosek}-\ref{fig: TimeSize-SDPT3} tends to increase monotonically. 
Here, we note that the main algorithm and the basic procedure can be terminated in less than their maximum number of iterations.
The main algorithm (and the basic procedure) might terminate in a significantly different number of iterations for problems with similar problem sizes.
Therefore, there is expected to be some variation in the execution time of Algorithm \ref{postpro alg} for problems with similar problem sizes.

\begin{figure}[H]
\begin{center}
\includegraphics[keepaspectratio, scale=0.4]{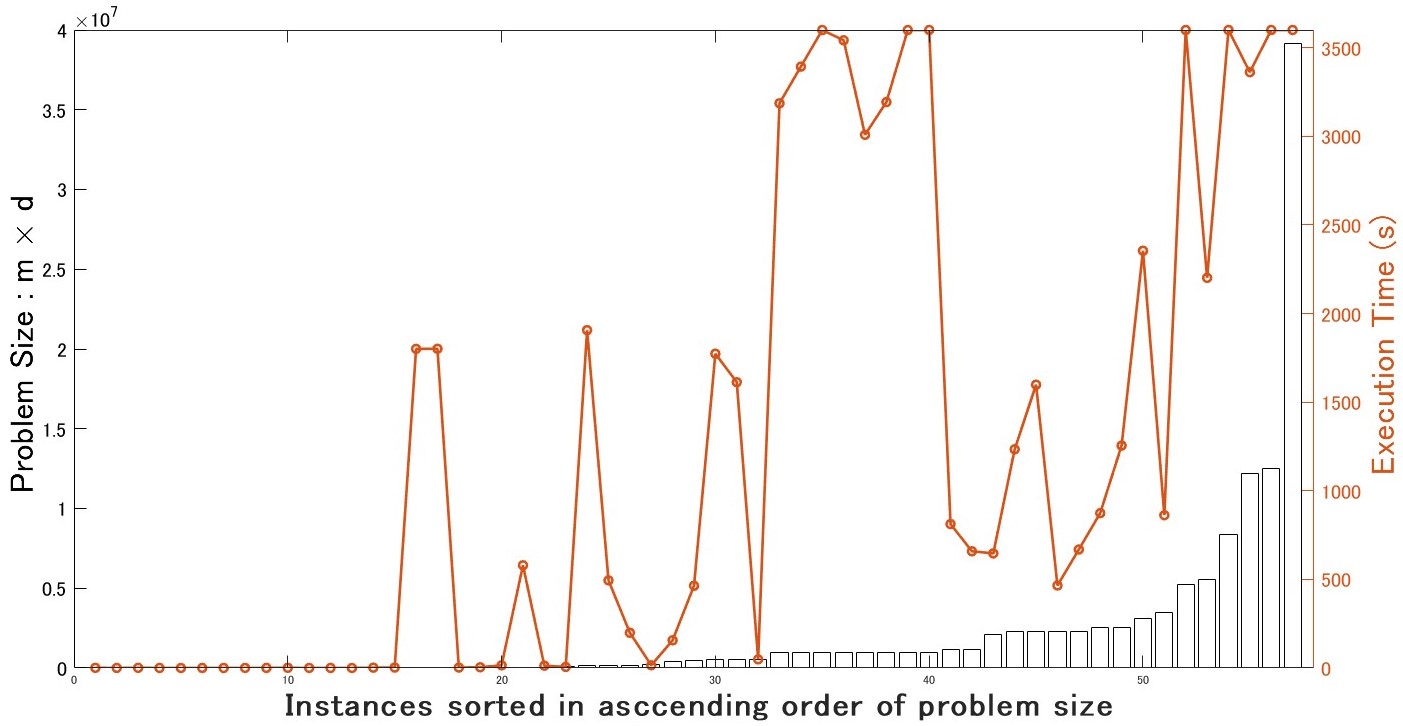}
\end{center}
\caption{Execution time of Algorithm \ref{postpro alg} using $(x_{\rm def}, y_{\rm def}, z_{\rm def})$ returned from Mosek}
\label{fig: TimeSize-Mosek}
\end{figure}

\begin{figure}[H]
\begin{center}
\includegraphics[keepaspectratio, scale=0.4]{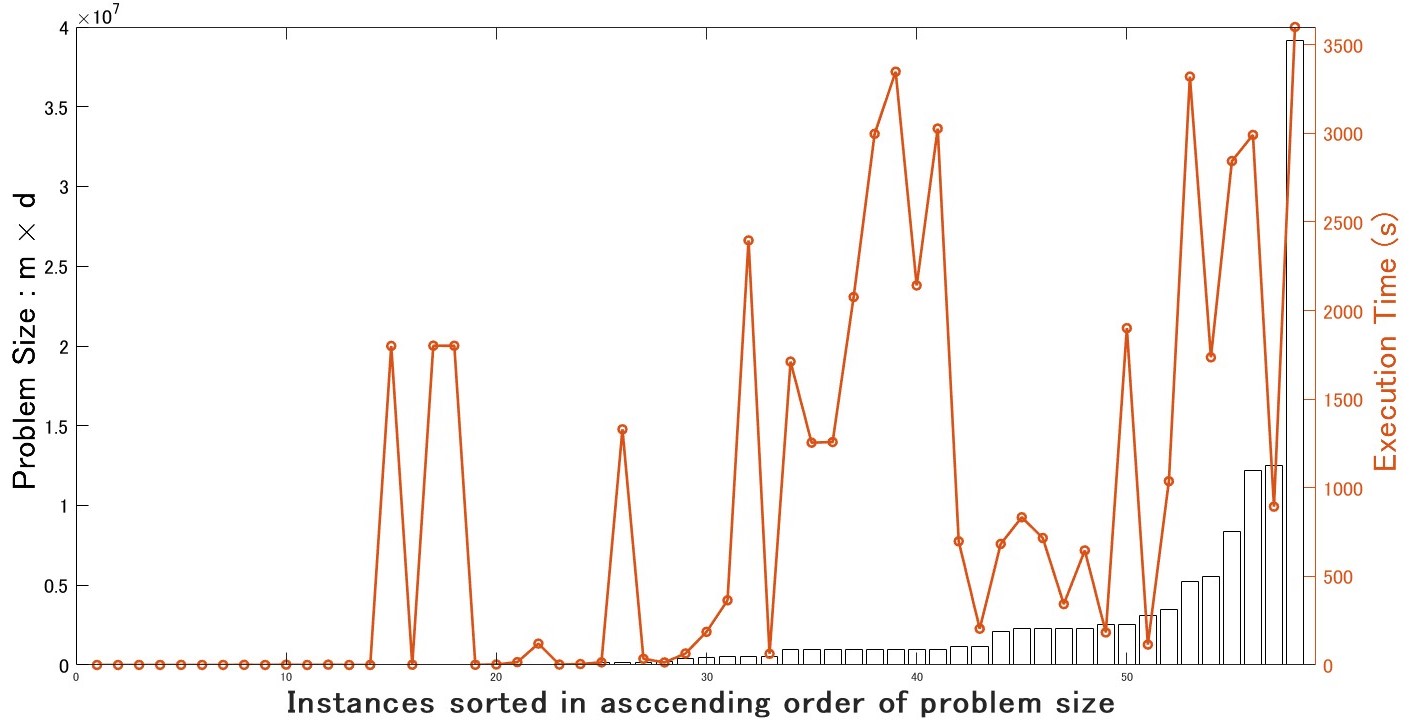}
\end{center}
\caption{Execution time of Algorithm \ref{postpro alg} using $(x_{\rm def}, y_{\rm def}, z_{\rm def})$ returned from SDPA}
\label{fig: TimeSize-SDPA}
\end{figure}

\begin{figure}[H]
\begin{center}
\includegraphics[keepaspectratio, scale=0.4]{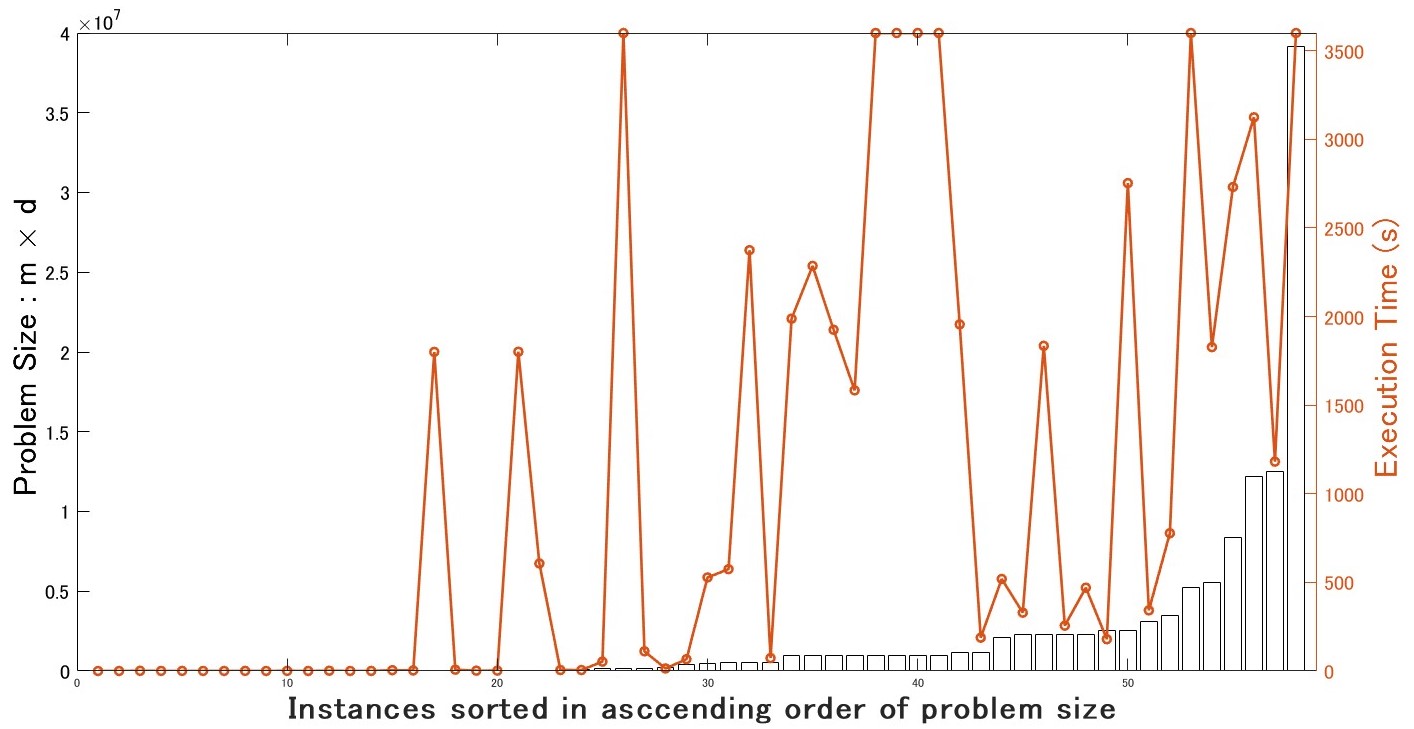}
\end{center}
\caption{Execution time of Algorithm \ref{postpro alg} using $(x_{\rm def}, y_{\rm def}, z_{\rm def})$ returned from SDPT3}
\label{fig: TimeSize-SDPT3}
\end{figure}

\subsection{Experiments to determine feasibility status}
\label{sec: numerical results check status}
The numerical results in the previous section showed that Algorithm \ref{postpro alg} obtained more accurate approximate optimal solutions for the well-posed group than the solvers.
On the other hand, the numerical results for the ill-posed group showed that Algorithm \ref{postpro alg} failed to detect that the primal problem is in weak status.
Algorithm \ref{postpro alg}, theoretically, can obtain a reducing direction for (P) (or (D)) if (P) (or (D)) is in weak status, but Algorithm \ref{postpro alg} returned not a reducing direction for (P) but an approximate interior feasible solution to (P) for the instances of the ill-posed group. 
Detecting the feasibility status of (P) and (D) and finding the reducing directions for (P) and (D) can be replaced by solving a certain SDP whose primal and dual problems are strongly feasible, \cite{Lourenco2021}.
Thus, using the results of \cite{Lourenco2021}, we tested whether Algorithm \ref{postpro alg} can detect the feasibility status of SDP more accurately than the solvers.
We conducted this experiment with the ill-posed instances of SDPLIB and the weakly infeasible cases \cite{Liu2018}.
These instances have the primal problems in weak status.

\subsubsection{Preparation to explain the flow of the experiment}
\label{sec: check status pre}
Before going to the outline of our experiments, we derive the $(P^D_{\mathcal{K}})$ and $(D^D_{\mathcal{K}})$, which is closely related to the feasibility status of (P), by using the following lemma.
\begin{lemma}[Lemma 3.4 in \cite{Lourenco2021}]
\label{lemma: Lourenco2021}
For (P) and (D), consider the following pair of primal and dual problems.
\begin{equation}
\begin{array}{lll}
(P^D_{\mathcal{K}})	&{\rm inf}_{x,t,w}	&t	\\
			&{\rm sub.to}	&-\langle c , x-t e \rangle + t - w = 0 \\
			&			&\langle ex \rangle + w = 1 \\
			&			&\mathcal{A}x - t \mathcal{A}e = 0 \\
			&			&(x,t,w) \in \mathcal{K} \times \mathbb{R}_+ \times \mathbb{R}_+
\end{array}
\notag
\end{equation}

\begin{equation}
\begin{array}{lll}
(D^D_{\mathcal{K}})	&{\rm sup}_{y_1,y_2,y_3}	&y_2	\\
			&{\rm sub.to}		&c y_1 - e y_2 - \mathcal{A}^* y_3 \in \mathcal{K} \\
			&				&1 - y_1 (1+\langle c,e \rangle ) + \langle e, \mathcal{A}^* y_3 \rangle \geq 0 \\
			&				&y_1 - y_2 \geq 0 
\end{array}
\notag
\end{equation}

The following properties hold.
\begin{enumerate}
\item $(P^D_{\mathcal{K}})$ and $(D^D_{\mathcal{K}})$ are strongly feasible. \\

Let $(x^*,t^*,w^*)$ be an optimal solution to $(P^D_{\mathcal{K}})$ and $(y_1^*, y_2^*, y_3^*)$ be an optimal solution to $(D^D_{\mathcal{K}})$.

\item The optimal value is zero if and only if (D) is not strongly feasible.
In this case, one of the two alternatives below must hold:
\begin{enumerate}
\item $x^*$ is an improving ray of (P), or
\item $x^*$ is a reducing direction for (D).
\end{enumerate}
\item The optimal value is positive if and only if (D) is strongly feasible.
\end{enumerate}
\end{lemma}

\begin{proof}
See Lemma 3.4 in \cite{Lourenco2021}.
\end{proof}

Note that \cite{Lourenco2021} showed that Lemma \ref{lemma: Lourenco2021} holds not only for the identity element $e$ of $\mathcal{K}$ but also for any interior point of $\mathcal{K}$.
Similar to Lemma \ref{lemma: Lourenco2021}, Proposition \ref{pro: check status p}, which is not in \cite{Lourenco2021}, can be easily proved.

\begin{proposition}
\label{pro: check status p}
For (P) and (D), consider the following pair of primal and dual problems.
\begin{equation}
\begin{array}{lll}
(P^P_{\mathcal{K}})&{\rm inf}_{y,z,t,w}	&t	\\
			&{\rm sub.to}	&b^\top y + t - w = 0 \\
			&			&\langle e , z \rangle + w = 1 \\
			&			&-\mathcal{A}^*y - z + t e = 0 \\
			&			&(y,s,t,w) \in \mathbb{R}^m \times \mathcal{K} \times \mathbb{R}_+ \times \mathbb{R}_+ \\
\\
(D^P_{\mathcal{K}})&{\rm sup}_{x_1, x_2, x_3}	&x_2	\\
			&{\rm sub.to}		&\mathcal{A} x_3 - x_1 b = 0 \\
			&				&1 - x_1 - \langle e, x_3 \rangle \geq 0 \\
			&				&x_1 - x_2 \geq 0 \\
			&				&x_3 - x_2 e \in \mathcal{K}
\end{array}
\notag
\end{equation}
The following properties hold.
\begin{enumerate}
\item Both $(P^P_{\mathcal{K}})$ and $(D^P_{\mathcal{K}})$ are strongly feasible.
\item Let $(y^*,z^*,t^*,w^*)$ be a primal optimal solution.
The optimal value is zero if and only if (P) is not strongly feasible.
Moreover, if the optimal value is zero, $(y^*,z^*)$ is an improving ray of (D) or a reducing direction for (P).
\item Let $(x_1^*, x_2^*, x_3^*)$ be a dual optimal solution.
If the optimal value is positive, $\frac{1}{x_1^*} x_3^*$ is an interior feasible solution to (P).
\end{enumerate}
\end{proposition}

\begin{proof}
\noindent (1): Let $(y^0, z^0, t^0, w^0) = \left( 0,\frac{1}{\langle e,e \rangle + 1} e,\frac{1}{\langle e,e \rangle + 1},\frac{1}{\langle e,e \rangle + 1} \right)$ and $(x_1^0, x_2^0, x_3^0) = (0, -1, 0)$.
Then $(y^0, z^0, t^0, w^0)$ and $(x_1^0, x_2^0, x_3^0)$ are interior feasible solutions to $(P^P_{\mathcal{K}})$ and $(D^P_{\mathcal{K}})$, respectively.
\medskip \\
\noindent (2): If $t^* = 0$, then $b^\top y^* = w^* \geq 0$ and $-\mathcal{A}^*y^* = z^* \in \mathcal{K}$ hold, which implies that (P) is not strongly feasible by Proposition \ref{pro: not sf iff}.
Conversely, if (P) is not strongly feasible, there exists $(y,z) \in \mathbb{R}^m \times \mathcal{K}$ such that $- \mathcal{A}^*y=z$, $\mathcal{A}^*y \neq 0$ and $b^\top y \geq 0$ by Proposition \ref{pro: not sf iff}.
Let $k = \langle e,z \rangle + b^\top y$.
Since $b^\top y \geq 0$, $z \in \mathcal{K}$ and $z \neq 0$ hold, $k$ is positive.
Letting $(\bar{y}, \bar{z}, \bar{t}, \bar{w}) = (\frac{1}{k} y, \frac{1}{k} z, 0, \frac{1}{k} b^\top y)$, we can easily see that $(\bar{y}, \bar{z}, \bar{t}, \bar{w})$ is an optimal solution for $(P^P_{\mathcal{K}})$, which implies that the optimal value is zero.

Suppose that $t^*=0$.
If $b^\top y^* =0$, then $z^* \neq 0$ holds since $w^* = 0$ and hence, $(y^*, z^*)$ is a reducing direction for (P).
If $b^\top y^* >0$, we can easily see that $(y^*, z^*)$ is an improving ray of (D).
\medskip \\
\noindent (3): If $x_2^* > 0$, then $x_1^* > 0$ and $x_3^* \in \mbox{int} \mathcal{K}$ hold.
Since $x_1^*$ and $x_3^*$ satisfy $\mathcal{A} x_3^* - x_1^* b = 0$, $\frac{1}{x_1^*} x_3^*$ is an interior feasible solution to (P).
\end{proof}

With some modifications to $(P^P_{\mathcal{K}})$ and $(D^P_{\mathcal{K}})$, we have optimization problems that Algorithm \ref{postpro alg} can handle well.

\begin{proposition}
\label{pro: check status p modify}
For (P) and (D), consider the following pair of primal and dual problems.
\begin{equation}
\begin{array}{lll}
(\overline{P}^P_{\mathcal{K}})&{\rm inf}_{\alpha, \beta, \gamma, s}	&\alpha	\\
			&{\rm sub.to}	&-\alpha + \beta + \gamma + \langle e,s \rangle = 0 \\
			&			&\frac{\alpha}{1 + \langle e,e \rangle} (b - \mathcal{A}e) - \gamma b + \mathcal{A}s = \frac{1}{1 + \langle e,e \rangle} (b - \mathcal{A}e)	\\
			&			&(\alpha, \beta, \gamma, s) \in \mathbb{R}_+ \times \mathbb{R}_+ \times \mathbb{R}_+ \times \mathcal{K} \\
\\
(\overline{D}^P_{\mathcal{K}})&{\rm sup}_{\kappa, f}	&\frac{1}{1 + \langle e,e \rangle} (b - \mathcal{A}e)^\top f	\\
			&{\rm sub.to}		&1 + \kappa - \frac{1}{1 + \langle e,e \rangle} (b - \mathcal{A}e)^\top f \geq 0 \\
			&				&-\kappa \geq 0 \\
			&				&-\kappa + b^\top f \geq 0 \\
			&				&-\kappa e - \mathcal{A}^* f \in \mathcal{K}
\end{array}
\notag
\end{equation}
The following properties hold.
\begin{enumerate}
\item Both $(\bar{P}^P_{\mathcal{K}})$ and $(\bar{D}^P_{\mathcal{K}})$ are strongly feasible.
\item The optimal value is smaller than or equal to 1.
\item Let $(\kappa^*, f^*)$ be a dual optimal solution.
If the optimal value is equal to 1, $(f^*,-\mathcal{A}^* f^*)$ is an improving ray of (D) or a reducing direction for (P).
\item Let $(\alpha^*, \beta^*, \gamma^*, s^*)$ be a primal optimal solution.
If the optimal value is smaller than 1, then
\begin{equation}
\notag
\frac{1+ \langle e,e \rangle}{\gamma^* (1 + \langle e,e \rangle) + 1 - \alpha^*} \left( s^* + \frac{1-\alpha^*}{1+\langle e,e \rangle} e\right)
\end{equation}
is an interior feasible solution to (P).
\end{enumerate}
\end{proposition}

\begin{proof}
\noindent (1): For any feasible solution $(x_1, x_2, x_3)$ to $(D^P_{\mathcal{K}})$, let
\begin{equation}
\notag
(\alpha, \beta, \gamma, s) = \left( 1- (1+ \langle e,e \rangle) x_2, 1-x_1-\langle e, x_3 \rangle, x_1 - x_2, x_3 - x_2 e \right).
\end{equation}
Then, $(\alpha, \beta, \gamma, s)$ is a feasible solution to $(\overline{P}^P_{\mathcal{K}} )$.
If $(x_1, x_2, x_3)$ is an interior feasible solution to $(D^P_{\mathcal{K}})$, we can easily see that $(\alpha, \beta, \gamma, s)$ is an interior feasible solution to $(\overline{P}^P_{\mathcal{K}})$.
Similarly, for any feasible solution $(y, z, t, w)$ to $(P^P_{\mathcal{K}})$, let $(\kappa, f) = (-t, y)$.
Then, $(\kappa, f)$ is a feasible solution to $(\overline{D}^P_{\mathcal{K}})$ because $-\kappa + b^\top f = w$ and
\begin{align*}
1+\kappa-\frac{1}{1+\langle e,e \rangle} (b-\mathcal{A}e)^\top f
&= \frac{\langle e,e \rangle}{1+\langle e,e \rangle} + \frac{1}{1+\langle e,e \rangle} \left( 1+ \kappa(1+\langle e,e \rangle) - (b-\mathcal{A}e)^\top f \right) \\
&= \frac{\langle e,e \rangle}{1+\langle e,e \rangle} + \frac{1}{1+\langle e,e \rangle} \left( 1+ \langle e, \kappa e + \mathcal{A}^* f \rangle + \kappa - b^\top f \right) \\
&= \frac{\langle e,e \rangle}{1+\langle e,e \rangle} + 1 - \langle e,z \rangle - w = \frac{\langle e,e \rangle}{1+\langle e,e \rangle}
\end{align*}
hold.
Moreover if $(y, z, t, w)$ is an interior feasible solution to $(P^P_{\mathcal{K}})$, we cab easily see that $(\kappa, f)$ is an interior feasible solution to $(\overline{D}^P_{\mathcal{K}})$.
Therefore $(\overline{P}^P_{\mathcal{K}})$ and $(\overline{D}^P_{\mathcal{K}})$ are strongly feasible.
\medskip \\
\noindent (2): $(D^P_{\mathcal{K}})$ has an optimal solution $(x_1^*, x_2^*, x_3^*)$ such that $x_2^* \geq 0$.
Noting that
\begin{equation}
\notag
(\alpha, \beta, \gamma, s) = \left( 1- (1+ \langle e,e \rangle) x^*_2, 1-x^*_1-\langle e, x^*_3 \rangle, x^*_1 - x^*_2, x^*_3 - x^*_2 e \right)
\end{equation}
is a feasible solution to $(\overline{P}^P_{\mathcal{K}})$, we can find that the optimal value of $(\overline{P}^P_{\mathcal{K}})$ is smaller than or equal to $1 - (1+ \langle e,e \rangle) x^*_2$.
Since $(\overline{P}^P_{\mathcal{K}})$ and $(\overline{D}^P_{\mathcal{K}})$ are strongly feasible and $1 + \langle e,e \rangle > 0$, the optimal values of $(\overline{P}^P_{\mathcal{K}})$ and $(\overline{D}^P_{\mathcal{K}})$ are smaller than or equal to 1.
\medskip \\
\noindent (3): If the optimal value is 1, we have
\begin{equation}
\notag
1 + \kappa^* - \frac{1}{1 + \langle e,e \rangle} (b - \mathcal{A}e)^\top f^* = \kappa^*\geq 0.
\end{equation}
Since $-\kappa^* \geq 0$ holds, we find $\kappa^*=0$.
Thus, $f^*$ satisfies $b^\top f^* \geq 0$, $-\mathcal{A}^* f^* \in \mathcal{K}$ and $b^\top f^* + \langle e, -\mathcal{A}^*f^* \rangle = 1+ \langle e,e \rangle$.
If $b^\top f^* = 0$, then $\langle e, -\mathcal{A}^*f^* \rangle = 1+ \langle e,e \rangle$ holds, which implies that $-\mathcal{A}^*f^* \neq 0$.
Therefore, $(f^*, -\mathcal{A}^*f^*)$ is a reducing direction for (P).
If $b^\top f^* > 0$, then we can easily see that $(f^*, -\mathcal{A}^*f^*)$ is an improving ray of (D).
\medskip \\
\noindent (4): Since $(\alpha^*, \beta^*, \gamma^*, s^*)$ is a feasible solution to $(\overline{P}^P_{\mathcal{K}})$, we find that
\begin{equation}
\notag
\mathcal{A} \left( s^* + \frac{1-\alpha^*}{1+\langle e,e \rangle} e\right) = \left( \gamma^* + \frac{1-\alpha^*}{1+\langle e,e \rangle} \right) b
\end{equation}
holds.
Let $\alpha^* < 1$.
Noting that $\gamma^* \geq 0$ and $s^* \in \mathcal{K}$, we can easily see that
\begin{equation}
\notag
\frac{1+ \langle e,e \rangle}{\gamma^* (1 + \langle e,e \rangle) + 1 - \alpha^*} \left( s^* + \frac{1-\alpha^*}{1+\langle e,e \rangle} e\right)
\end{equation}
is an interior feasible solution to (P). 
\end{proof}

\subsubsection{Experimental flow and results}
Now, let us explain the outline of our experiments to determine the feasibility status of SDPLIB instances and the weakly infeasible instances from \cite{Liu2018}. 
The flow of this experiment is almost the same as that of Section \ref{sec: numerical results post-pro}.
In other words, we obtained approximate optimal primal-dual solutions $(\alpha^*, \beta^*, \gamma^*, s^*, \kappa^*, f^*)$ to $(\overline{P}^P_{\mathcal{K}})$ and $(\overline{D}^P_{\mathcal{K}})$ using the solvers with the default setting, the solvers with the tight setting or Algorithm \ref{postpro alg}. 
Then, we compared these solutions in terms of accuracy and optimality. 
While the accuracy and optimality of outputs were measured using DIMACS errors in Section \ref{sec: numerical results post-pro}, we evaluated the outputs $(\alpha^*, \beta^*, \gamma^*, s^*, \kappa^*, f^*)$ by checking the values of $1-\alpha^*$, $1-\frac{1}{1 + \langle e,e \rangle} (b - \mathcal{A}e)^\top f^*$, $b^\top f^*$ and $\lambda_{\min} (-\mathcal{A}^* f^*)$ in this experiment. 
In what follows, $\bar{\theta}_{P}$ and $\bar{\theta}_{D}$ denote the primal and dual objective values of $(\overline{P}^P_{\mathcal{K}})$ and $(\overline{D}^P_{\mathcal{K}})$ for $(\alpha^*, \beta^*, \gamma^*, s^*, \kappa^*, f^*)$, respectively.

The first test cases are the ill-posed instances of SDPLIB defined in Section \ref{sec: numerical results post-pro}.
When using the tight tolerances, Mosek obtained approximate optimal solutions for these instances with such high accuracy that there was no need to use Algorithm \ref{postpro alg}. 
Thus, we did not apply Algorithm \ref{postpro alg} and just summarized the results of Mosek for the ill-posed group in Table \ref{Table: ill-posed status}.
From Table \ref{Table: ill-posed status}, we can see that Mosek, using the tight setting, determined that the primal problem for the ill-posed instances was in weak status and obtained highly accurate reduction directions.

\begin{table}
\caption{Results of feasibility determination experiments for the ill-posed group instances from SDPLIB: Mosek with the tight settings}
\label{Table: ill-posed status}
\begin{center}
\setlength{\tabcolsep}{5pt}
\footnotesize
\begin{tabular}{cccccc} \toprule
Instance	&$1-\bar{\theta}_P$	&$1-\bar{\theta}_D$	&$b^\top f^*$	&$\lambda_{\min} (-\mathcal{A}^* f^*)$	&time(s)	\\	\midrule
gpp100   &8.22e-15   &8.33e-15   &4.58e-16   &-1.29e-14   &4.26e-02  \\
gpp124-1   &2.95e-14   &3.00e-14   &7.31e-17   &-2.76e-14   &6.61e-02  \\
gpp124-2   &2.95e-14   &3.00e-14   &7.31e-17   &-2.76e-14   &6.40e-02  \\
gpp124-3   &2.95e-14   &3.00e-14   &7.31e-17   &-2.76e-14   &6.55e-02  \\
gpp124-4   &2.95e-14   &3.00e-14   &7.31e-17   &-2.76e-14   &6.52e-02  \\
hinf1   &4.22e-13   &4.17e-13   &-2.83e-13   &-3.69e-13   &8.78e-03  \\
hinf3   &2.22e-15   &2.44e-15   &-1.39e-15   &-3.08e-15   &8.45e-03  \\
hinf4   &4.04e-14   &4.07e-14   &-2.61e-14   &-2.82e-14   &8.88e-03  \\
hinf5   &1.39e-13   &1.39e-13   &-9.24e-14   &-1.34e-13   &1.02e-02  \\
hinf6   &3.47e-13   &3.30e-13   &-2.12e-13   &-1.97e-13   &8.48e-03  \\
hinf7   &1.84e-13   &1.84e-13   &-1.25e-13   &-1.17e-13   &9.54e-03  \\
hinf8   &5.98e-14   &5.91e-14   &-4.43e-14   &-4.40e-14   &9.56e-03  \\
hinf10   &1.62e-13   &1.63e-13   &-5.36e-14   &-5.15e-14   &1.13e-02  \\
hinf11   &5.38e-14   &5.33e-14   &-3.91e-14   &-3.76e-14   &1.22e-02  \\
hinf12   &4.44e-16   &3.33e-16   &1.30e-21   &-1.30e-21   &1.02e-02  \\
hinf13   &4.64e-14   &4.65e-14   &-1.99e-14   &-3.16e-14   &1.95e-02  \\
hinf14   &8.50e-14   &8.48e-14   &-4.53e-14   &-6.84e-14   &2.25e-02  \\
hinf15   &1.96e-13   &1.96e-13   &-1.53e-14   &-1.19e-13   &2.48e-02  \\
qap5   &1.00e-13   &1.03e-13   &-7.82e-14   &-9.85e-14   &1.03e-02  \\
qap6   &3.17e-13   &3.07e-13   &-3.55e-14   &-1.03e-13   &2.52e-02  \\
qap7   &2.16e-13   &2.19e-13   &-1.71e-13   &-2.01e-13   &3.40e-02  \\
qap8   &1.55e-15   &1.78e-15   &5.68e-14   &-2.50e-15   &8.68e-02  \\
qap9   &8.88e-16   &6.66e-16   &-5.68e-14   &-6.36e-18   &1.17e-01  \\
qap10   &1.55e-15   &1.55e-15   &-5.68e-14   &-9.00e-15   &2.86e-01  \\
\bottomrule
\end{tabular}
\end{center}
\end{table}

The following test cases are the weakly infeasible instances of \cite{Liu2018}.
These instances are classified into four sets, ``clean-10-10", ``clean-20-10", ``messy-10-10" and ``messy-20-10". 
For example, the set ``clean-10-10" contains instances where $m=10$ and $r=10$, and the set ``clean-20-10" contains instances where $m=20$ and $r=10$. 
The instance set labeled ``messy" includes instances with a less obvious structure leading to weak infeasibility than the set labeled ``clean.'' 
The results for the four instance sets are summarized in Figures \ref{fig: Mosek-weakinf}-\ref{fig: SDPT3-weakinf} and Table \ref{Table: WeakInf Time}.
Since each set contains 100 instances, we use box plots to report the values of $|1-\bar{\theta}_P|$, $|1-\bar{\theta}_D|$, $|b^\top f^*|$ and $| \min \{ \lambda_{\min} (-\mathcal{A}^*f^*), 0 \} |$ logarithmized by the base 10.
In these graphs, ``Val1", ``Val2", ``Val3", and ``Val4" represent $|1-\bar{\theta}_P|$, $|1-\bar{\theta}_D|$, $|b^\top f^*|$ and $| \min \{ \lambda_{\min} (-\mathcal{A}^*f^*), 0 \} |$, respectively.
Blue objects represent the results obtained with the solver using the default setting; red objects represent the results obtained with Algorithm \ref{postpro alg}, and yellow objects denote results obtained with the solver using the tight settings. 
Table \ref{Table: WeakInf Time} summarizes the average execution time for each method.
Because the problem size was small, there was no significant difference in the execution time of each method.

\begin{table}[H]
\caption{Average execution time for each method in the feasibility determination experiment}
\label{Table: WeakInf Time}
\begin{center}
\setlength{\tabcolsep}{5pt}
\footnotesize
\begin{tabular}{ccccc} \toprule
Method	&clean-10-10	&clean-20-10	&messy-10-10	&messy-20-10	\\	\midrule
Mosek (default)
&8.31e-03	&9.88e-03	&8.11e-03	&9.05e-03	\\
Algorithm \ref{postpro alg} using Mosek
&3.31e-01	&1.27e+00	&8.83e-01	&1.13e+00	\\
Mosek (tight)
&1.07e-02	&1.51e-02	&1.31e-02	&1.95e-02	\\
SDPA (default)
&1.99e+00	&2.05e+00	&2.02e+00	&2.14e+00	\\
Algorithm \ref{postpro alg} using SDPA
&1.62e-01	&1.81e-01	&2.52e-01	&2.87e-01	\\
SDPA (tight)
&2.17e+00	&2.90e+00	&4.20e+00	&3.51e+00	\\
SDPT3 (default)
&7.86e-02	&8.52e-02	&8.31e-02	&9.77e-02	\\
Algorithm \ref{postpro alg} using SDPT3
&1.99e-01	&2.25e-01	&3.05e-01	&3.02e-01	\\
SDPT3 (tight)
&1.09e-01	&1.21e-01	&1.05e-01	&1.15e-01	\\
\bottomrule
\end{tabular}
\end{center}
\end{table}

First, let us compare the results on the instances labeled ``clean".
From Figures \ref{fig: Mosek-weakinf} and \ref{fig: SDPT3-weakinf}, we see that Mosek and SDPT3, using the tight settings, obtained the approximate optimal values with such high accuracy that there is no need to use Algorithm \ref{postpro alg}.
For the SDPA using the tight tolerances, the optimal values were obtained with roughly the same accuracy as when using the default settings. 
Figures \ref{fig: Mosek-weakinf}-\ref{fig: SDPT3-weakinf} also show that Algorithm \ref{postpro alg} consistently obtained highly accurate optimal values regardless of the solver used together. 
Next, let us compare the results on the instances labeled ``messy".
Here, we note that Mosek returned incorrect outputs for some instances when using the tight tolerances. 
While Mosek with the default setting obtained approximate optimal solutions $(\alpha^*, \beta^*, \gamma^*, s^*, \kappa^*, f^*)$ for all instances, Mosek with the tight setting returned incorrect infeasibility certificates to $(\overline{P}^P_{\mathcal{K}})$ for 23 instances of ``messy-10-10" and 35 instances of ``messy-20-10".
Figures \ref{fig: Mosek-weakinf}-\ref{fig: SDPT3-weakinf} show that the accuracy of the optimal values obtained by Mosek and SDPT3 was worse than the results for the instance sets labeled ``clean". 
On the other hand, the accuracy of the dual optimal values $\bar{\theta}_D$ obtained by SDPA using tight settings was better than the accuracy of the dual optimal values obtained for ``clean" instances.
Algorithm \ref{postpro alg} obtained optimal values for the instance sets labeled ``messy" with the same accuracy as for the instance sets labeled ``clean.''
In addition, Figures \ref{fig: Mosek-weakinf}-\ref{fig: SDPT3-weakinf} show that the solvers and Algorithm \ref{postpro alg} obtained approximate reducing directions for the weakly infeasible instances of \cite{Liu2018}, but their accuracy was inferior to the results for the ill-posed instances of SDPLIB in Table \ref{Table: ill-posed status}.

\begin{figure}[htbp]
\hspace{-1cm}
    \begin{tabular}{lr}
      \begin{minipage}[t]{0.485\hsize}
        \centering
        \includegraphics[keepaspectratio, scale=0.65]{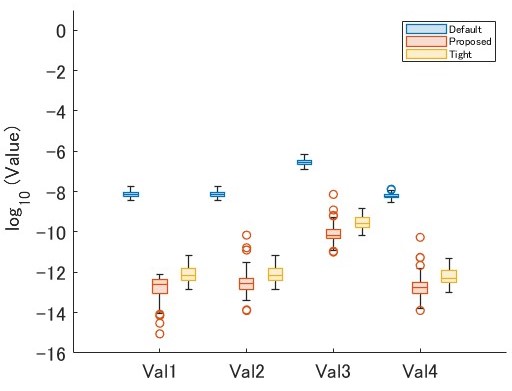}
        \subcaption{The results on clean-10-10}
        \label{fig: Mosek-clean10}
      \end{minipage} \hspace{0.5cm}&\hspace{0.5cm}
      \begin{minipage}[t]{0.485\hsize}
        \centering
        \includegraphics[keepaspectratio, scale=0.65]{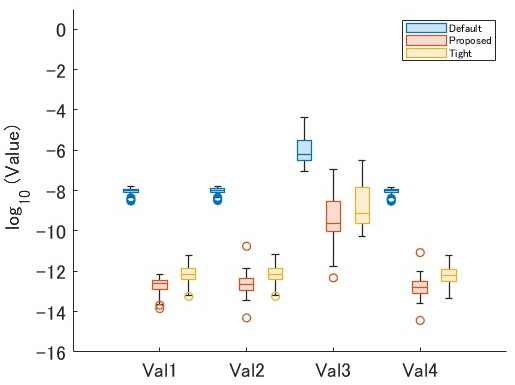}
	\subcaption{The results on clean-20-10}
        \label{fig: Mosek-clean20}
      \end{minipage} \\

      \begin{minipage}[t]{0.485\hsize}
        \centering
        \includegraphics[keepaspectratio, scale=0.65]{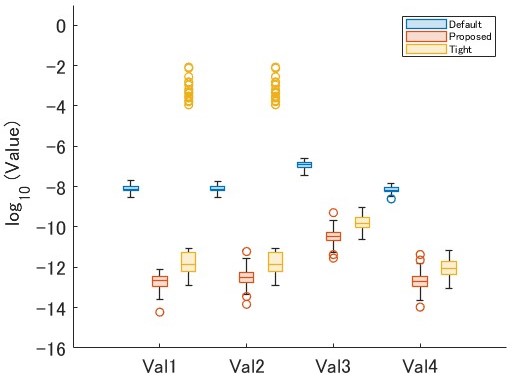}
	\subcaption{The results on messy-10-10}
        \label{fig: Mosek-messy10}
      \end{minipage} \hspace{0.5cm}&\hspace{0.5cm}
      \begin{minipage}[t]{0.485\hsize}
        \centering
        \includegraphics[keepaspectratio, scale=0.65]{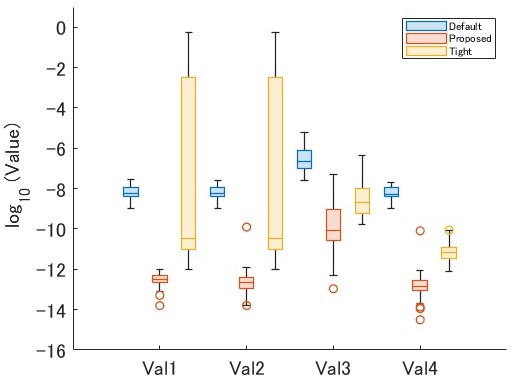}
	\subcaption{The results on messy-20-10}
        \label{fig: Mosek-messy20}
      \end{minipage}
   \end{tabular}
\caption{Numreical results for the weakly infeasible instances of \cite{Liu2018} with Mosek}
\label{fig: Mosek-weakinf}
\end{figure}

\begin{figure}[htbp]
\hspace{-1cm}
    \begin{tabular}{lr}
      \begin{minipage}[t]{0.485\hsize}
        \centering
        \includegraphics[keepaspectratio, scale=0.65]{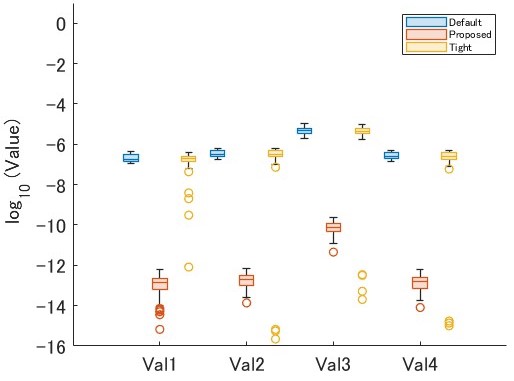}
        \subcaption{The results on clean-10-10}
        \label{fig: SDPA-clean10}
      \end{minipage} \hspace{0.5cm}&\hspace{0.5cm}
      \begin{minipage}[t]{0.485\hsize}
        \centering
        \includegraphics[keepaspectratio, scale=0.65]{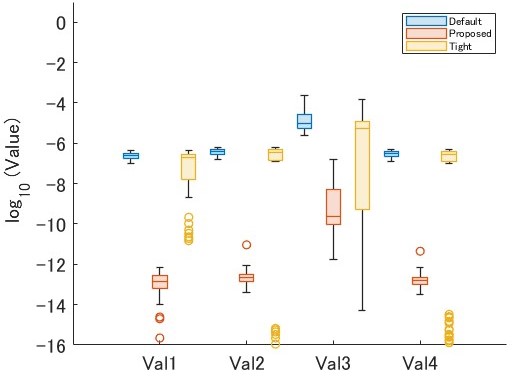}
	\subcaption{The results on clean-20-10}
        \label{fig: SDPA-clean20}
      \end{minipage} \\

      \begin{minipage}[t]{0.485\hsize}
        \centering
        \includegraphics[keepaspectratio, scale=0.65]{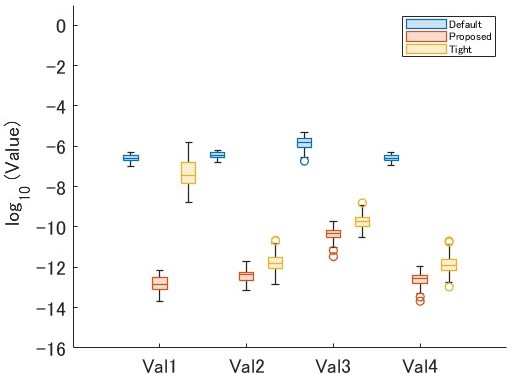}
	\subcaption{The results on messy-10-10}
        \label{fig: SDPA-messy10}
      \end{minipage} \hspace{0.5cm}&\hspace{0.5cm}
      \begin{minipage}[t]{0.485\hsize}
        \centering
        \includegraphics[keepaspectratio, scale=0.65]{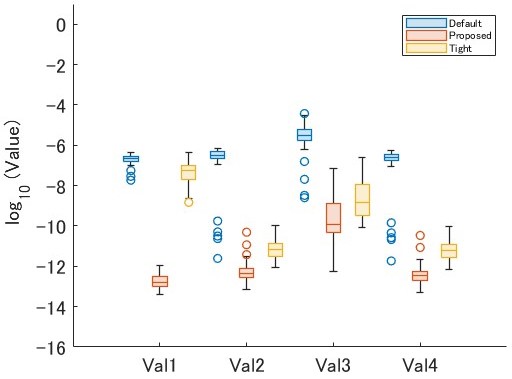}
	\subcaption{The results on messy-20-10}
        \label{fig: SDPA-messy20}
      \end{minipage}
   \end{tabular}
\caption{Numreical results for the weakly infeasible instances of \cite{Liu2018} with SDPA}
\label{fig: SDPA-weakinf}
\end{figure}

\begin{figure}[htbp]
\hspace{-1cm}
    \begin{tabular}{lr}
      \begin{minipage}[t]{0.485\hsize}
        \centering
        \includegraphics[keepaspectratio, scale=0.65]{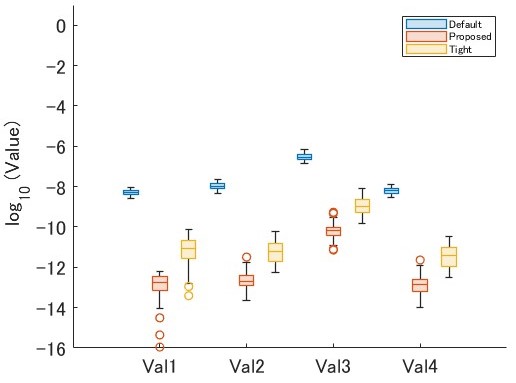}
        \subcaption{The results on clean-10-10}
        \label{fig: SDPT3-clean10}
      \end{minipage} \hspace{0.5cm}&\hspace{0.5cm}
      \begin{minipage}[t]{0.485\hsize}
        \centering
        \includegraphics[keepaspectratio, scale=0.65]{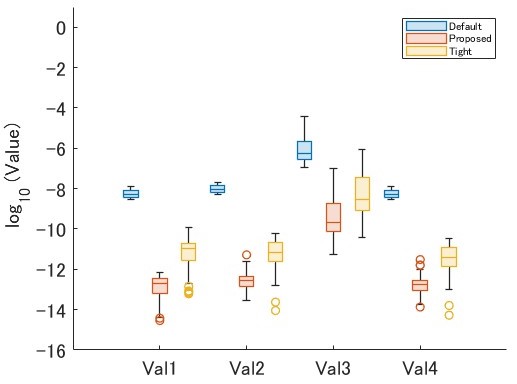}
	\subcaption{The results on clean-20-10}
        \label{fig: SDPT3-clean20}
      \end{minipage} \\

      \begin{minipage}[t]{0.485\hsize}
        \centering
        \includegraphics[keepaspectratio, scale=0.65]{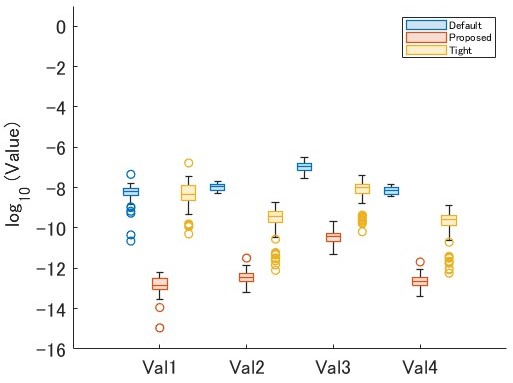}
	\subcaption{The results on messy-10-10}
        \label{fig: SDPT3-messy10}
      \end{minipage} \hspace{0.5cm}&\hspace{0.5cm}
      \begin{minipage}[t]{0.485\hsize}
        \centering
        \includegraphics[keepaspectratio, scale=0.65]{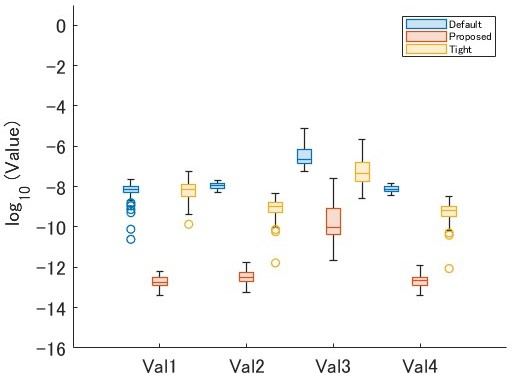}
	\subcaption{The results on messy-20-10}
        \label{fig: SDPT3-messy20}
      \end{minipage}
   \end{tabular}
\caption{Numreical results for the weakly infeasible instances of \cite{Liu2018} with SDPT3}
\label{fig: SDPT3-weakinf}
\end{figure}

\newpage

\section{Concluding remarks}
\label{sec: pp conclusion}
In this study, we proposed the algorithm for solving SCPs using projection and rescaling methods. 
Although our algorithm can solve SCPs by itself, we proposed it intending to use it as a post-processing step for the interior point method.
In addition, we proposed some techniques to make our algorithm more practical. 
We also conducted numerical experiments with SDPLIB instances and compared the accuracy of the approximate optimal solution obtained from our method with that obtained from Mosek, SDPA, and SDPT3. 
Our numerical results showed that
\begin{itemize}
\item For the well-posed group, i.e., the primal and dual problems are expected to be strongly feasible, our algorithm consistently obtained a more accurate approximate optimal solution than that returned from the solvers. 
\item For the ill-posed group, i.e., at least one of the primal or dual problems is expected not to be strongly feasible, our algorithm obtained a more accurate feasible solution than the solvers. However, our algorithm did not stably return a solution with good optimality for the ill-posed group compared to the results for the well-posed group. 
\end{itemize}
In addition, using the formulation in Proposition \ref{pro: check status p modify}, we performed feasibility status determination experiments on the ill-posed instances of SDPLIB and the weakly infeasible instances of \cite{Liu2018}.
For the ill-posed instances of SDPLIB, the solver obtained approximate optimal solutions with such high accuracy that there was no need to use our algorithm.
However, for the instances of \cite{Liu2018}, our algorithm obtained more accurate optimal values than the solvers, which implies that our algorithm can contribute to detecting the feasibility status of SDP.
The numerical results also indicated the difficulties related to the execution time of our algorithm.

To overcome this problem, we have two future directions. 
The first one is to consider an efficient computation of a projection matrix.
If the problem size is large, the computation of the projection matrix is expected to take much more time.
Thus, efficient methods of computing projection matrices will reduce the execution time of our algorithm.
For example, using an approximate projection matrix instead of an exact one might make our algorithm more practical, although the output solution may be less accurate.  
The second is to process the operations of Algorithms \ref{Practical p alg} and \ref{Practical d alg} more efficiently using parallel computation. 
Algorithms \ref{Practical p alg} and \ref{Practical d alg} choose the input value $\theta \in (LB, UB)$, call the projection and rescaling algorithm with the corresponding feasibility problem, and update $UB$ or $LB$ with $\theta$ according to the output from the projection and rescaling algorithm until $UB-LB \leq \theta_{acc}$, where $\theta_{acc}$ is an accuracy parameter specified by the user.
Therefore, if Algorithms \ref{Practical p alg} and \ref{Practical d alg} execute the projection and rescaling algorithms in parallel for multiple input values, the execution time of these algorithms can be reduced, which will reduce the execution time of Algorithm \ref{postpro alg}. 
In addition, we should consider why the projection and rescaling methods could obtain feasible solutions with higher accuracy than the solvers. 
To examine the differences between the projection and rescaling methods and interior point methods, it might be a good idea first to explore the problem structure that the projection and rescaling methods can solve more accurately than the solvers.

\section*{Acknowledgement}
This work was supported by JSPS KAKENHI Grant Numbers JP21J20875, 22K18866, 23H01633.

\bibliographystyle{plain}
%

%
%
%
%
%
%
%
%
%
%
%
%

\appendix

\section{Details of the modification in Section \ref{sec: tech2} }
\label{Appendix A}
Algorithm \ref{d update pro} computes the direction $d$ such that $b^\top d > 0$ using $y_{tmp}$ and $\bar{y}$, and then sets the value of $u$ as the upper bound on the step size $\alpha$. 
After that, Algorithm \ref{d update pro} adjust the step size $\alpha$ and compute $y := \bar{y} + \alpha d$ until $c-\mathcal{A}^*y \in \mathcal{K}$ or $b^\top (\alpha d) \leq 1$e-16 holds. 
There is room for improvement in setting the value of $u$ and choosing the step size $\alpha$. 
However, recall that Algorithms \ref{p alg} and \ref{d alg} are used in a post-processing step.
Since the approximate optimal value is known in advance, the projection and rescaling algorithms in Algorithm \ref{p alg} and \ref{d alg} will return an approximate optimal solution from the first iteration. 
Thus, we can expect $\alpha$ satisfying $c-\mathcal{A}^*(\bar{y} + \alpha d) \in \mathcal{K}$ to be small, which implies that further study on how to choose $u$ and $\alpha$ is not essential. 

\begin{algorithm}[H]
 \caption{Update procedure for dual feasible solutions}
 \label{d update pro}
 \begin{algorithmic}[1]
 \renewcommand{\algorithmicrequire}{\textbf{Input: }}
 \renewcommand{\algorithmicensure}{\textbf{Output: }}
 \renewcommand{\stop}{\textbf{stop }}
 \renewcommand{\return}{\textbf{return }}

 \STATE \algorithmicrequire $\mathcal{A}$, $b$, $c$, $\mathcal{K}$, $y_{tmp}$ and $\bar{y}$ such that $c-\mathcal{A}^*\bar{y} \in \mathcal{K}$.
 \STATE \algorithmicensure A feasible solution $(y,z)$ to (D).
 \IF {$b^\top \bar{y} = b^\top y_{tmp}$}
 \STATE $u \leftarrow 0$
 \ELSE
 \IF {$b^\top \bar{y} > b^\top y_{tmp}$}
 \STATE $d \leftarrow \bar{y} - y_{tmp}$, $u \leftarrow 5$
 \ELSE
 \STATE $d \leftarrow y_{tmp} - \bar{y}$
 \IF {$c-\mathcal{A}^*y_{tmp} \in \mathcal{K}$}
 \STATE $\bar{y} \leftarrow y_{tmp}$, $u \leftarrow 5$
 \ELSE
 \STATE $u \leftarrow 1$
 \ENDIF
 \ENDIF
 \ENDIF
 \IF {$u>0$}
 \STATE $\alpha \leftarrow u$
 \WHILE {$b^\top (\alpha d) >$1e-16}
 \STATE $y \leftarrow \bar{y} + \alpha d$
 \IF{$c-\mathcal{A}^*y \in \mathcal{K}$}
 \STATE \stop Algorithm \ref{d update pro} and \return $(y, c-\mathcal{A}^*y)$ 
 \ELSE
 \STATE $u \leftarrow \alpha$, $\alpha \in (0, u)$
 \ENDIF
 \ENDWHILE
 \STATE \return $(\bar{y}, c-\mathcal{A}^*\bar{y})$
 \ELSE
 \STATE \return $(\bar{y}, c-\mathcal{A}^*\bar{y})$
 \ENDIF

 \end{algorithmic} 
 \end{algorithm}

\section{Further discussion of the modification in Section \ref{sec: tech7}}
\label{Appendix B}
Let us consider the following simple minimization problem (Ex.B).
The feasible region of (Ex.B) is illustrated in Figure \ref{figure: Ex.1}.
The optimal solution $x^*$ is $(1,0)^\top \in \mathbb{R}^2$ and the optimal value $\theta^*$ is $1$. 
\begin{equation}
\notag
\begin{array}{cccll}
{\rm (Ex.B)} & \displaystyle \min_x \ x_1& {\rm s.t.} & \frac{1}{4} x_1 + x_2 = 1, &x = 
\begin{pmatrix}
x_1 \\
x_2
\end{pmatrix} \in \mathbb{R}^2_{++}.
\end{array}
\end{equation}

\begin{figure}[htb]
\begin{center}
\includegraphics[keepaspectratio, scale=0.325]{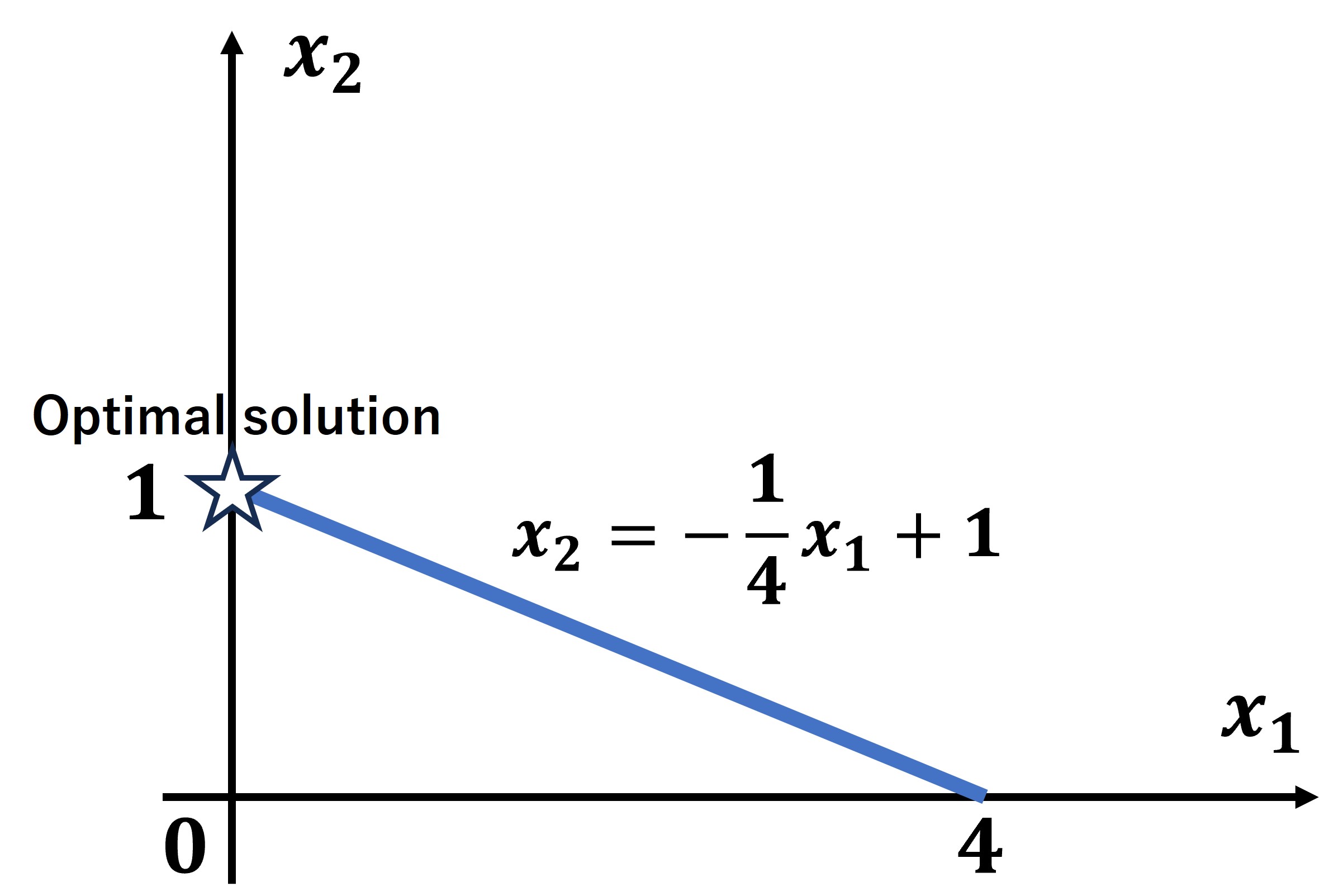}
\end{center}
\caption{Feasible region of (Ex.B)}
\label{figure: Ex.1}
\end{figure}

Using the formulation in Section \ref{sec: modeling1}, the feasibility problem $\mbox{FP}_{S_\infty} (\mbox{ker} \mathcal{A}(\theta), \mathbb{R}^4_{++})$ is obtained, where $\theta \in \mathbb{R}$ and 
\begin{equation}
\notag
\mathcal{A}(\theta) = 
\begin{pmatrix}
\frac{1}{4}	&1	&-1	&0	\\
1	&0	&-\theta	&1
\end{pmatrix}
.
\end{equation}
Here, let us define $v \in \mathbb{R}^4_{++}$ as $v := (\frac{1}{\sqrt{\alpha}}, \frac{1}{\sqrt{\beta},} 1, 1)^\top$ for any $(\alpha, \beta)^\top \in \mathbb{R}^2_{++}$ and consider the scaled feasibility problem $\mbox{FP}_{S_\infty} (\mbox{ker} \bar{\mathcal{A}} (\theta), \mathbb{R}^4_{++})$ such that $\mbox{ker} \bar{\mathcal{A}} (\theta) = Q_v (\mbox{ker} \mathcal{A} (\theta))$.
\begin{equation}
\notag
\begin{array}{llll}
\mbox{FP}_{S_\infty} (\mbox{ker} \bar{\mathcal{A}} (\theta), \mathbb{R}^4_{++}):	&\mbox{find}
\begin{pmatrix}
x_1 \\
x_2 \\
\tau \\
\rho
\end{pmatrix}
\in \mathbb{R}^4
& {\rm s.t.} &
\begin{pmatrix}
\frac{\alpha}{4}	&\beta	&-1	&0	\\
\alpha	&0	&-\theta	&1
\end{pmatrix}
\begin{pmatrix}
x_1 \\
x_2 \\
\tau \\
\rho
\end{pmatrix}
=0, \\
&&&0 < x_1 \leq 1, \ 0< x_2 \leq 1, \\
&&&0 < \tau \leq 1, \ 0 < \rho \leq 1.
\end{array}
\end{equation}
To know the value of $\delta_\infty (\mbox{ker} \bar{\mathcal{A}} (\theta) \cap \mathbb{R}^4_{++})$, let us represent the feasible region of $\mbox{FP}_{S_\infty} (\mbox{ker} \bar{\mathcal{A}} (\theta), \mathbb{R}^4_{++})$ on the $\tau$-$\rho$ plane. 
Noting that $x_1= \frac{\theta}{\alpha}\tau - \frac{1}{\alpha} \rho$ and $x_2 = \frac{4-\theta}{4\beta} \tau + \frac{1}{4\beta} \rho$ hold for any feasible solution $(x_1, x_2, \tau, \rho)$ to $\mbox{FP}_{S_\infty} (\mbox{ker} \bar{\mathcal{A}} (\theta), \mathbb{R}^4_{++})$, the feasible region of $\mbox{FP}_{S_\infty} (\mbox{ker} \bar{\mathcal{A}} (\theta), \mathbb{R}^4_{++})$ is given by
\begin{equation}
\notag
\left\{ \left( \frac{\theta}{\alpha} \tau - \frac{1}{\alpha} \rho, \frac{4-\theta}{4\beta}\tau + \frac{1}{4\beta} \rho, \tau, \rho \right) : 0 < \theta \tau - \rho \leq \alpha, 0 < (4-\theta) \tau + \rho \leq 4\beta, 0 < \tau, \rho \leq 1 \right\}.
\end{equation}
For example, if $0 < \theta < 4$, $\frac{\alpha + 4\beta}{4} < \frac{4\beta}{4 - \theta}$, $\theta \beta \leq 1$, and $\frac{\alpha + 4\beta}{4} \leq 1$ hold, the feasible region of $\mbox{FP}_{S_\infty} (\mbox{ker} \bar{\mathcal{A}} (\theta), \mathbb{R}^4_{++})$ is as shown in Figure \ref{figure: example1}. 
In this case, we can easily see that the point giving the maximum value of $\delta_\infty (\mbox{ker} \bar{\mathcal{A}} (\theta) \cap \mathbb{R}^4_{++})$ lies in the red line segment in Figure \ref{figure: example1}. 
Thus, the value of $\delta_\infty (\mbox{ker} \bar{\mathcal{A}} (\theta) \cap \mathbb{R}^4_{++})$ is represented as the optimal value of a maximization problem $\max_{\tau \in {\rm dom} T} f(\tau)$, where $f(\tau) = \frac{4}{\alpha} \tau (\tau - \beta) (4\beta - (4-\theta) \tau)$ and ${\rm dom}\ T = (\beta, \frac{\alpha+4\beta}{4}]$.
This maximization problem can be solved with derivatives and simple calculations.
The function $f(\tau)$ is obtained by substituting $x_1 = \frac{\theta}{\alpha} \tau - \frac{1}{\alpha} \rho, x_2=1$, and $\rho = 4 \beta - (4-\theta) \tau$ into $x_1 x_2 \tau \rho$. 
Next, let us consider a more complex case. 
If $0 < \theta < 4$, $\frac{\alpha + 4\beta}{4} < \frac{4\beta}{4 - \theta}$, $\theta \beta \leq 1$, and $\beta \leq 1 < \frac{\alpha + 4\beta}{4}$ hold, the feasible region of $\mbox{FP}_{S_\infty} (\mbox{ker} \bar{\mathcal{A}} (\theta), \mathbb{R}^4_{++})$ is as shown in Figure \ref{figure: example2}.
The point giving the maximum value of $\delta_\infty (\mbox{ker} \bar{\mathcal{A}} (\theta) \cap \mathbb{R}^4_{++})$ lies in the red or blue line segment in Figure \ref{figure: example2}.
Thus, the value of $\delta_\infty (\mbox{ker} \bar{\mathcal{A}} (\theta) \cap \mathbb{R}^4_{++})$ is represented as the optimal value of $\max \{ \max_{\tau \in {\rm dom} T} f(\tau), \max_{\rho \in {\rm dom} R} g(\rho)\}$, where $f(\tau) = \frac{4}{\alpha} \tau (\tau - \beta) (4\beta - (4-\theta) \tau)$, $g(\rho) = \left( \frac{\theta}{\alpha} - \frac{1}{\alpha} \rho \right) \left( \frac{4-\theta}{4\beta} + \frac{1}{4\beta} \rho \right) \rho$, ${\rm dom}\ T = (\beta, 1]$, and ${\rm dom}\ R = [ \max \{0, \theta - \alpha \}, -(4-\theta) + 4\beta]$.
The function $g(\rho)$ is obtained by substituting $x_1 = \frac{\theta}{\alpha} \tau - \frac{1}{\alpha} \rho, x_2=\frac{4-\theta}{4\beta} \tau + \frac{1}{4\beta} \rho$, and $\tau=1$ into $x_1 x_2 \tau \rho$.

\begin{figure}[htbp]
    \begin{tabular}{l}
      \begin{minipage}[t]{\linewidth}
        \includegraphics[keepaspectratio, scale=0.5]{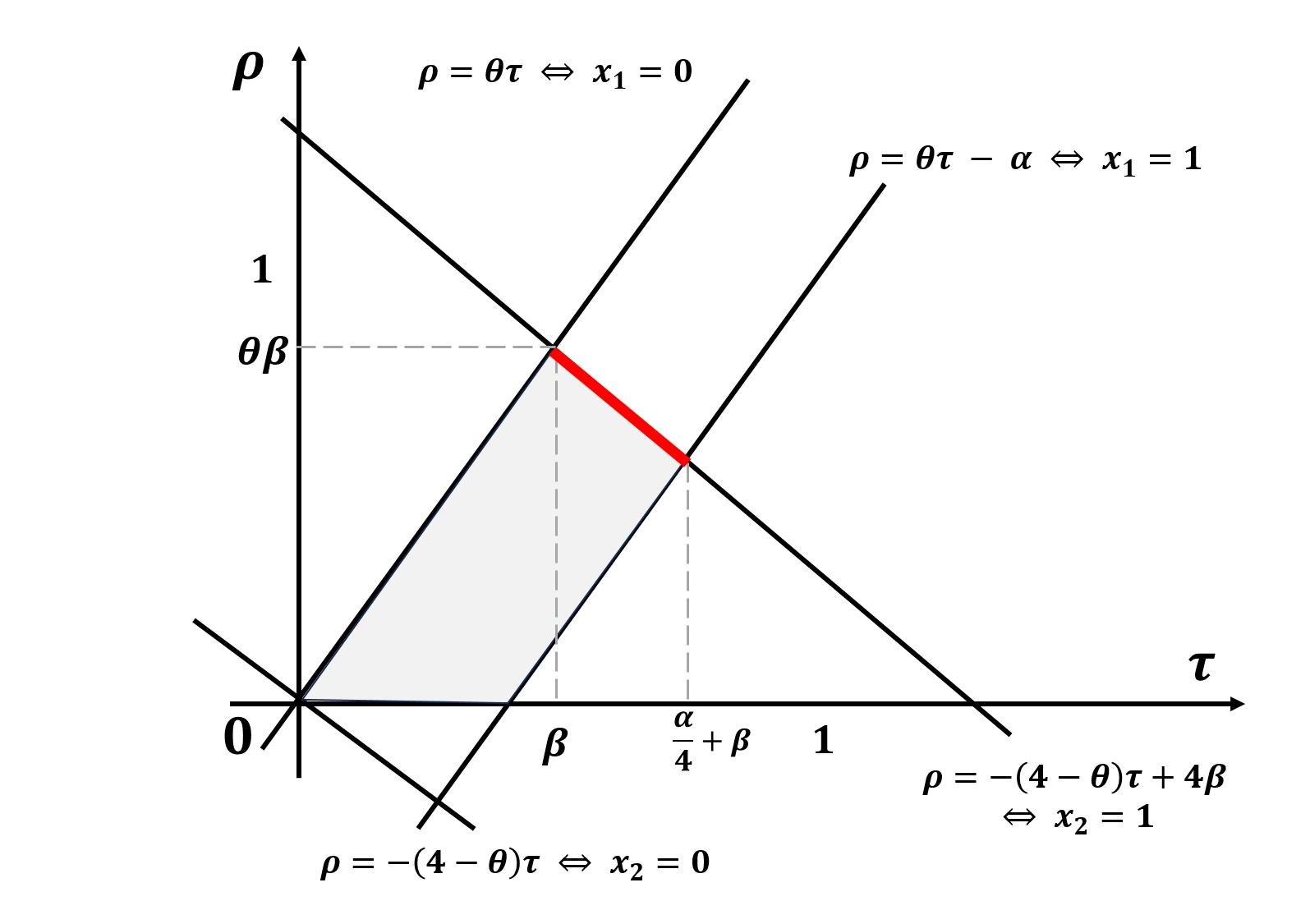}
        \subcaption{$0 < \theta < 4$, $\frac{\alpha + 4\beta}{4} < \frac{4\beta}{4 - \theta}$, $\theta \beta \leq 1$, and $\frac{\alpha + 4\beta}{4} \leq 1$}
        \label{figure: example1}
      \end{minipage} \vspace{5mm} \\
      \begin{minipage}[t]{\linewidth}
        \includegraphics[keepaspectratio, scale=0.5]{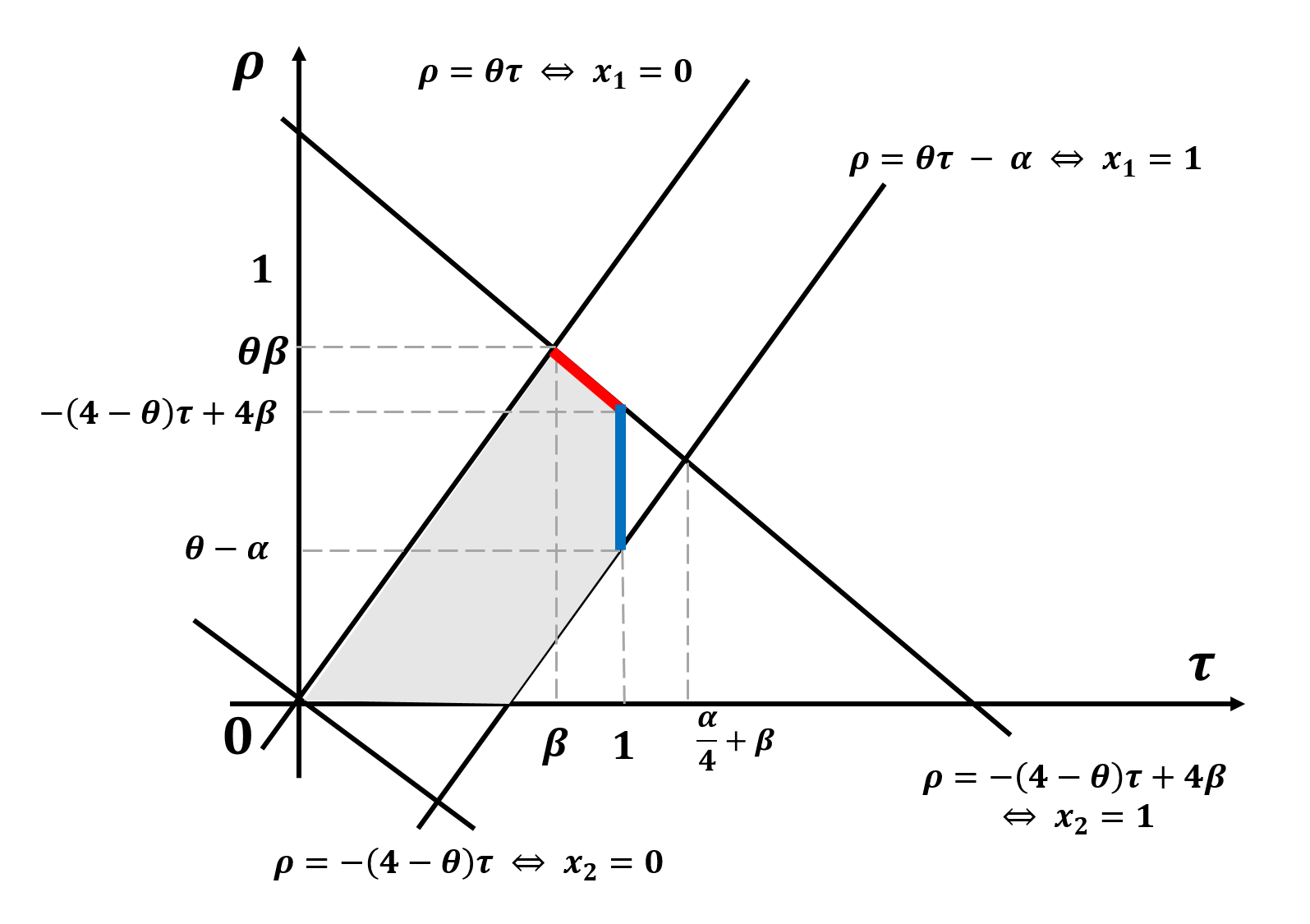}
        \subcaption{$0 < \theta < 4$, $\frac{\alpha + 4\beta}{4} < \frac{4\beta}{4 - \theta}$, $\theta \beta \leq 1$, and  $\beta \leq 1 < \frac{\alpha + 4\beta}{4}$}
	\label{figure: example2}
      \end{minipage} 
    \end{tabular}
     \caption{Feasible region of $\mbox{FP}_{S_\infty} (\mbox{ker} \bar{\mathcal{A}} (\theta), \mathbb{R}^4_{++})$}
\end{figure}

Considering the other cases in the same way, the values of $\delta_\infty (\mbox{ker} \bar{\mathcal{A}} (\theta) \cap \mathbb{R}^4_{++})$ are summarized in Tables \ref{Table: case 1}-\ref{Table: case 4}. 
The columns in these tables show the relations that hold for $\alpha, \beta$, and $\theta$. 
For example, the second column of Table \ref{Table: case 1} shows the case where $0 < \theta < 4$, $\frac{\alpha + 4\beta}{4} < \frac{4\beta}{4 - \theta}$, $\theta \beta < 1$ and $1 \leq \beta$ hold. 
The ``${\rm dom}\ T$" and ``${\rm dom}\ R$" rows show the domain of $\tau$ and $\rho$, respectively. 
The ``$\delta_\infty$" row shows the value of $\delta_\infty (\mbox{ker} \bar{\mathcal{A}} (\theta) \cap \mathbb{R}^4_{++})$.
Based on Tables \ref{Table: case 1}-\ref{Table: case 4}, we computed the value of $\delta_\infty (\mbox{ker} \bar{\mathcal{A}} (\theta) \cap \mathbb{R}^4_{++})$ by substituting specific values into $\alpha, \beta$, and $\theta$, and then checked whether Assumption \ref{assumption} holds and the point $(\alpha, \beta)^\top \in \mathbb{R}^2_{++}$ such that $\delta_\infty (\mbox{ker} \bar{\mathcal{A}} (\theta) \cap \mathbb{R}^4_{++}) \geq \delta_\infty (\mbox{ker} \mathcal{A} (\theta) \cap \mathbb{R}^4_{++})$ can likely hold. 

First, let us see whether Assumption \ref{assumption} holds using Figure \ref{figure: Assumption}.
Figure \ref{figure: Assumption} shows a heatmap of the value of $\delta_\infty (\mbox{ker} \bar{\mathcal{A}} (\theta) \cap \mathbb{R}^4_{++})$ when $\beta = 1 - \frac{\alpha}{4}$, i.e., $(\alpha, \beta)^\top$ satisfies the linear constraint of (Ex.B). 
In Figure \ref{figure: Assumption}, the horizontal axis represents the value of $\theta$, the vertical axis represents the value of $\alpha$, and the cells show the values of $\delta_\infty (\mbox{ker} \bar{\mathcal{A}} (\theta) \cap \mathbb{R}^4_{++})$ corresponding to $\theta$ and $\alpha$. 
This figure provides us with an intuitive understanding that Assumption \ref{assumption} holds. 
We confirmed Assumption \ref{assumption} holds for this example by checking the values of all cells.

Next, let us see what $\alpha$ and $\beta$ are most likely to satisfy $\delta_\infty (\mbox{ker} \bar{\mathcal{A}} (\theta) \cap \mathbb{R}^4_{++}) \geq \delta_\infty (\mbox{ker} \mathcal{A} (\theta) \cap \mbox{int } \bar{\mathcal{K}})$.
We computed the value of $\delta_\infty (\mbox{ker} \bar{\mathcal{A}} (\theta) \cap \mathbb{R}^4_{++})$ for each $\theta \in \{ 0.1, 0.2, \dots, 4.0 \}$, $\alpha \in \{ 0.1, 0.2, \dots, 4.2 \}$, and $\beta \in \{ 0.1, 0.2, \dots, 1.3 \}$, and then made Figure \ref{figure: approximate scale}. 
In each graph in Figure \ref{figure: approximate scale}, the horizontal axis represents the value of $\alpha$, the vertical axis represents the value of $\beta$, and the cells show the values of $\delta_\infty (\mbox{ker} \bar{\mathcal{A}} (\theta) \cap \mathbb{R}^4_{++})$ corresponding to $\theta, \alpha$, and $\beta$. 
We note that $\alpha = \beta = 1$ implies $\mbox{ker} \bar{\mathcal{A}} (\theta) = \mbox{ker} \mathcal{A} (\theta)$.
From these figures, we can observe the following: 
\begin{itemize}
\item When $\max \{ \theta_p, \theta - 1 \}  = \theta -1$, i.e., $\theta > 1$, the closer the point $(\alpha, \beta)^\top$ is to the interior feasible solution $(x_1, x_2)$ such that $x_1 = \theta-1$, i.e., $(\theta-1, \frac{5 - \theta}{4})^\top$ the more likely it seems that $\delta_\infty (\mbox{ker} \bar{\mathcal{A}} (\theta) \cap \mathbb{R}^4_{++}) \geq \delta_\infty (\mbox{ker} \mathcal{A} (\theta) \cap \mathbb{R}^4_{++})$ holds unless $e$ is an approximate feasible solution for (P) such that $\langle c, e \rangle \simeq \theta-1$. 
\item When $\max \{ \theta_p, \theta - 1 \}  = \theta_p$, i.e., $0 < \theta < 1$ the closer the point $(\alpha, \beta)^\top$ is to the optimal solution $(0, 1)^\top$, the more likely it seems that $\delta_\infty (\mbox{ker} \bar{\mathcal{A}} (\theta) \cap \mathbb{R}^4_{++}) \geq \delta_\infty (\mbox{ker} \mathcal{A} (\theta) \cap \mathbb{R}^4_{++})$ holds.
\end{itemize}

\begin{figure}[H]
\begin{center}
\includegraphics[keepaspectratio, scale=0.375]{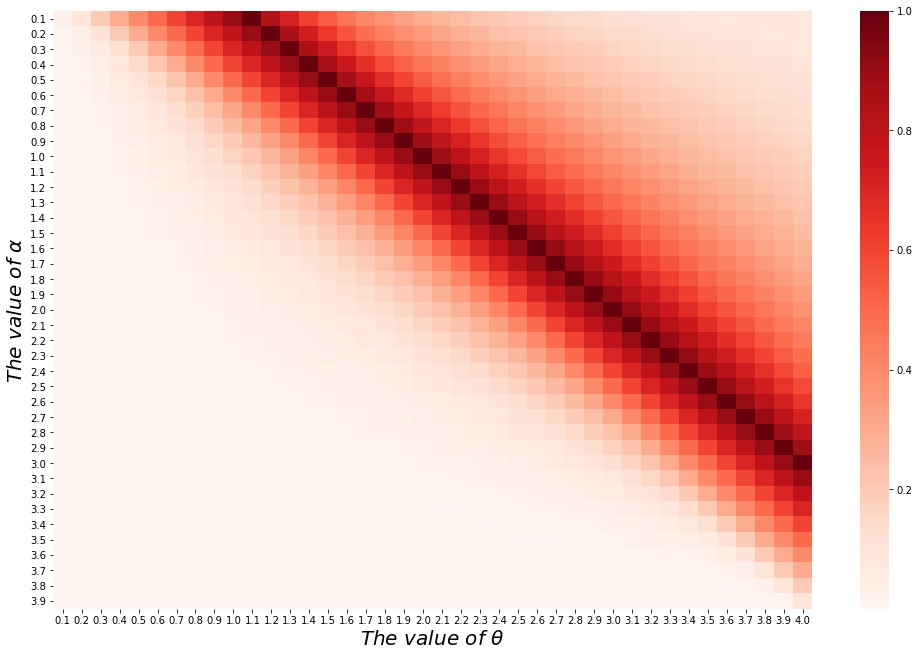}
\end{center}
\caption{Heatmap of the value of $\delta_\infty (\mbox{ker} \bar{\mathcal{A}} (\theta) \cap \mathbb{R}^4_{++})$ when $(\alpha, \beta)$ = $(\alpha, 1-\frac{\alpha}{4})$}
\label{figure: Assumption}
\end{figure}

\begin{figure}[htbp]
    \begin{tabular}{cc}
      \begin{minipage}[t]{0.5\hsize}
        \centering
        \includegraphics[keepaspectratio, scale=0.5825]{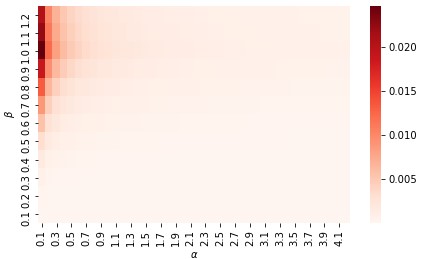}
        \subcaption{$\theta = 0.1$}
        \label{composite}
      \end{minipage} &
      \begin{minipage}[t]{0.5\hsize}
        \centering
        \includegraphics[keepaspectratio, scale=0.5825]{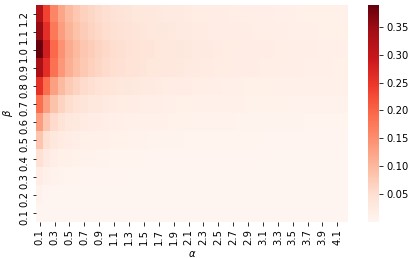}
        \subcaption{$\theta = 0.5$}
        \label{Gradation}
      \end{minipage} \\

      \begin{minipage}[t]{0.5\hsize}
        \centering
        \includegraphics[keepaspectratio, scale=0.5825]{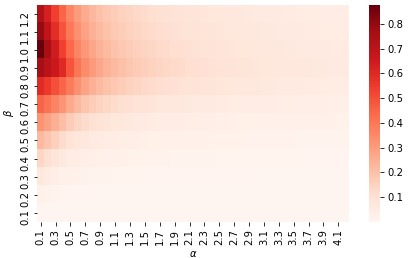}
        \subcaption{$\theta = 1$}
        \label{Gradation}
      \end{minipage} &
   
      \begin{minipage}[t]{0.5\hsize}
        \centering
        \includegraphics[keepaspectratio, scale=0.5825]{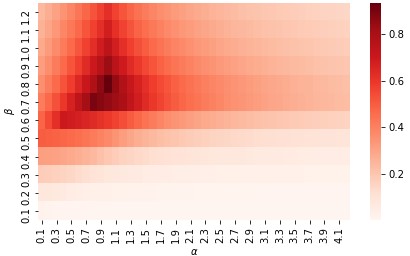}
        \subcaption{$\theta = 2$}
        \label{fill}
      \end{minipage} \\

      \begin{minipage}[t]{0.5\hsize}
        \centering
        \includegraphics[keepaspectratio, scale=0.5825]{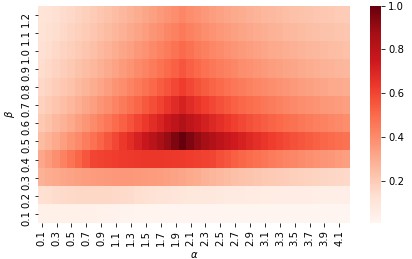}
        \subcaption{$\theta = 3$}
        \label{transform}
      \end{minipage} &
      \begin{minipage}[t]{0.5\hsize}
        \centering
        \includegraphics[keepaspectratio, scale=0.5825]{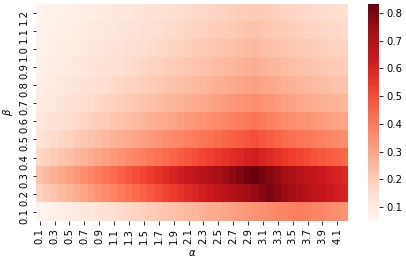}
        \subcaption{$\theta = 4$}
        \label{Gradation}
      \end{minipage} 
    \end{tabular}
     \caption{The values of $\delta_\infty (\mbox{ker} \bar{\mathcal{A}} (\theta) \cap \mathbb{R}^4_{++})$}
     \label{figure: approximate scale}
\end{figure}

\begin{landscape}
\begin{table}
\caption{The values of $\delta_\infty (\mbox{ker} \bar{\mathcal{A}} (\theta) \cap \mathbb{R}^4_{++})$ when $0 < \theta < 4$ and $\frac{\alpha + 4\beta}{4} < \frac{4\beta}{4 - \theta}$.}
\label{Table: case 1}
\begin{center}
\setlength{\tabcolsep}{5pt}
\footnotesize
\begin{tabular}{c|c|c|c|c|c|c|c|c|c|c|c|c|c} \toprule
	&\multicolumn{13}{c}{$\frac{\alpha + 4\beta}{4} < \frac{4\beta}{4 - \theta}$}	\\	\midrule
	&\multicolumn{3}{c|}{$\theta \beta < 1$}	&\multicolumn{7}{c|}{$\frac{\theta - 4}{4} \alpha + \theta \beta \leq 1 < \theta \beta$}	&\multicolumn{3}{c}{$1 < \frac{\theta - 4}{4} \alpha + \theta \beta$}	\\	\midrule
	&\multirow{3}{*}{$1\leq \beta$}	&\multirow{3}{*}{$\beta < 1 < \frac{\alpha+4\beta}{4}$}	&\multirow{3}{*}{$\frac{\alpha+4\beta}{4} < 1$}
	&\multicolumn{3}{c}{$1 \leq \beta$}	&\multicolumn{3}{c|}{$\beta < 1 \leq \frac{\alpha+4\beta}{4}$}	&\multirow{3}{*}{$\frac{\alpha+4\beta}{4} < 1$}
	&\multirow{3}{*}{$1 < \frac{1}{\theta}$}	&\multirow{3}{*}{$\frac{1}{\theta} \leq 1 < \frac{1+\alpha}{\theta}$}	&\multirow{3}{*}{$\frac{1+\alpha}{\theta} \leq 1$}
	\\	\cmidrule(lr){5-10}

	&	&	&	&\multirow{2}{*}{$1 \leq \frac{1}{\theta}$}	&\multicolumn{2}{c|}{$\frac{1}{\theta} < 1$}	&\multirow{2}{*}{$1 \leq \frac{4\beta-1}{4-\theta}$}	&\multicolumn{2}{c|}{$\frac{4\beta-1}{4-\theta} < 1$}	&	&	&	&	\\	\cmidrule(lr){6-7}	\cmidrule(lr){9-10}
	&	&	&	&	&$\theta-\alpha < 1$	&$\theta - \alpha \geq 1$	&	&$\theta - \alpha < 1$	&$\theta - \alpha \geq 1$	&	&	&	&	\\ \hline \hline

${\rm dom} T$	&	&$(\beta, 1]$	&$(\beta, \frac{\alpha+4\beta}{4}]$	&	&	&$\frac{1+\alpha}{\theta}$	&	&$\frac{1+\alpha}{\theta}$	&$[\frac{4\beta-1}{4-\theta}, 1]$	&$[\frac{4\beta-1}{4-\theta}, \frac{\alpha + 4\beta}{4}]$	&	&	&$\frac{1+\alpha}{\theta}$	\\

${\rm dom} R$	&$[L, \theta)$	&$[L, 4\beta - (4-\theta)]$ 	&	&$[L, \theta)$ 	&$[L, 1]$ 	&$1$	&$[L, 1]$ 	&$1$	&$[L, 4\beta - (4-\theta)]$ 	&	&$[L, \theta)$ 	&$[L, 1)$ 	&$1$	\\

$\delta_\infty$	&$G$	&$FG$	&$F$	&$G$	&$G$	&$V$	&$G$	&$V$	&$FG$	&$F$	&$G$	&$G$	&$V$	\\ \bottomrule
\end{tabular}
\begin{tablenotes}
\item[a] $L := \max \{0, \theta - \alpha \}$, \ $F := \displaystyle \max_{\tau \in {\rm dom} T} f(\tau)$, \ $G := \displaystyle \max_{\rho \in {\rm dom} R} g(\rho)$, \ $FG := \max \{ \displaystyle \max_{\tau \in {\rm dom} T} f(\tau),  \displaystyle \max_{\rho \in {\rm dom} R} g(\rho)\}$, $V := \frac{1+\alpha}{4\beta \theta^2}((4-\theta)(1+\alpha)+\theta)$.
\end{tablenotes}
\end{center}
\end{table}

\begin{table}
\caption{The values of $\delta_\infty (\mbox{ker} \bar{\mathcal{A}} (\theta) \cap \mathbb{R}^4_{++})$ when $0 < \theta < 4$ and $\frac{\alpha + 4\beta}{4} \geq \frac{4\beta}{4 - \theta}$.}
\label{Table: case 2}
\begin{center}
\setlength{\tabcolsep}{5pt}
\footnotesize
\begin{tabular}{c|c|c|c|c|c|c|c} \toprule
	&\multicolumn{7}{c}{$\frac{\alpha + 4\beta}{4} \geq \frac{4\beta}{4 - \theta}$}	\\	\midrule
	&\multicolumn{3}{c|}{$\theta \beta \leq 1$}	&\multicolumn{4}{c}{$\theta \beta > 1$}	\\	\midrule
	&$1 \leq \beta$	&$\beta < 1 \leq \frac{4\beta}{4-\theta}$	&$\frac{4\beta}{4-\theta} < 1$
	&$1 \leq \frac{1}{\theta}$	&$\frac{1}{\theta} < 1 \leq \frac{4\beta - 1}{4-\theta}$	&$\frac{4\beta - 1}{4-\theta} < 1 \leq \frac{4\beta}{4 - \theta}$	&$\frac{4\beta}{4 - \theta} < 1$	\\ \hline \hline

${\rm dom} T$	&	&$(\beta, 1]$	&$(\beta, \frac{4\beta}{4-\theta}]$	&	&	&$[\frac{4\beta-1}{4-\theta}, 1]$	&$[\frac{4\beta-1}{4-\theta}, \frac{4\beta}{4-\theta}]$	\\
${\rm dom} R$	&$(0, \theta)$	&$(0, 4\beta - (4-\theta)]$	&	&$(0, \theta)$	&$(0, 1]$	&$(0, 4\beta - (4-\theta)]$	&	\\
$\delta_\infty$	&$\displaystyle \max_{\rho \in {\rm dom} R} g(\rho)$	&$\max \{ \displaystyle \max_{\tau \in {\rm dom} T} f(\tau),  \displaystyle \max_{\rho \in {\rm dom} R} g(\rho)\}$	&$\displaystyle \max_{\tau \in {\rm dom} T} f(\tau)$	&$\displaystyle \max_{\rho \in {\rm dom} R} g(\rho)$	&$\displaystyle \max_{\rho \in {\rm dom} R} g(\rho)$	&$\max \{ \displaystyle \max_{\tau \in {\rm dom} T} f(\tau),  \displaystyle \max_{\rho \in {\rm dom} R} g(\rho)\}$	&$\displaystyle \max_{\tau \in {\rm dom} T} f(\tau)$	\\	\bottomrule
\end{tabular}
\end{center}
\end{table}

\begin{table}
\caption{The values of $\delta_\infty (\mbox{ker} \bar{\mathcal{A}} (\theta) \cap \mathbb{R}^4_{++})$ when $\theta = 4$ and $\frac{\alpha}{4} \geq \beta$.}
\label{Table: case 3}
\begin{center}
\setlength{\tabcolsep}{5pt}
\footnotesize
\begin{tabular}{c|c|c|c|c|c|c|c|c} \toprule
	&\multicolumn{8}{c}{$\frac{\alpha}{4} \geq \beta$}	\\	\midrule
	&\multirow{3}{*}{$1\leq \beta$}	&\multicolumn{2}{c|}{$\beta < 1 \leq \frac{1}{4\alpha}$}	&\multicolumn{3}{c|}{$\frac{1}{4\alpha} < 1 \leq \frac{\alpha+4\beta}{4}$}	&\multicolumn{2}{c}{$\frac{\alpha+4\beta}{4} < 1$}	\\	\cmidrule(lr){3-9}
	&	&\multirow{2}{*}{$4\beta < 1$}	&\multirow{2}{*}{$1\leq 4\beta$}	&\multirow{2}{*}{$4\beta < 1$}	&\multicolumn{2}{c|}{$1 \leq 4\beta$}	&\multirow{2}{*}{$4\beta < 1$}		&\multirow{2}{*}{$1 \leq 4\beta$} \\	\cmidrule(lr){6-7}
	&	&	&	&	&$4-\alpha<1$	&$1\leq 4-\alpha$	&	&	\\ \hline \hline

${\rm dom} T$	&	&	&	&	&	&$\frac{1+\alpha}{4}$	&$\frac{\alpha+4\beta}{4}$	&$\frac{1+\alpha}{4}$	\\
${\rm dom} R$	&$(0, 1]$	&$(0, 4\beta]$	&$(0, 1]$	&$[4 - \alpha, 4\beta]$	&$[4 - \alpha, 1]$	&$1$	&$4\beta$	&$1$	\\
$\delta_\infty$	&$\displaystyle \max_{\rho \in {\rm dom} R} g(\rho)$	&$\displaystyle \max_{\rho \in {\rm dom} R} g(\rho)$	&$\displaystyle \max_{\rho \in {\rm dom} R} g(\rho)$	&$\displaystyle \max_{\rho \in {\rm dom} R} g(\rho)$	&$\displaystyle \max_{\rho \in {\rm dom} R} g(\rho)$	&$\frac{1+\alpha}{16\beta}$	&$\alpha \beta + 4\beta^2$	&$\frac{1+\alpha}{16\beta}$	\\	\bottomrule
\end{tabular}
\end{center}
\end{table}

\begin{table}
\caption{The values of $\delta_\infty (\mbox{ker} \bar{\mathcal{A}} (\theta) \cap \mathbb{R}^4_{++})$ when $\theta = 4$ and $\frac{\alpha}{4} < \beta$.}
\label{Table: case 4}
\begin{center}
\setlength{\tabcolsep}{5pt}
\footnotesize
\begin{tabular}{c|c|c|c|c|c|c|c|c} \toprule
	&\multicolumn{8}{c}{$\frac{\alpha}{4} < \beta$}	\\	\midrule
	&\multirow{2}{*}{$1\leq \frac{\alpha}{4}$}	&\multicolumn{2}{c|}{$\frac{\alpha}{4} < 1 \leq \beta$}	&\multicolumn{3}{c|}{$\beta < 1 \leq \frac{\alpha+4\beta}{4}$}	&\multicolumn{2}{c}{$\frac{\alpha+4\beta}{4}<1$}	\\	\cmidrule(lr){3-9}
	&	&$4-\alpha<1$	&$1\leq 4- \alpha$	&$4\beta < 1$	&$4-\alpha < 1 \leq 4\beta$	&$1 \leq 4-\alpha$	&$4\beta<1$	&$1 \leq 4\beta$	\\ \hline \hline
${\rm dom} T$	&	&	&$\frac{1+\alpha}{4}$	&	&	&$\frac{1+\alpha}{4}$	&$\frac{\alpha+4\beta}{4}$	&$\frac{1+\alpha}{4}$	\\
${\rm dom} R$	&$(0, 1]$	&$[4-\alpha, 1]$	&$1$	&$[4-\alpha, 4\beta]$	&$[4-\alpha, 1]$	&$1$	&$4\beta$	&$1$	\\
$\delta_\infty$	&$\displaystyle \max_{\rho \in {\rm dom} R} g(\rho)$	&$\displaystyle \max_{\rho \in {\rm dom} R} g(\rho)$	&$\frac{1+\alpha}{16\beta}$	&$\displaystyle \max_{\rho \in {\rm dom} R} g(\rho)$	&$\displaystyle \max_{\rho \in {\rm dom} R} g(\rho)$	&$\frac{1+\alpha}{16\beta}$	&$\alpha \beta + 4\beta^2$	&$\frac{1+\alpha}{16\beta}$	\\	\bottomrule
\end{tabular}
\end{center}
\end{table}

\end{landscape}

%
%
\section{Detailed settings of the proposed algorithm}
\label{Appendix C}

\subsection{Practical termination conditions}
\label{Appendix C1}
\subsubsection{Termination condition to deal with cases suffering from numerically unstable outputs}
\label{Appendix C1-1}
The projection and rescaling algorithms can return incorrect outputs due to numerical errors caused by scaling operations.
Such outputs prevent Algorithms \ref{Practical p alg} and \ref{Practical d alg} from working correctly.
Thus, we added the practical termination conditions to our algorithms.
In our numerical experiments, Algorithms \ref{Practical p alg} and \ref{Practical d alg} were set to terminate when 30 consecutive incorrect outputs were returned from the projection and rescaling method.
We defined Algorithms \ref{Practical p alg} and \ref{Practical d alg} as obtaining incorrect outputs from the projection and rescaling method if they obtain any output except the following:
\begin{itemize}
\item a reducing direction for (P) or (D),
\item an improving ray of (P) or (D),
\item vectors $(y, \gamma)$ that satisfies $\gamma > 0$ and $\lambda_{\min} (c + \mathcal{A}^* \frac{1}{\gamma} y) \geq$ -1e-4, 
\item vectors $(x,\tau,\rho)$ that satisfies $\tau > 0$, $\|\mathcal{A} \frac{1}{\tau} x - b\| \leq$ 1e-4 and $\lambda_{\min} (\frac{1}{\tau} x) \geq$ -1e-4, or
\item a certificate that there is no $\varepsilon$-feasible solution to the corresponding feasibility problem, i.e., $\mbox{FP} ( \mbox{ker} \mathcal{A}(\theta^k), \mbox{int} \bar{\mathcal{K}})$ or $\mbox{FP} (\mbox{range} {\mathcal{A}(\theta^k)}^*, \mbox{int} \bar{\mathcal{K}})$.
\end{itemize}

\subsubsection{Termination condition to deal with cases suffering from non-useful outputs}
\label{Appendix C1-2}
The projection and rescaling algorithms can return non-useful outputs for Algorithms \ref{Practical p alg} and \ref{Practical d alg}.
For example, if the projection and rescaling method called in Algorithm \ref{Practical p alg} returns a certificate that there is no $\varepsilon$-feasible solution to the input feasibility problem $\mbox{FP} ( \mbox{ker} \mathcal{A}(\theta^k), \mbox{int} \bar{\mathcal{K}})$, all Algorithm \ref{Practical p alg} can do is update the value of $LB$.
Since such output is likely to be obtained when (P) is not strongly feasible, we considered it prudent to terminate Algorithm \ref{Practical p alg} when the projection and rescaling method begin to return such outputs in succession.
Thus, in the numerical experiments, Algorithms \ref{Practical p alg} and \ref{Practical d alg} were set to terminate when the projection and rescaling algorithms proved 30 consecutive times that there was no $\varepsilon$-feasible solution to the corresponding feasibility problem. 
Note that in our experiments, Algorithms \ref{Practical p alg} and \ref{Practical d alg} did not terminate with this termination condition.

\subsection{Implementation details of the projection and rescaling algorithm}
Algorithm \ref{postpro alg} used the projection and rescaling algorithm proposed in \cite{Kanoh2023}, which employs the smooth perceptron scheme \cite{Soheili2012, Soheili2013} in the basic procedure. 
The reason for using this projection and rescaling method is that the numerical experiments in \cite{Kanoh2023} show that their method obtained accurate approximate solutions in a shorter time than the other methods. 
We set the termination parameter as $\xi = 1/4$ in the basic procedure for the same reasons stated in \cite{Kanoh2023}.
In addition, we set the accuracy parameter as $\varepsilon =$ 1e-16 in the main algorithm, which allows the projection and rescaling algorithm to obtain approximate optimal solutions for (P) and (D) near the boundaries of the cone $\mathcal{K}$.

In Algorithm \ref{postpro alg}, the projections $\mathcal{P}_{{\rm ker}\mathcal{A}(\theta)}$ and $\mathcal{P}_{{\rm range} \mathcal{A}(\theta)^*}$ were computed in the same way as described in \cite{Kanoh2023}. 
That is, we computed the projections $\mathcal{P}_{{\rm ker}\mathcal{A}(\theta)}$ and $\mathcal{P}_{{\rm range} \mathcal{A}(\theta)^*}$ using the singular value decomposition. 
Let $A \in \mathbb{R}^{m+1 \times d+2}$ be a matrix representing the linear operator $\mathcal{A}(\theta)$ and $I$ be the identity matrix.
Suppose that the singular value decomposition of a matrix $A$ is given by 
$A = U \Sigma V^\top = U
\begin{pmatrix}
\Sigma_{m+1} \ O
\end{pmatrix}
V^\top
$ where $U \in \mathbb{R}^{m+1 \times m+1}$ and $V \in \mathbb{R}^{d+2 \times d+2}$ are orthogonal matrices, and $\Sigma_{m+1} \in \mathbb{R}^{m+1 \times m+1}$ is a diagonal matrix with $m+1$ singular values on the diagonal.
Since the projections $\mathcal{P}_{{\rm ker}\mathcal{A}(\theta)}$ and $\mathcal{P}_{{\rm range} \mathcal{A}(\theta)^*}$ are given by $\mathcal{P}_{{\rm ker}\mathcal{A}(\theta)} = I - A^\top (AA^\top)^{-1} A$ and $\mathcal{P}_{{\rm range} \mathcal{A}(\theta)^*} = A^\top (AA^\top)^{-1} A$, respectively, we can compute these projections as follows: 
\begin{align*}
\mathcal{P}_{{\rm range} \mathcal{A}(\theta)^*}
&= A^\top (AA^\top)^{-1} A 	\\
&= A^\top (U \Sigma \Sigma^\top U^\top)^{-1} A \\
&= A^\top U^{-\top} (\Sigma_{m+1}^2)^{-1} U^{-1} A \\
&= V \Sigma^\top \Sigma_{m+1}^{-2} \Sigma V^\top 
= V
\begin{pmatrix}
I_{m+1} & O \\
O & O
\end{pmatrix}
V^\top = V_{:, 1:m+1} V_{:,1:m+1}^\top,
\end{align*}
and $\mathcal{P}_{{\rm ker}\mathcal{A}(\theta)} = I - \mathcal{P}_{{\rm range} \mathcal{A}(\theta)^*} = I - V_{:, 1:m+1} V_{:,1:m+1}^\top$, where $V_{:, 1:m+1}$ represents the submatrix from column $1$ to column $m+1$ of $V$.

\subsection{How to compute $(y, \gamma)$ from $(z, \omega, \kappa) \in \mbox{range} \mathcal{A}(\theta)^* \cap \bar{\mathcal{K}}$}
If a nonzero point $(z, \omega, \kappa) \in \mbox{range} \mathcal{A}(\theta)^* \cap \bar{\mathcal{K}}$ is obtained from the projection and rescaling methods in Algorithms \ref{Practical p alg} and \ref{Practical d alg}, we compute $(y, \gamma) \in \mathbb{R}^{m+1}$ such that $z = \mathcal{A}^*y + \gamma c$, $\omega = -b^\top y - \gamma \theta$ and $\kappa = \gamma$.
The matrices $U$, $\Sigma$, and $V$ described in the previous section are also used to compute such a $(y, \gamma)$. 
Suppose that $A \in \mathbb{R}^{m+1 \times d+2}$ is a matrix representation of the linear operator $\mathcal{A}(\theta)$ and $A$ is decomposed into $A = U \Sigma V^\top$ by the singular value decomposition.
Then, we find 
\begin{equation}
\notag
\begin{pmatrix}
z\\
\omega \\
\kappa
\end{pmatrix}
=
A^\top
\begin{pmatrix}
y \\
\gamma
\end{pmatrix}
= V
\begin{pmatrix}
\Sigma_{m+1} \ O
\end{pmatrix}^\top U^\top
\begin{pmatrix}
y \\
\gamma
\end{pmatrix}
.
\end{equation}
Since $U$ and $V$ are orthogonal matrices, we have
\begin{equation}
\notag
\begin{pmatrix}
y \\
\gamma
\end{pmatrix}
=
U
\begin{pmatrix}
\Sigma_{m+1}^{-1} \ O
\end{pmatrix}
V^\top
\begin{pmatrix}
z\\
\omega \\
\kappa
\end{pmatrix}
.
\end{equation}
\subsection{Modification of the basic procedure}
The basic procedure of \cite{Kanoh2023} requires a constant $\xi \in \mathbb{R}$ such that $0<\xi<1$ as an input.
This procedure finds a Jordan frame $\{c_1, c_2, \dots, c_r \}$ such that $\langle c_i, x \rangle \leq \xi$ holds for any feasible solution $x$ of the input problem $\mbox{FP}_{S_\infty} (\mathcal{L}, \mbox{int } \mathcal{K})$ and for some $i \in \{1, 2, \dots, r\}$ in at most $\frac{p^2 r_{max}^2}{\xi^2}$ iterations, where $\mathcal{K}$ is a Cartesian product of $p$ simple symmetric cones $\mathcal{K}_1, \dots, \mathcal{K}_p$, $r = \sum_{i=1}^p r_i$ is a rank of $\mathcal{K}$ and $r_{max} = \max \{ r_1, \dots, r_p \}$.
Whether $\langle c_i, x \rangle \leq \xi$ is valid is determined by calculating the upper bound $u_i$ of $\langle c_i, x \rangle$. 
If such a Jordan frame $\{c_1, c_2, \dots, c_r \}$ is obtained, the basic procedure terminates and returns the sets $C := \{c_1, c_2, \dots, c_r \}$ and $H := \{ i : u_i \leq \xi \}$ to the main algorithm.   
Then, the main algorithm scales the problem $\mbox{FP}_{S_\infty} (\mathcal{L}, \mbox{int } \mathcal{K})$ as $\mbox{FP}_{S_\infty} (Q_v (\mathcal{L}), \mbox{int } \mathcal{K})$, where $v = \frac{1}{\sqrt{\xi}}  \sum_{h \in H}  c_h + \sum_{h \notin H} c_h$.
In this case, however, the main algorithm can scale the problem more efficiently using $v = \sum_{h \in H} \frac{1}{\sqrt{u_h}} c_h + \sum_{h \notin H} c_h$, which will reduce the computational time of Algorithms \ref{Practical p alg} and \ref{Practical d alg}.
Furthermore, if the same constant $\xi$ is used for all feasibility problems, Algorithms \ref{Practical p alg} and \ref{Practical d alg} might encounter problems where the basic procedure requires many iterations, close to the maximum number of iterations $\frac{p^2 r_{max}^2}{\xi^2}$.
Therefore, we modified the projection and rescaling algorithm of \cite{Kanoh2023} called in Algorithms \ref{Practical p alg} and \ref{Practical d alg} as follows:
\begin{itemize}
\item Modification of the basic procedure
\begin{itemize}
\item Let $k$ be the number of iterations of the basic procedure.
The termination condition varies depending on the value of $k$. 
\begin{enumerate}
\item If $k \leq 100$, the basic procedure terminates when a Jordan frame  $\{c_1, c_2, \dots, c_r \}$ such that $\langle c_i, x \rangle \leq u_i \leq \xi$ holds for any feasible solution $x$ of the input problem $\mbox{FP}_{S_\infty} (\mathcal{L}, \mbox{int } \mathcal{K})$ and for some $i \in \{1, 2, \dots, r\}$ is obtained. Then, return $C := \{c_1, c_2, \dots, c_r \}$, $H := \{ i : u_i \leq \xi \}$ and $U := \{ u_i : u_i \leq \xi \}$ to the main algorithm. 
\item If $k > 100$, the basic procedure terminates when a Jordan frame  $\{c_1, c_2, \dots, c_r \}$ such that $\langle c_i, x \rangle \leq u_i < 1$ holds for any feasible solution $x$ of the input problem $\mbox{FP}_{S_\infty} (\mathcal{L}, \mbox{int } \mathcal{K})$ and for some $i \in \{1, 2, \dots, r\}$ is obtained. Then, return $C := \{c_1, c_2, \dots, c_r \}$, $H := \{ i : u_i < 1 \}$ and $U := \{ u_i : u_i < 1\}$ to the main algorithm. 
\end{enumerate}
\end{itemize}
\item Modification of the main algorithm
\begin{enumerate}
\item Suppose that the basic procedure returns the sets $C$, $H$ and $U$ at the $k$-th iteration of the main algorithm.
Then, scale the linear subspace $\mathcal{L}^k$ as $L^{k+1} \leftarrow Q_v(L^k)$, where $v = \sum_{h \in H} \frac{1}{\sqrt{u_h}} c_h + \sum_{h \notin H} c_h$. 
\end{enumerate}
\end{itemize}

\section{Computational results on SDPLIB}
The detailed results of the experiments conducted in Section \ref{sec: numerical results post-pro} are summarized in Tables \ref{Table: MosekDef-Well}-\ref{Table: ProSDPT3-ill}. 
Please refer to the respective solver's documentation for the meaning of the messages returned by each solver in Tables \ref{Table: MosekDef-Well}-\ref{Table: SDPT3Tight-ill}. 
Tables \ref{Table: ProMosek-Well}-\ref{Table: ProSDPT3-ill} summarize the results of Algorithm \ref{postpro alg}. 
The columns ``Algorithm \ref{Practical d alg}" and ``Algorithm \ref{Practical p alg}" in Tables \ref{Table: ProMosek-Well}-\ref{Table: ProSDPT3-ill} summarize the execution time of each algorithm and what termination conditions they met to finish. 
``Complete" means that the algorithm terminated because $UB-LB \leq \theta_{acc}$ was satisfied, ``Time Over" means that the algorithm terminated because its execution time exceeded 30 minutes, and ``Numerical Error" means that the algorithm terminated because the termination condition defined in Appendix \ref{Appendix C1} was met. 
\begin{landscape}
\subsection{Results of Mosek}
\begin{table}[H]
\caption{Results of Mosek using the default setting on well-posed group instances}
\label{Table: MosekDef-Well}
\begin{center}
\setlength{\tabcolsep}{5pt}
\footnotesize
\begin{tabular}{ccccccccccc} \toprule
Problem	&${\rm err_1}$&${\rm err_2}$	&${\rm err_3}$	&${\rm err_4}$	&${\rm err_5}$	&${\rm err_6}$	&time(s)	&Message	&Problem status	&Solution status \\	\midrule
arch0   &1.55e-05   &1.15e-06   &5.77e-13   &0.00e+00   &-5.08e-09   &4.80e-07   &3.03e-01   &MSK RES OK   &PRIMAL AND DUAL FEASIBLE   &OPTIMAL  \\
arch2   &8.99e-06   &7.04e-07   &4.11e-13   &0.00e+00   &-4.05e-09   &3.31e-07   &2.84e-01   &MSK RES OK   &PRIMAL AND DUAL FEASIBLE   &OPTIMAL  \\
arch4   &6.34e-06   &5.36e-07   &2.69e-08   &0.00e+00   &-3.29e-09   &3.31e-07   &2.84e-01   &MSK RES OK   &PRIMAL AND DUAL FEASIBLE   &OPTIMAL  \\
arch8   &2.65e-06   &2.37e-07   &5.13e-12   &0.00e+00   &-1.60e-09   &2.32e-07   &3.38e-01   &MSK RES OK   &PRIMAL AND DUAL FEASIBLE   &OPTIMAL  \\
control1   &2.12e-08   &0.00e+00   &2.62e-10   &0.00e+00   &-1.36e-09   &2.06e-08   &1.26e-02   &MSK RES OK   &PRIMAL AND DUAL FEASIBLE   &OPTIMAL  \\
control2   &1.19e-07   &0.00e+00   &4.57e-10   &0.00e+00   &-3.89e-09   &1.65e-07   &3.00e-02   &MSK RES OK   &PRIMAL AND DUAL FEASIBLE   &OPTIMAL  \\
control3   &7.90e-08   &0.00e+00   &1.64e-09   &0.00e+00   &-2.30e-09   &1.37e-07   &7.52e-02   &MSK RES OK   &PRIMAL AND DUAL FEASIBLE   &OPTIMAL  \\
\bottomrule
\end{tabular}
\end{center}
\end{table}
\end{landscape}

\begin{landscape}
\begin{table}[H]
\caption{Results of Mosek using the default setting on ill-posed group instances}
\label{Table: MosekDef-ill}
\begin{center}
\setlength{\tabcolsep}{5pt}
\footnotesize
\begin{tabular}{ccccccccccc} \toprule
Problem	&${\rm err_1}$&${\rm err_2}$	&${\rm err_3}$	&${\rm err_4}$	&${\rm err_5}$	&${\rm err_6}$	&time(s)	&Message	&Problem status	&Solution status \\	\midrule
gpp100   &2.16e-08   &0.00e+00   &2.75e-04   &0.00e+00   &-9.09e-11   &7.95e-08   &2.75e-01   &MSK RES TRM STALL   &PRIMAL AND DUAL FEASIBLE   &OPTIMAL  \\
gpp124-1   &1.40e-08   &0.00e+00   &7.13e-12   &0.00e+00   &-8.63e-10   &2.70e-07   &3.01e-01   &MSK RES OK    &PRIMAL AND DUAL FEASIBLE   &OPTIMAL  \\
gpp124-2   &5.31e-06   &0.00e+00   &4.63e-07   &0.00e+00   &6.95e-09   &8.42e-08   &3.28e-01   &MSK RES TRM STALL   &PRIMAL AND DUAL FEASIBLE   &OPTIMAL  \\
gpp124-3   &1.62e-06   &0.00e+00   &1.12e-05   &0.00e+00   &6.00e-09   &5.15e-08   &3.34e-01   &MSK RES TRM STALL   &PRIMAL AND DUAL FEASIBLE   &OPTIMAL  \\
gpp124-4   &2.23e-07   &0.00e+00   &1.82e-04   &0.00e+00   &-1.60e-10   &3.30e-08   &4.92e-01   &MSK RES TRM STALL   &PRIMAL AND DUAL FEASIBLE   &OPTIMAL  \\
hinf1   &2.12e-08   &0.00e+00   &4.13e-10   &0.00e+00   &-1.25e-08   &1.35e-07   &3.32e-02   &MSK RES TRM STALL   &PRIMAL AND DUAL FEASIBLE   &OPTIMAL  \\
hinf3   &9.39e-11   &0.00e+00   &2.09e-10   &0.00e+00   &-1.73e-08   &3.18e-07   &2.81e-02   &MSK RES TRM STALL   &PRIMAL AND DUAL FEASIBLE   &OPTIMAL  \\
\bottomrule
\end{tabular}
\end{center}
\end{table}
\end{landscape}

\begin{landscape}
\begin{table}[H]
\caption{Results of Mosek using the tight setting on well-posed group instances}
\label{Table: MosekTight-Well}
\begin{center}
\setlength{\tabcolsep}{3.5pt}
\footnotesize
\begin{tabular}{ccccccccccc} \toprule
Problem	&${\rm err_1}$&${\rm err_2}$	&${\rm err_3}$	&${\rm err_4}$	&${\rm err_5}$	&${\rm err_6}$	&time(s)	&Message	&Problem status	&Solution status \\	\midrule
arch0   &4.05e-08   &2.25e-09   &5.75e-13   &0.00e+00   &-9.65e-12   &9.30e-10   &6.24e-01   &MSK RES TRM STALL   &PRIMAL AND DUAL FEASIBLE   &OPTIMAL  \\
arch2   &7.79e-09   &3.70e-12   &4.03e-13   &0.00e+00   &-2.11e-14   &1.73e-12   &3.19e-01   &MSK RES OK    &PRIMAL AND DUAL FEASIBLE   &OPTIMAL  \\
arch4   &1.10e-08   &2.52e-10   &6.95e-13   &0.00e+00   &-1.55e-12   &1.56e-10   &5.38e-01   &MSK RES TRM STALL   &PRIMAL AND DUAL FEASIBLE   &OPTIMAL  \\
arch8   &2.14e-07   &1.70e-10   &5.04e-12   &0.00e+00   &-1.14e-12   &1.66e-10   &5.61e-01   &MSK RES TRM STALL   &PRIMAL AND DUAL FEASIBLE   &OPTIMAL  \\
control1   &2.56e-11   &0.00e+00   &2.85e-10   &2.73e-11   &-1.12e-13   &1.68e-12   &3.28e-02   &MSK RES TRM STALL   &PRIMAL AND DUAL FEASIBLE   &OPTIMAL  \\
control2   &1.98e-08   &0.00e+00   &1.41e-03   &0.00e+00   &-5.33e-10   &2.12e-08   &5.25e-02   &MSK RES TRM STALL   &UNKNOWN   &UNKNOWN  \\
control3   &2.14e-07   &0.00e+00   &2.74e-03   &0.00e+00   &-3.63e-09   &2.05e-07   &1.24e-01   &MSK RES TRM STALL   &UNKNOWN   &UNKNOWN  \\
\bottomrule
\end{tabular}
\end{center}
\end{table}
\end{landscape}

\begin{landscape}
\begin{table}[H]
\caption{Results of Mosek using the tight setting on ill-posed group instances}
\label{Table: MosekTight-ill}
\begin{center}
\setlength{\tabcolsep}{5pt}
\footnotesize
\begin{tabular}{ccccccccccc} \toprule
Problem	&${\rm err_1}$&${\rm err_2}$	&${\rm err_3}$	&${\rm err_4}$	&${\rm err_5}$	&${\rm err_6}$	&time(s)	&Message	&Problem status	&Solution status \\	\midrule
gpp100   &2.16e-08   &0.00e+00   &2.75e-04   &0.00e+00   &-9.09e-11   &7.95e-08   &2.71e-01   &MSK RES TRM STALL   &UNKNOWN   &UNKNOWN  \\
gpp124-1   &8.96e-10   &0.00e+00   &3.98e-11   &0.00e+00   &-3.45e-10   &5.92e-08   &5.69e-01   &MSK RES TRM STALL   &PRIMAL AND DUAL FEASIBLE   &OPTIMAL  \\
gpp124-2   &5.31e-06   &0.00e+00   &4.63e-07   &0.00e+00   &6.95e-09   &8.42e-08   &3.25e-01   &MSK RES TRM STALL   &UNKNOWN   &UNKNOWN  \\
gpp124-3   &1.62e-06   &0.00e+00   &1.12e-05   &0.00e+00   &6.00e-09   &5.15e-08   &3.27e-01   &MSK RES TRM STALL   &UNKNOWN   &UNKNOWN  \\
gpp124-4   &2.23e-07   &0.00e+00   &1.82e-04   &0.00e+00   &-1.60e-10   &3.30e-08   &4.92e-01   &MSK RES TRM STALL   &UNKNOWN   &UNKNOWN  \\
hinf1   &1.17e-11   &0.00e+00   &2.39e-01   &0.00e+00   &-2.48e-08   &2.57e-07   &2.41e-02   &MSK RES TRM STALL   &UNKNOWN   &UNKNOWN  \\
hinf3   &1.01e-10   &0.00e+00   &6.49e+00   &0.00e+00   &-1.81e-08   &3.39e-07   &2.55e-02   &MSK RES TRM STALL   &UNKNOWN   &UNKNOWN  \\
\bottomrule
\end{tabular}
\end{center}
\end{table}
\end{landscape}

\begin{landscape}
\subsection{Results of SDPA}
\begin{table}[H]
\caption{Results of SDPA using the default setting on well-posed group instances}
\label{Table: SDPADef-Well}
\begin{center}
\setlength{\tabcolsep}{5pt}
\footnotesize
\begin{tabular}{ccccccccc} \toprule
Problem	&${\rm err_1}$&${\rm err_2}$	&${\rm err_3}$	&${\rm err_4}$	&${\rm err_5}$	&${\rm err_6}$	&time(s)	&Message	\\	\midrule
arch0   &5.01e-08   &0.00e+00   &6.41e-13   &0.00e+00   &1.18e-07   &1.18e-07   &5.31e+00   &pdFEAS  \\
arch2   &1.10e-08   &0.00e+00   &4.82e-13   &0.00e+00   &1.96e-07   &1.96e-07   &4.82e+00   &pdFEAS  \\
arch4   &4.28e-09   &0.00e+00   &4.50e-13   &0.00e+00   &5.96e-08   &5.96e-08   &4.81e+00   &pdFEAS  \\
arch8   &1.01e-07   &0.00e+00   &1.30e-12   &0.00e+00   &6.64e-09   &5.59e-09   &5.33e+00   &pFEAS  \\
control1   &1.67e-09   &0.00e+00   &7.38e-11   &0.00e+00   &2.62e-08   &2.59e-08   &6.63e+00   &pdOPT  \\
control2   &9.98e-09   &0.00e+00   &8.19e-11   &0.00e+00   &2.72e-09   &8.35e-08   &7.39e+00   &pdOPT  \\
control3   &1.31e-07   &0.00e+00   &2.52e-10   &0.00e+00   &1.66e-07   &1.81e-07   &9.76e+00   &pFEAS  \\
\bottomrule
\end{tabular}
\end{center}
\end{table}

\begin{table}[H]
\caption{Results of SDPA using the default setting on ill-posed group instances}
\label{Table: SDPADef-ill}
\begin{center}
\setlength{\tabcolsep}{5pt}
\footnotesize
\begin{tabular}{ccccccccc} \toprule
Problem	&${\rm err_1}$&${\rm err_2}$	&${\rm err_3}$	&${\rm err_4}$	&${\rm err_5}$	&${\rm err_6}$	&time(s)	&Message	\\	\midrule
gpp100   &1.19e-07   &0.00e+00   &2.19e-12   &0.00e+00   &4.35e-09   &8.81e-09   &3.24e+00   &pdOPT  \\
gpp124-1   &3.06e-07   &0.00e+00   &6.81e-12   &0.00e+00   &1.20e-08   &1.52e-08   &4.27e+00   &pFEAS  \\
gpp124-2   &3.48e-08   &0.00e+00   &7.06e-12   &0.00e+00   &4.87e-08   &5.01e-08   &3.37e+00   &pdOPT  \\
gpp124-3   &1.72e-08   &0.00e+00   &2.33e-12   &0.00e+00   &3.91e-08   &4.05e-08   &3.27e+00   &pdOPT  \\
gpp124-4   &1.28e-11   &0.00e+00   &8.30e-12   &0.00e+00   &3.97e-08   &4.10e-08   &3.61e+00   &pdOPT  \\
hinf1   &2.32e-09   &0.00e+00   &3.63e-12   &0.00e+00   &1.72e-07   &3.19e-06   &5.88e+00   &pdFEAS  \\
hinf3   &1.22e-06   &0.00e+00   &4.87e-13   &0.00e+00   &1.29e-04   &1.84e-04   &4.19e+00   &pFEAS  \\
\bottomrule
\end{tabular}
\end{center}
\end{table}

\begin{table}[H]
\caption{Results of SDPA using the tight setting on well-posed group instances}
\label{Table: SDPATight-Well}
\begin{center}
\setlength{\tabcolsep}{5pt}
\footnotesize
\begin{tabular}{ccccccccc} \toprule
Problem	&${\rm err_1}$&${\rm err_2}$	&${\rm err_3}$	&${\rm err_4}$	&${\rm err_5}$	&${\rm err_6}$	&time(s)	&Message	\\	\midrule
arch0   &2.13e-08   &0.00e+00   &6.39e-13   &0.00e+00   &9.40e-10   &9.76e-10   &5.75e+00   &pFEAS  \\
arch2   &1.39e-09   &0.00e+00   &4.88e-13   &0.00e+00   &2.39e-10   &2.32e-10   &5.93e+00   &pFEAS  \\
arch4   &3.05e-09   &0.00e+00   &4.42e-13   &0.00e+00   &2.41e-10   &2.13e-10   &5.95e+00   &pFEAS  \\
arch8   &1.01e-07   &0.00e+00   &1.30e-12   &0.00e+00   &6.64e-09   &5.59e-09   &5.37e+00   &pFEAS  \\
control1   &6.78e+01   &0.00e+00   &1.17e-12   &0.00e+00   &-1.18e-01   &1.36e-01   &1.91e+00   &pdINF  \\
control2   &1.16e+02   &0.00e+00   &1.51e-12   &0.00e+00   &-5.49e-02   &2.30e-01   &2.24e+00   &pdINF  \\
control3   &2.80e+02   &0.00e+00   &2.79e-12   &0.00e+00   &-8.78e-02   &3.45e-01   &2.69e+00   &pdINF  \\
\bottomrule
\end{tabular}
\end{center}
\end{table}

\begin{table}[H]
\caption{Results of SDPA using the tight setting on ill-posed group instances}
\label{Table: SDPATight-ill}
\begin{center}
\setlength{\tabcolsep}{5pt}
\footnotesize
\begin{tabular}{ccccccccc} \toprule
Problem	&${\rm err_1}$&${\rm err_2}$	&${\rm err_3}$	&${\rm err_4}$	&${\rm err_5}$	&${\rm err_6}$	&time(s)	&Message	\\	\midrule
gpp100   &1.29e-06   &0.00e+00   &2.60e-12   &0.00e+00   &1.74e-05   &1.74e-05   &2.37e+00   &pdINF  \\
gpp124-1   &4.18e-06   &0.00e+00   &6.10e-12   &0.00e+00   &6.72e-05   &6.74e-05   &2.90e+00   &pdINF  \\
gpp124-2   &1.25e-05   &0.00e+00   &8.94e-12   &0.00e+00   &1.28e-05   &1.28e-05   &2.83e+00   &pdINF  \\
gpp124-3   &9.50e-07   &0.00e+00   &1.58e-12   &0.00e+00   &7.03e-06   &7.06e-06   &2.63e+00   &pdINF  \\
gpp124-4   &1.39e-06   &0.00e+00   &1.00e-12   &0.00e+00   &5.53e-06   &5.54e-06   &2.71e+00   &pdINF  \\
hinf1   &8.31e-09   &0.00e+00   &1.46e-12   &0.00e+00   &6.23e-05   &8.69e-05   &4.20e+00   &pdINF  \\
hinf3   &7.82e-07   &0.00e+00   &4.57e-13   &0.00e+00   &1.29e-05   &4.98e-05   &4.55e+00   &noINFO  \\
\bottomrule
\end{tabular}
\end{center}
\end{table}
\end{landscape}

\begin{landscape}
\subsection{Results of SDPT3}
\begin{table}[H]
\caption{Results of SDPT3 using the default setting on well-posed group instances}
\label{Table: SDPT3Def-Well}
\begin{center}
\setlength{\tabcolsep}{5pt}
\footnotesize
\begin{tabular}{cccccccccc} \toprule
Problem	&${\rm err_1}$&${\rm err_2}$	&${\rm err_3}$	&${\rm err_4}$	&${\rm err_5}$	&${\rm err_6}$	&time(s)	&Message	\\	\midrule
arch0   &3.37e-08   &0.00e+00   &3.50e-11   &0.00e+00   &1.56e-09   &1.59e-09   &9.76e-01   &stop: max(relative gap, infeasibilities) $<$ 1.00e-08  \\
arch2   &3.57e-09   &0.00e+00   &8.81e-11   &0.00e+00   &3.66e-09   &4.02e-09   &7.57e-01   &stop: max(relative gap, infeasibilities) $<$ 1.00e-08  \\
arch4   &4.97e-09   &0.00e+00   &5.38e-11   &0.00e+00   &3.57e-09   &3.84e-09   &6.99e-01   &stop: max(relative gap, infeasibilities) $<$ 1.00e-08  \\
arch8   &6.30e-08   &0.00e+00   &9.55e-11   &0.00e+00   &2.23e-09   &1.83e-09   &8.28e-01   &stop: steps too short consecutively  \\
control1   &4.56e-09   &0.00e+00   &3.52e-11   &0.00e+00   &-3.42e-10   &1.72e-09   &2.71e-01   &stop: max(relative gap, infeasibilities) $<$ 1.00e-08  \\
control2   &5.26e-09   &0.00e+00   &6.48e-11   &0.00e+00   &2.40e-08   &3.39e-08   &2.20e-01   &stop: primal infeas has deteriorated too much  \\
control3   &8.15e-08   &0.00e+00   &2.14e-10   &0.00e+00   &8.82e-08   &1.26e-07   &5.61e-01   &lack of progress in infeas  \\
\bottomrule
\end{tabular}
\end{center}
\end{table}

\begin{table}[H]
\caption{Results of SDPT3 using the default setting on ill-posed group instances}
\label{Table: SDPT3Def-ill}
\begin{center}
\setlength{\tabcolsep}{5pt}
\footnotesize
\begin{tabular}{ccccccccc} \toprule
Problem	&${\rm err_1}$&${\rm err_2}$	&${\rm err_3}$	&${\rm err_4}$	&${\rm err_5}$	&${\rm err_6}$	&time(s)	&Message	\\	\midrule
gpp100   &1.51e-09   &0.00e+00   &2.87e-10   &0.00e+00   &-1.09e-08   &9.81e-09   &2.83e-01   &stop: max(relative gap, infeasibilities) $<$ 1.00e-08  \\
gpp124-1   &2.09e-10   &0.00e+00   &2.59e-11   &0.00e+00   &-3.47e-08   &8.64e-10   &2.51e-01   &stop: max(relative gap, infeasibilities) $<$ 1.00e-08  \\
gpp124-2   &6.68e-10   &0.00e+00   &1.05e-10   &0.00e+00   &-6.28e-09   &6.17e-09   &1.91e-01   &stop: max(relative gap, infeasibilities) $<$ 1.00e-08  \\
gpp124-3   &2.53e-09   &0.00e+00   &4.51e-10   &0.00e+00   &-5.01e-09   &4.15e-09   &1.78e-01   &stop: max(relative gap, infeasibilities) $<$ 1.00e-08  \\
gpp124-4   &3.40e-09   &0.00e+00   &7.32e-10   &0.00e+00   &-1.80e-08   &9.08e-10   &1.88e-01   &stop: max(relative gap, infeasibilities) $<$ 1.00e-08  \\
hinf1   &5.35e-08   &0.00e+00   &4.83e-09   &0.00e+00   &-1.02e-05   &1.45e-08   &2.96e-01   &stop: relative gap $<$ infeasibility  \\
hinf3   &9.02e-06   &0.00e+00   &8.16e-12   &0.00e+00   &-1.18e-04   &2.00e-09   &1.15e-01   &stop: relative gap $<$ infeasibility  \\
\bottomrule
\end{tabular}
\end{center}
\end{table}

\begin{table}[H]
\caption{Results of SDPT3 using the tight setting on well-posed group instances}
\label{Table: SDPT3Tight-Well}
\begin{center}
\setlength{\tabcolsep}{5pt}
\footnotesize
\begin{tabular}{ccccccccc} \toprule
Problem	&${\rm err_1}$&${\rm err_2}$	&${\rm err_3}$	&${\rm err_4}$	&${\rm err_5}$	&${\rm err_6}$	&time(s)	&Message	\\	\midrule
arch0   &2.85e-08   &0.00e+00   &1.76e-11   &0.00e+00   &6.76e-10   &8.01e-10   &8.76e-01   &lack of progress in infeas  \\
arch2   &2.60e-09   &0.00e+00   &4.71e-12   &0.00e+00   &1.96e-10   &2.15e-10   &8.51e-01   &lack of progress in infeas  \\
arch4   &1.81e-09   &0.00e+00   &3.96e-12   &0.00e+00   &1.49e-10   &1.58e-10   &9.22e-01   &lack of progress in infeas  \\
arch8   &6.30e-08   &0.00e+00   &9.55e-11   &0.00e+00   &2.23e-09   &1.83e-09   &8.35e-01   &lack of progress in infeas  \\
control1   &4.56e-09   &0.00e+00   &3.52e-11   &0.00e+00   &-3.42e-10   &1.72e-09   &9.74e-02   &lack of progress in infeas  \\
control2   &5.26e-09   &0.00e+00   &6.48e-11   &0.00e+00   &2.40e-08   &3.39e-08   &2.02e-01   &lack of progress in infeas  \\
control3   &8.15e-08   &0.00e+00   &2.14e-10   &0.00e+00   &8.82e-08   &1.26e-07   &5.02e-01   &lack of progress in infeas  \\
\bottomrule
\end{tabular}
\end{center}
\end{table}

\begin{table}[H]
\caption{Results of SDPT3 using the tight setting on ill-posed group instances}
\label{Table: SDPT3Tight-ill}
\begin{center}
\setlength{\tabcolsep}{5pt}
\footnotesize
\begin{tabular}{ccccccccc} \toprule
Problem	&${\rm err_1}$&${\rm err_2}$	&${\rm err_3}$	&${\rm err_4}$	&${\rm err_5}$	&${\rm err_6}$	&time(s)	&Message	\\	\midrule
gpp100   &1.50e-09   &0.00e+00   &6.39e-12   &0.00e+00   &-2.05e-08   &2.24e-11   &1.87e-01   &lack of progress in infeas  \\
gpp124-1   &2.09e-10   &8.67e-17   &2.99e-11   &0.00e+00   &-3.50e-08   &5.49e-10   &2.79e-01   &lack of progress in infeas  \\
gpp124-2   &6.49e-10   &0.00e+00   &2.21e-11   &0.00e+00   &-1.23e-08   &1.06e-10   &2.42e-01   &lack of progress in infeas  \\
gpp124-3   &2.51e-09   &0.00e+00   &1.22e-11   &0.00e+00   &-9.01e-09   &2.88e-11   &2.32e-01   &lack of progress in infeas  \\
gpp124-4   &3.41e-09   &0.00e+00   &8.18e-12   &0.00e+00   &-1.89e-08   &1.02e-11   &2.65e-01   &lack of progress in infeas  \\
hinf1   &5.35e-08   &0.00e+00   &4.83e-09   &0.00e+00   &-1.02e-05   &1.45e-08   &2.12e-01   &lack of progress in infeas  \\
hinf3   &9.02e-06   &0.00e+00   &8.16e-12   &0.00e+00   &-1.18e-04   &2.00e-09   &9.94e-02   &lack of progress in infeas  \\
\bottomrule
\end{tabular}
\end{center}
\end{table}
\end{landscape}

\begin{landscape}
\subsection{Results of the proposed method}
\begin{table}[H]
\caption{Results of Algorithm \ref{postpro alg} using the solutions from Mosek on well-posed group instances}
\label{Table: ProMosek-Well}
\begin{center}
\setlength{\tabcolsep}{5pt}
\footnotesize
\begin{tabular}{ccccccccccccc} \toprule
	&\multicolumn{7}{c}{Algorithm \ref{postpro alg}}	&\multicolumn{2}{c}{Algorithm \ref{Practical d alg}}	&\multicolumn{2}{c}{Algorithm \ref{Practical p alg}}	\\	\midrule
Problem	&${\rm err_1}$&${\rm err_2}$	&${\rm err_3}$	&${\rm err_4}$	&${\rm err_5}$	&${\rm err_6}$	&time(s)	&Message	&time(s)	&Message	&time(s)	\\	\midrule
arch0   &4.82e-13   &5.89e-17   &0.00e+00   &0.00e+00   &4.14e-13   &4.29e-13   &1.23e+03   &Complete   &1.19e+03   &Complete   &4.30e+01  \\
arch2   &4.20e-13   &1.83e-17   &0.00e+00   &0.00e+00   &1.71e-13   &1.58e-13   &1.60e+03   &Complete   &1.57e+03   &Complete   &2.42e+01  \\
arch4   &1.80e-13   &2.01e-17   &0.00e+00   &0.00e+00   &2.90e-13   &2.98e-13   &4.66e+02   &Complete   &4.41e+02   &Complete   &2.45e+01  \\
arch8   &6.91e-13   &3.88e-16   &0.00e+00   &0.00e+00   &4.63e-14   &9.17e-14   &6.69e+02   &Complete   &5.85e+02   &Complete   &8.40e+01  \\
control1   &6.66e-15   &0.00e+00   &0.00e+00   &0.00e+00   &4.39e-14   &4.66e-14   &1.70e-01   &Complete   &1.23e-01   &Complete   &4.76e-02  \\
control2   &1.26e-14   &2.85e-17   &0.00e+00   &0.00e+00   &3.72e-13   &3.84e-13   &2.64e+00   &Complete   &2.50e+00   &Complete   &1.45e-01  \\
\bottomrule
\end{tabular}
\end{center}
\end{table}

\begin{table}[H]
\caption{Results of Algorithm \ref{postpro alg} using the solutions from Mosek on ill-posed group instances}
\label{Table: ProMosek-ill}
\begin{center}
\setlength{\tabcolsep}{5pt}
\footnotesize
\begin{tabular}{ccccccccccccc} \toprule
	&\multicolumn{7}{c}{Algorithm \ref{postpro alg}}	&\multicolumn{2}{c}{Algorithm \ref{Practical d alg}}	&\multicolumn{2}{c}{Algorithm \ref{Practical p alg}}	\\	\midrule
Problem	&${\rm err_1}$&${\rm err_2}$	&${\rm err_3}$	&${\rm err_4}$	&${\rm err_5}$	&${\rm err_6}$	&time(s)	&Message	&time(s)	&Message	&time(s)	\\	\midrule
gpp100   &2.75e-13   &0.00e+00   &0.00e+00   &0.00e+00   &1.32e-10   &2.20e-10   &1.61e+03   &Complete   &9.02e+02   &Complete   &7.10e+02  \\
gpp124-1   &2.29e-13   &0.00e+00   &0.00e+00   &0.00e+00   &1.61e-11   &1.72e-09   &3.01e+03   &Time Over   &1.85e+03   &Complete   &1.16e+03  \\
gpp124-2   &3.30e-13   &0.00e+00   &0.00e+00   &0.00e+00   &1.28e-12   &1.42e-09   &3.19e+03   &Time Over   &1.85e+03   &Complete   &1.34e+03  \\
gpp124-3   &7.27e-13   &2.16e-16   &0.00e+00   &0.00e+00   &-4.62e-13   &3.83e-11   &4.01e+03   &Time Over   &2.73e+03   &Complete   &1.27e+03  \\
gpp124-4   &1.31e-13   &4.01e-18   &0.00e+00   &0.00e+00   &9.77e-11   &2.56e-10   &3.72e+03   &Time Over   &1.97e+03   &Complete   &1.75e+03  \\
hinf1   &5.37e-16   &0.00e+00   &0.00e+00   &0.00e+00   &4.81e-08   &4.94e-08   &1.07e+00   &Complete   &7.93e-01   &Complete   &2.78e-01  \\
\bottomrule
\end{tabular}
\end{center}
\end{table}

\begin{table}[H]
\caption{Results of Algorithm \ref{postpro alg} using the solutions from SDPA on well-posed group instances}
\label{Table: ProSDPA-well}
\begin{center}
\setlength{\tabcolsep}{5pt}
\footnotesize
\begin{tabular}{ccccccccccccc} \toprule
	&\multicolumn{7}{c}{Algorithm \ref{postpro alg}}	&\multicolumn{2}{c}{Algorithm \ref{Practical d alg}}	&\multicolumn{2}{c}{Algorithm \ref{Practical p alg}}	\\	\midrule
Problem	&${\rm err_1}$&${\rm err_2}$	&${\rm err_3}$	&${\rm err_4}$	&${\rm err_5}$	&${\rm err_6}$	&time(s)	&Message	&time(s)	&Message	&time(s)	\\	\midrule
arch0   &5.24e-13   &8.24e-18   &0.00e+00   &0.00e+00   &6.73e-13   &6.90e-13   &8.34e+02   &Complete   &1.98e+02   &Complete   &6.35e+02  \\
arch2   &2.68e-13   &2.07e-18   &0.00e+00   &0.00e+00   &3.77e-13   &3.76e-13   &7.17e+02   &Complete   &4.07e+02   &Complete   &3.09e+02  \\
arch4   &2.19e-13   &1.12e-17   &0.00e+00   &0.00e+00   &5.47e-14   &5.20e-14   &3.43e+02   &Complete   &2.73e+02   &Complete   &7.01e+01  \\
arch8   &8.93e-13   &8.27e-17   &0.00e+00   &0.00e+00   &1.08e-13   &1.22e-13   &6.47e+02   &Complete   &3.70e+02   &Complete   &2.77e+02  \\
control1   &2.32e-14   &7.09e-16   &0.00e+00   &0.00e+00   &6.12e-15   &6.99e-15   &5.30e-02   &Complete   &4.34e-02   &Complete   &9.65e-03  \\
control2   &1.56e-14   &5.47e-17   &0.00e+00   &0.00e+00   &5.94e-14   &5.43e-14   &1.06e+00   &Complete   &9.65e-01   &Complete   &9.38e-02  \\
\bottomrule
\end{tabular}
\end{center}
\end{table}

\begin{table}[H]
\caption{Results of Algorithm \ref{postpro alg} using the solutions from SDPA on ill-posed group instances}
\label{Table: ProSDPA-ill}
\begin{center}
\setlength{\tabcolsep}{5pt}
\footnotesize
\begin{tabular}{ccccccccccccc} \toprule
	&\multicolumn{7}{c}{Algorithm \ref{postpro alg}}	&\multicolumn{2}{c}{Algorithm \ref{Practical d alg}}	&\multicolumn{2}{c}{Algorithm \ref{Practical p alg}}	\\	\midrule
Problem	&${\rm err_1}$&${\rm err_2}$	&${\rm err_3}$	&${\rm err_4}$	&${\rm err_5}$	&${\rm err_6}$	&time(s)	&Message	&time(s)	&Message	&time(s)	\\	\midrule
gpp100   &1.91e-13   &0.00e+00   &0.00e+00   &0.00e+00   &-2.91e-13   &3.54e-10   &2.40e+03   &Time Over   &1.92e+03   &Complete   &4.77e+02  \\
gpp124-1   &2.34e-13   &3.58e-16   &0.00e+00   &0.00e+00   &-1.48e-11   &-6.68e-11   &3.00e+03   &Time Over   &1.83e+03   &Complete   &1.16e+03  \\
gpp124-2   &1.08e-13   &8.52e-17   &0.00e+00   &0.00e+00   &3.63e-10   &2.56e-10   &3.35e+03   &Time Over   &1.90e+03   &Complete   &1.44e+03  \\
gpp124-3   &1.20e-13   &6.70e-17   &0.00e+00   &0.00e+00   &-1.85e-07   &1.14e-06   &2.14e+03   &Complete   &1.49e+03   &Complete   &6.56e+02  \\
gpp124-4   &3.41e-13   &3.26e-16   &0.00e+00   &0.00e+00   &-5.18e-12   &2.89e-11   &3.03e+03   &Time Over   &1.82e+03   &Complete   &1.21e+03  \\
hinf1   &6.07e-15   &8.27e-17   &0.00e+00   &0.00e+00   &3.82e-08   &6.90e-08   &3.45e-01   &Complete   &2.44e-01   &Complete   &1.01e-01  \\
\bottomrule
\end{tabular}
\end{center}
\end{table}

\begin{table}[H]
\caption{Results of Algorithm \ref{postpro alg} using the solutions from SDPT3 on well-posed group instances}
\label{Table: ProSDPT3-well}
\begin{center}
\setlength{\tabcolsep}{5pt}
\footnotesize
\begin{tabular}{ccccccccccccc} \toprule
	&\multicolumn{7}{c}{Algorithm \ref{postpro alg}}	&\multicolumn{2}{c}{Algorithm \ref{Practical d alg}}	&\multicolumn{2}{c}{Algorithm \ref{Practical p alg}}	\\	\midrule
Problem	&${\rm err_1}$&${\rm err_2}$	&${\rm err_3}$	&${\rm err_4}$	&${\rm err_5}$	&${\rm err_6}$	&time(s)	&Message	&time(s)	&Message	&time(s)	\\	\midrule
arch0   &3.25e-13   &1.51e-17   &0.00e+00   &0.00e+00   &1.33e-13   &1.29e-13   &3.29e+02   &Complete   &2.75e+02   &Complete   &5.42e+01  \\
arch2   &4.11e-13   &6.14e-17   &0.00e+00   &0.00e+00   &8.27e-13   &8.18e-13   &1.83e+03   &Complete   &1.68e+03   &Complete   &1.57e+02  \\
arch4   &2.90e-13   &2.09e-17   &0.00e+00   &0.00e+00   &4.78e-13   &4.93e-13   &2.56e+02   &Complete   &2.05e+02   &Complete   &5.12e+01  \\
arch8   &7.71e-13   &2.73e-16   &0.00e+00   &0.00e+00   &1.15e-13   &1.15e-13   &4.70e+02   &Complete   &2.73e+02   &Complete   &1.97e+02  \\
control1   &3.87e-15   &1.11e-15   &0.00e+00   &0.00e+00   &1.25e-12   &1.25e-12   &1.80e-01   &Complete   &1.25e-01   &Complete   &5.42e-02  \\
control2   &2.16e-14   &0.00e+00   &0.00e+00   &0.00e+00   &7.85e-14   &7.65e-14   &1.88e+00   &Complete   &1.55e+00   &Complete   &3.32e-01  \\
\bottomrule
\end{tabular}
\end{center}
\end{table}

\begin{table}[H]
\caption{Results of Algorithm \ref{postpro alg} using the solutions from SDPT3 on ill-posed group instances}
\label{Table: ProSDPT3-ill}
\begin{center}
\setlength{\tabcolsep}{5pt}
\footnotesize
\begin{tabular}{ccccccccccccc} \toprule
	&\multicolumn{7}{c}{Algorithm \ref{postpro alg}}	&\multicolumn{2}{c}{Algorithm \ref{Practical d alg}}	&\multicolumn{2}{c}{Algorithm \ref{Practical p alg}}	\\	\midrule
Problem	&${\rm err_1}$&${\rm err_2}$	&${\rm err_3}$	&${\rm err_4}$	&${\rm err_5}$	&${\rm err_6}$	&time(s)	&Message	&time(s)	&Message	&time(s)	\\	\midrule
gpp100   &2.68e-13   &1.82e-15   &0.00e+00   &0.00e+00   &9.24e-11   &-9.73e-11   &2.37e+03   &Time Over   &1.92e+03   &Complete   &4.58e+02  \\
gpp124-1   &8.66e-14   &0.00e+00   &0.00e+00   &0.00e+00   &1.23e-08   &1.42e-08   &3.72e+03   &Time Over   &1.82e+03   &Time Over   &1.89e+03  \\
gpp124-2   &4.06e-13   &0.00e+00   &0.00e+00   &0.00e+00   &8.98e-10   &1.13e-09   &3.81e+03   &Time Over   &1.83e+03   &Time Over   &1.98e+03  \\
gpp124-3   &2.26e-13   &0.00e+00   &0.00e+00   &0.00e+00   &-4.64e-12   &1.50e-09   &4.03e+03   &Time Over   &2.22e+03   &Time Over   &1.82e+03  \\
gpp124-4   &2.97e-13   &0.00e+00   &0.00e+00   &0.00e+00   &5.27e-08   &5.72e-08   &4.42e+03   &Time Over   &2.28e+03   &Time Over   &2.14e+03  \\
hinf1   &1.13e-13   &0.00e+00   &0.00e+00   &0.00e+00   &2.71e-08   &7.39e-08   &1.10e+00   &Complete   &5.25e-01   &Complete   &5.77e-01  \\
\bottomrule
\end{tabular}
\end{center}
\end{table}
\end{landscape}

\end{document}